\documentclass[12pt, a4paper, parskip=half, abstracton]{scrartcl}

\usepackage{array}
\usepackage{marginnote}
\usepackage{xcolor}
\usepackage{amscd,amssymb,amsfonts,amsmath,latexsym,amsthm}
\usepackage{hyperref}
\usepackage[all,cmtip]{xy}
\usepackage{mathrsfs}
\usepackage{graphicx}
\usepackage{bm} 
\usepackage[nameinlink, capitalise]{cleveref}

\usepackage{etoolbox}

\def\hB{\hspace*{\fill}$\qed$}

\usepackage[nottoc]{tocbibind} 

\usepackage{defs_pp1}
\usepackage{slashed}
\usepackage[utf8]{inputenc}
\usepackage{microtype}
\usepackage[english]{babel}
\usepackage{mathtools}
\usepackage{bm}
\usepackage{esvect}

\title{Transgressions and Chern characters in coarse homotopy theory}
\author{
Ulrich Bunke\thanks{Fakult{\"a}t f{\"u}r Mathematik,
Universit{\"a}t Regensburg,
93040 Regensburg,
ulrich.bunke@mathematik.uni-regensburg.de} 
}

\numberwithin{equation}{section}
\newtheorem{theorem}{Theorem}[section] 
\newtheorem{prop}[theorem]{Proposition}
\newtheorem{lem}[theorem]{Lemma}

\newtheorem{ddd}[theorem]{Definition}
\newtheorem{kor}[theorem]{Corollary}
\newtheorem{ass}[theorem]{Assumption}

\newtheorem{construction}[theorem]{Construction}
\newtheorem{prob}[theorem]{Problem}

\theoremstyle{remark}
\theoremstyle{definition}

\newtheorem{ex}[theorem]{Example}
\newtheorem{rem}[theorem]{Remark}

\newcommand{\nloc}{\mathrm{nloc}}
\newcommand{\Nat}{\mathrm{Nat}}

\newcommand{\BM}{\mathrm{BM}}
\newcommand{\LoF}{\mathrm{LF}}
\newcommand{\asmbl}{\mathrm{asmbl}}
\newcommand{\PH}{\mathrm{PH}}
\newcommand{\propp}{\mathrm{prop}}

\newcommand{\Mix}{\mathbf{Mix}}
\newcommand{\rmMix}{\mathrm{Mix}}
\newcommand{\HP}{\mathrm{HP}}
\newcommand{\CH}{\mathbf{CH}}
\newcommand{\rmch}{\mathrm{ch}}
\newcommand{\rmCH}{\mathrm{CH}}

\newcommand{\BT}{\mathbf{BornTop}}

\newcommand{\K}{\mathrm{K}}

\newcommand{\homol}{\mathrm{homol}}
\newcommand{\Ring}{\mathbf{Ring}}

\newcommand{\sepa}{\mathrm{sep}}
\newcommand{\fin}{\mathrm{fin}}

\newcommand{\EE}{\mathbf{E}}
\newcommand{\coarse}{\mathrm{coarse}}

\newcommand{\rmH}{\mathrm{H}}

\newcommand{\ctr}{\mathrm{ctr}}

\newcommand{\Acyc}{\mathbf{Acyc}}

\newcommand{\fg}{\mathrm{fg}}
\newcommand{\topp}{\mathrm{top}}
\newcommand{\ee}{\mathrm{e}}

\newcommand{\LCH}{\mathbf{LCH}}

\newcommand{\alg}{\mathrm{alg}}

\newcommand{\nCalg}{C^{*}\mathbf{Alg}^{\mathrm{nu}}}

\newcommand{\can}{\mathrm{can}}
\newcommand{\F}{\mathbb{F}}

\renewcommand{\CW}{\mathbf{CW}}

\newcommand{\All}{\mathbf{All}}

\newcommand{\UBC}{\mathbf{UBC}}

\newcommand{\Yo}{\mathrm{Yo}}

\newcommand{\Res}{\mathrm{Res}}

\newcommand{\Orb}{\mathbf{Orb}}
\newcommand{\TB}{\mathbf{TopBorn}}

\newcommand{\Hilb}{\mathbf{Hilb}}
\newcommand{\cR}{\mathcal{R}}

\newcommand{\bQ}{\mathbf{Q}}

\newcommand{\BC}{\mathbf{BC}}

\newcommand{\cN}{\mathcal{N}}

\newcommand{\Fin}{\mathbf{Fin}}

\newcommand{\Ob}{\mathrm{Ob}}

\newcommand{\Cofib}{\mathrm{Cofib}}

\newcommand{\Fib}{{\mathrm{Fib}}}

\newcommand{\Sing}{\mathrm{Sing}}

\newcommand{\incl}{\mathrm{incl}}

\newcommand{\bM}{\mathbf{M}}

\newcommand{\mot}{\mathrm{mot}}

\newcommand{\CAlg}{{\mathbf{CAlg}}}

\newcommand{\cI}{{\mathcal{I}}}

\newcommand{\cL}{{\mathcal{L}}}
\newcommand{\cW}{{\mathcal{W}}}
\newcommand{\PSh}{{\mathbf{PSh}}}

\newcommand{\Add}{{\mathtt{Add}}}

\newcommand{\bA}{{\mathbf{A}}}

\newcommand{\Alg}{{\mathbf{Alg}}}
\newcommand{\nAlg}{\mathbf{Alg}^{\mathrm{nu}}}

\newcommand{\cO}{{\mathcal{O}}}

\newcommand{\cY}{{\mathcal{Y}}}
\newcommand{\Var}{\mathrm{Var}}

\newcommand{\cD}{{\mathcal{D}}}

 \newcommand{\Cat}{{\mathbf{Cat}}}

\DeclareMathOperator{\proj}{proj}

\newcommand{\lf}{\mathrm{lf}}

\newcommand{\Spc}{\mathbf{Spc}}

\newcommand{\PCH}{\mathrm{PCH}}

\newcommand{\Ccat}{{C^{\ast}\mathbf{Cat}}}
\newcommand{\Calg}{{\mathbf{C}^{\ast}\mathbf{Alg}}}

\renewcommand{\Add}{\mathbf{Add}}
\renewcommand{\Pr}{\mathbf{Pr}}
\newcommand{\bP}{\mathbf{P}}

\newcommand{\st}{\mathrm{st}}

\newcommand{\comm}{\mathrm{comm}}

\newcommand{\op}{\mathrm{op}}

\newcommand{\kk}{\mathrm{kk}}
\newcommand{\KK}{\mathbf{KK}}

\newcommand{\std}{\mathrm{std}}
\newcommand{\cone}{\mathrm{cone}}

\newcommand{\nCcat}{C^{*}\mathbf{Cat}^{\mathrm{nu}}}

\renewcommand{\tr}{\mathrm{tr}}
\renewcommand{\id}{\mathrm{id}}

\newcommand{\CM}{\mathbf{CM} }

\newcommand{\disc}{\mathrm{disc}}

\newcommand{\an}{\mathrm{an}}

\begin{document}
 \setcounter{tocdepth}{2}

\maketitle

\begin{abstract}   
This paper investigates a variety of coarse homology theories and natural transformations between them. We in particular 
 study the commutativity of a square relating analytical and topological transgressions with algebraic 
and homotopy theoretic Chern characters. Here a transgression is a natural transformation from a coarse homology theory to a functor which factorizes over the Higson corona functor, and a Chern character is a transformation from a $K$-theory like coarse or Borel-Moore type homology theory to an ordinary version.
  \end{abstract}

 \tableofcontents

 \section{Introduction}

\subsection{Overview}  \label{kohperthethgertgetr}

The objects of equivariant  coarse homotopy theory for a group $G$ are  $G$-bornological coarse spaces which will be   analysed using equivariant coarse homology theories. The Higson corona functor $\partial_{h}$ sends $G$-bornological coarse spaces to compact Hausdorff $G$-spaces whose homotopy theory will by studied using equivariant Borel-Moore  homology theories. 
Let $E^{G}$ be an equivariant coarse homology theory and $F^{G}$ be an equivariant Borel-Moore homology theory with the same target category.
A    natural transformation
\begin{equation}\label{hwrtgretgrertrgrtger}T^{G}:E^{G}\to F^{G}\circ \partial_{h}
\end{equation}
of functors defined on $G$-bornological coarse spaces
will be called a transgression.
Currently we know two examples of transgressions. The first, the analytic transgression denoted by $T^{G,\an}$, relates the equivariant topological coarse $K$-homology
with the equivariant analytic $K$-homology. Its construction depends on the internal analytic details of the construction
of the $K$-theory functors.
The  second, the topological transgression denoted by $T^{G,\topp} $, connects the coarsification of
a Borel-Moore homology theory with the Borel-Moore homology itself. It is induced by a homotopy-theoretic construction using  the strong  excision property of the Borel-Moore homology theory.

These two transgressions appear in the horizontal parts of the following diagram

 {\small \begin{equation}\label{ferwferfwrefw}\hspace{-2.3cm}
\xymatrix{ K\cX_{G_{can,max}}^{G,\ctr}(X)\ar[rr]^{c^{G}}_{comparison} \ar[dd]_{alg. Chern}^{\rmch^{G,\alg}}&&K\cX_{G_{can,max}}^{G}(X) \ar[rr]^{T^{G,\an}}_{anal.\ transgression}&&\Sigma K^{G,\an}(\partial_{h}X)\ar[dddd]^{\beta}_{Borelific.}  \\   &\mbox{\begin{minipage}[c]{2cm} eq. alg.\\ coarse \\$K$-homol.\end{minipage}}\ar@{..}[ul]&\mbox{\begin{minipage}[c]{2cm} eq. top.\\ coarse \\$K$-homol.\end{minipage}}\ar@{..}[u]&\mbox{\begin{minipage}[c]{1.3cm} analytic \\$K$-homol.\end{minipage}}\ar@{..}[ur]& \\
  \PCH\cX^{G}_{G_{can,max}}(X)\ar[d]_{trace}^{\tau^{G}}& \mbox{\begin{minipage}[c]{2cm}eq. coarse \\ per.  cycl. \\homol.\end{minipage}}\ar@{..}[l]& & &   \\ \ar[d]^{\beta}_{Borelific.} \PH\cX^{G}_{G_{can,max}}(X,\C) & \mbox{\begin{minipage}[c]{2cm}eq. coarse\\ per. homol.\end{minipage}} \ar@{..}[l]& &  &  \\ 
\PH\cX^{hG}_{G_{can,max}}(X ,\C)  \ar[d]^{\pr_{X}} & \mbox{\begin{minipage}[c]{2cm}Borel-eq. \\coarse\\ per. homol.\end{minipage}} \ar@{..}[l] \ar@{..}[dl]&&  \mbox{\begin{minipage}[c]{2cm}Borel-eq. \\analytic\\ K-homol.\end{minipage}} \ar@{..}[r]& \Sigma K^{\an,hG}(\partial_{h} X)\ar[ddd]_{Borel-eq. h.\ Chern}^{\ch^{hG} }    \\   
\PH\cX^{hG}(X,\C)\ar[dd]^{P(\chi_{\BM})^{hG}}_{coarse\ char.}&  &&  &   \\ 
&\mbox{\begin{minipage}[c]{2cm}Borel-eq\\ coarsif. of \\  per. BM-homol.\end{minipage}}\ar@{..}[dl]&&\mbox{\begin{minipage}[c]{2cm} Borel-eq\\ per. \\BM-homol.\end{minipage}}\ar@{..}[dr]&\\ 
  (\prod_{k\in \Z} \Sigma^{2k}H\C_{\BM} \bP)^{hG} (X)\ar[rrrr]_{top.\  transgression}^{(T^{\topp})^{hG}}\ &&&&\Sigma  (\prod_{k\in \Z} \Sigma^{2k}H\C_{\BM})^{hG}(\partial_{h}X)  }\ .
\end{equation}
}of functors and natural transformations evaluated at a $G$-bornological coarse space $X$.   The functors on the left part and the middle are equivariant coarse homology theories.  The functors on the right are Borel-Moore homology  theories composed with the Higson corona functor.  
Its vertical compositions involve Chern-character maps from  versions of coarse or analytic $K$-homology  to the coarsification of a Borel-Moore homology or the Borel-Moore homology itself, respectively.
Thereby   the coarse algebraic Chern character is defined on
an algebraic coarse $K$-homology only, which is connected with the topological $K$-homology theory
by a comparison map $c^{G}$. 
In our examples the data determining a $G$-equivariant coarse or Borel-Moore homology actually determines a family
of corresponding equivariant homology theories for all subgroups of $G$. The non-equivariant homology theory
associated  to the trivial group in turn determines a Borel-equivariant version indicated by a superscript $hG$. In all our examples we have  Borelification maps, i.e., 
  natural transformations denoted by $\beta$  from the original $G$-equivariant homology theory to the associated
  Borel-equivariant version.

The existence of this diagram
 allows us to   discuss the compatibility of  the  Chern characters   with the
 transgressions.    There is no obvious reason that this diagram  commutes.
 It is one goal of 
  the present paper  to
provide commutativity   under certain additional  assumptions and in the sense of
 the extended diagram \eqref{gwerwefwerfrefw} below.
 The diagram \eqref{ferwferfwrefw} is our interpretation of the diagram  \cite[(1.2)]{Engel:2025aa}
 in the framework of equivariant coarse homotopy theory as developed in \cite{buen}, \cite{equicoarse}.

In \cref{gjkopwergerfwerf}   we start with providing a rough description of  the functors and transformations   in \eqref{ferwferfwrefw}. The complete details or references will be given in the subsections of \cref{orkhprtegrgerg}.
It is one of the
goals of this paper is to produce a complete  reference for  constructions of the functors and transformations
 in their natural generality. 

 Having explained the functors and transformations, in \cref{koipgwergwerfwerfrw} we 
  then extend the diagram  to the diagram \eqref{gwerwefwerfrefw} and discuss the steps towards   commutativity.
 Proving commutativity   means showing the existence of fillers.  
 We do not see technical relations between the algebraic and homotopy theoretic Chern characters
 $\rmch^{G,\alg}$ and $\rmch^{hG}$, and also no relations between the analytic and topological transgressions
 $T^{G,\an}$ and $(T^{\topp})^{hG}$ which could provide the desired fillers in general. 
 
We approach the problem as follows. We first consider the case of the trivial group  $G$ and evaluate the diagram at $\cO^{\infty}(*)$, the cone over a point, see \cref{koopehrrtgetrgeg}. In this case we can calculate the compositions and normalize $\rmch^{h}$ such that the  resulting diagram commutes. We then take advantage of the fact that  we have a diagram of natural transformations between spectrum-valued functors which  have homological properties. Using this we get a natural extension of the commutativity to cones over finite CW-complexes. By naturality this implies commutativity for the Borel-equivariant version on finite $G$-CW-complexes for finite groups.
Using the comparison of equivariant coarse homology theories with their Borel equivariant versions we conclude the commutativity of the diagram in the form \eqref{gwerwefwerfreffffefefedededew} for finite groups  and cones over finite $G$-CW-complexes.
In order to extend the commutativity further
we then  introduce
an additional geometric construction leading to the motivic  
transgression and extend the diagram further  to  \eqref{gwerwefwerfrefw1}. 

The motivic transgression is an attempt to associate to every equivariant coarse homology theory $E^{G}$ an associated Borel-Moore homology theory $F^{G}$   and a transgression transformation \eqref{hwrtgretgrertrgrtger}. As a first approximation we propose to take $F^{G}:=E^{G}\cO^{\infty}$ and the motivic transgression $T^{\mot}$ from \eqref{kophrgertgertrgertgergert}. But note that $F^{G}$ is not really a Borel-Moore homology, and $T^{\mot}$ is not a natural transformation of functors defined on  $G\BC$. One could ask for a better construction.

   Note that our argument for commutativity  is completely different from the arguments  in
   \cite[(1.2)]{Engel:2025aa} which are based on explicit calculations of the two
compositions in their diagram for special  cycles. In contrast, our proof tries to max out the homotopical properties of the functors and transformations.

In \cref{ergerwfwerfwrf} we discuss the pairing of the right corners of the diagram \eqref{ferwferfwrefw} with
cohomology classes. We further  derive the consequences of commutativity for the values of the pairings and state an analoge \eqref{hrtepogkeportgertgetrge} of \cite[(1.1)]{Engel:2025aa}.

   {\em Acknowledgement:  The author was supported by the SFB 1085 (Higher Invariants) funded by the Deutsche Forschungsgemeinschaft (DFG).  He  thanks M. Ludewig and A. Engel for motivating discussions on early stages of \cite{Engel:2025aa}. He further thanks Th. Nikolaus for  clarifying remarks on   Borel-Moore (co)homology theories. 
 Finally he thanks  B. Dünzinger  for his patient interest and clarifying remarks on various pieces of this paper. 
 }

%
%
%
%
  
  \subsection{The functors and transformations}\label{gjkopwergerfwerf}

 We start with the description of the functors and transformations in \eqref{ferwferfwrefw}.
  
 Let $G$ be a group. 
 By  $G\BC$ we denote the symmetric monoidal  category of $G$-bornological coarse spaces \cite{equicoarse} (\cref{kopehrtgegrtgrtge}). 
 The symbol $X$ in \eqref{ferwferfwrefw} denotes a $G$-bornological coarse space.

The functors in the left  and middle column in  \eqref{ferwferfwrefw} are instances of equivariant coarse homology theories in the sense introduced in \cite{buen},  \cite{equicoarse}.
These are functors $$E:G\BC\to \cC$$ to a cocomplete stable $\infty$-categories satisfying certain axioms 
named in  \cref{okprhertgrtge9}. If $Y$ is a $G$-bornological coarse space, then we can consider the twist $$E_{Y}:=E(-\otimes Y):G\BC\to \cC$$ of $E$. Twisting preserves coarse homology theories. In
 \eqref{ferwferfwrefw} we twist by the object $G_{can,max}$ given  the group $G$ with the canonical coarse structure and the maximal bornology. 

To any $C^{*}$-category with strict $G$-action $\bC$   which is effectively additive and admits countable $AV$-sums and to  any object $A$ in the stable $\infty$-category $\EE$ representing $E$-theory of $C^{*}$-algebras we can associate the equivariant  topological coarse $K$-homology theory (\cref{gu90erwfwef})
 $$K\cX^{G}_{\bC,A}:=\bE(\C,  \bV^{G}_{\bC}\otimes A):G\BC\to \Mod(KU)\ .$$
 This is a minor generalization of the construction of $K\cX^{G}_{\bC}$
from \cite{coarsek} by adding the variable $A$ in $\EE$. 
 As the formula indicates it is constructed  by applying
the topological $K$-theory functor   for $C^{*}$-categories with coefficients in $A$ to the functor which sends a $G$-bornological coarse space  $X$ to the Roe category $\bV_{\bC}^{G}(X)$   from \eqref{gweoihjgoiwerfwerfrwef} of locally finite equivariant $X$-controlled objects in $\bC$ \cite{coarsek}. The equivariant  topological coarse $K$-homology theory
 is a $\Mod(KU)$-valued and equivariant generalization   of the functor which associates to a proper  metric space the $K$-theory groups  of the associated Roe algebra first introduced in
\cite{MR1147350}.
The functor $K\cX^{G}_{G_{can,max}}$  in \eqref{ferwferfwrefw} is obtained from $K\cX^{G}_{\bC,A}$ by   twisting  with $G_{can,max}$ and specializing  to   $\bC=\Hilb_{c}(\C)$ and $A=\ee(\C)$.

%
%
%

We let
 $G\LCH^{+}$ denote the category of locally compact Hausdorff spaces with $G$-action and partially defined $G$-equivariant proper maps (\cref{kopwhwthrh}). It contains the subcategory of compact Hausdorff spaces $G\CH$ with $G$-action and everywhere defined maps. Associating to each bornological coarse space its Higson corona gives rise to 
 the Higson corona functor (\cref{kiogwergewrfwrefw})
$$\partial_{h} :G\BC \to G\CH\ .$$
The compact Hausdorff $G$-space $\partial_{h}X$ in   \eqref{ferwferfwrefw}  is thus the Higson corona of the $G$-bornological coarse space  $X$.  
 
 Associated to the data $\bC$ and $A$ as above we introduce
 the equivariant analytic
 $K$-homology theory (\cref{kophertgrgrgerg})
 $$K^{G,\an}
_{\bC,A}:=\EE^{G}(C_{0}(-),\bC^{(G)}_{\std}\otimes A):G\LCH^{+}\to \Mod(KU)$$
with coefficients in $(\bC,A)$. As the formula shows,
it is constructed  as the equivariant $E$-theory cohomology (with coefficients derived from $\bC$ and $A$) of the  $C^{*}$-algebra of continuous functions on the space that vanish at infinity. 
Here we use the  $\infty$-categorical version of  equivariant $E$-theory which was  constructed in \cite{budu}. The latter is   a homotopical version of the group-valued equivariant $E$-theory functor of \cite{Guentner_2000}. 
The construction of the equivariant analytic $K$-homology
$K^{G,\an}_{\bC,A}$ with coefficients in $(\bC,A)$ is a minor generalization of
the constructions  from \cite{KKG}, \cite{Bunke:2024aa}, \cite{bel-paschke}.
In contrast to  the $KK$-theory version from  \cite{KKG}, a nice consequence of the  present $E$-theory based construction is 
that $K^{G,\an}_{\bC,A}$ is an equivariant Borel-Moore homology theory (\cref{ogpwerferfewrfwef}).
Here we use that in contrast to $KK$ the  $E$-theory functor is unconditionally exact and preserves all filtered colimits.
%
%
  The functor $K^{G,\an}$ in  \eqref{ferwferfwrefw}  is obtained by specializing $K^{G,\an}
_{\bC,A}$ to   $\bC=\Hilb_{c}(\C)$ and $A=\ee(\C)$. 
%

%

 Using an analytic construction on the level of Roe categories inspired by \cite{quro} we define the analytic transgression transformation (\cref{rguweruigowerferfrfwr})
 $$T^{G,\an}:K\cX^{G}_{\bC,A,G_{can,max}}\to K^{G,\an}
_{\bC,A}\circ \partial_{h}:G\BC\to \Mod(KU)\ .$$
 The idea for its construction   was already envisaged in the construction of the coarse corona pairing in \cite{Bunke:2024aa}. Its technical details  are closely related with the construction of the Paschke morphism in  \cite{bel-paschke}.

%
%
%
%

At the moment we do not know a Chern character transformation from coarse topological $K$-homology
to a coarse periodic cyclic homology. In fact, the  existence of such a transformation would have strong 
implications explained in \cref{kophwegergwerwf}. But we will construct a 
coarse algebraic Chern character transformation whose domain is  the equivariant
coarse algebraic $K$-theory functor with coefficients in $\bC$  (\cref{kopgwegwerfrfwref})
$$K\cX^{G,\ctr}_{\bC}:G\BC\to \Sp\ .$$
This functor is  obtained by composing  the functor 
\begin{equation}\label{regerferwfrwefwfer}K^{\Cat_{\Z}}H(\cL^{1}\otimes^{\alg}_{\C} -):\Add_{\C}\to \Sp
\end{equation} with the uncompleted
Roe category functor $\bV^{G,\ctr}_{\bC}$ from   \eqref{gweoihjgoiwerfwerfrwef}, where $K^{\Cat_{\Z}}H$ is a version of Weibel's homotopy $K$-theory \cite{zbMATH04095731}  for  $\Z$-linear  categories
and $\cL^{1}$ is the algebra of trace class operators on a separable Hilbert space.
The functor  $K\cX_{G_{can,max}}^{G,\ctr}$  in \eqref{ferwferfwrefw} is the specialization of $K\cX^{G,\ctr}_{\bC}$ to  $\bC=\Hilb_{c}(\C)$, twisted by $G_{can,max}$. If $X$ is presented by a proper metric space $X$, then the tensor product $\cL^{1}\otimes_{\C} \bV^{G,\ctr}(X)$ is our (equivariant) version of the Roe algebra, often denoted by $\cB_{X}$, of controlled and locally trace class operators considered, e.g., in  
\cite{MR1147350}, \cite{Yu_1995}, \cite{Engel:2025aa}.

The equivariant
coarse algebraic $K$-theory functor
 comes with a 
  natural comparison transformation (\cref{kopgwegwerfrfwref1})
\begin{equation}\label{vweiohovwevsdf}c^{G}:K\cX^{G,\ctr}_{\bC}\to K\cX^{G}_{\bC}:G\BC\to \Sp
\end{equation}
to the topological equivariant coarse $K$-homology.
Here we  omit the index $A=\beins_{\EE}$ at
the target $K\cX^{G}_{\bC}$. 
The exactness and continuity properties of the homotopy $K$-theory functor $KH$ are important 
for verifying the coarse homology theory axioms for $K\cX^{G,\ctr}_{\bC}$.
On the other hand, the tensor product with the trace class operators  in \eqref{regerferwfrwefwfer} is chosen as a compromise. On the one hand hand we want  an  interesting trace map to periodic cyclic homology.  On the other hand
we want that the equivariant
coarse algebraic $K$-theory is an approximation of the topological equivariant coarse 
$K$-homology in the sense that the comparison map \eqref{vweiohovwevsdf} is an equivalence on sufficiently finite $G$-bornological coarse spaces. For our choices this follows from 
 results of  \cite{Corti_as_2008}, see \cref{gkopwreferfwerfrwef} for a precise statement.

  The target of the coarse algebraic Chern character  is the  coarse periodic cyclic homology theory (\cref{kopgewerfrefewf})
$$\PCH\cX^{G}_{\bC}:G\BC\to \Mod(H\C)\ .$$ 
 In order to construct this functor we start with  composing   the algebraic periodic cyclic homology functor $\PCH:\Add_{\C}\to \Mod(H\C)$ 
with the uncompleted
Roe category functor $\bV^{G,\ctr}_{\bC}$ from   \eqref{gweoihjgoiwerfwerfrwef}.
The  algebraic periodic cyclic homology functor is exact, but does not preserve filtered colimits. For this reason
we must restore  $u$-continuity  in a second step. Note that this second step was not needed in the cases of equivariant coarse cyclic or Hochschild homology constructed in a similar manner in \cite{Caputi_2020}.
The functor $\PCH\cX_{G_{can,max}}^{G}$ in \eqref{ferwferfwrefw}  is the specialization of $\PCH\cX^{G}_{\bC}$ to  
$\bC=\Hilb_{c}(\C)$, twisted by $G_{can,max}$.

The coarse algebraic Chern character transformation (\cref{iopgwregfrwefwerfrefw})
$$\rmch^{G,\alg}:K\cX^{G,\ctr}_{\bC}\to \PCH\cX^{G}_{\bC}:G\BC\to \Sp$$
is  induced by the algebraic Chern character \begin{equation}\label{ojvbiopgwerfwefrv}\rmch^{GJ}:K^{\Cat_{\Z}}H\to \PCH:\Add_{\C}\to \Mod(H\C)
\end{equation} 
and a trace map employing the trace of $\cL^{1}$. The superscript $GJ$ indicates that  the origin of $\rmch^{GJ}$
  is the Goodwillie-Jones trace
 from the algebraic $K$-theory of rings
to the negative cyclic homology or rings. The Goodwillie-Jones trace induces $\rmch^{GJ}$
by forcing homotopy invariance 
 and the  extension of these  functors  and transformations   from rings to $\C$-linear additive categories. 
The   coarse algebraic Chern character transformation is a variant of the construction 
of a Chern character from coarse algebraic  $K$-theory with coefficients in additive categories to coarse cyclic or Hochschild homology given in \cite{Caputi_2020}.

For a commutative ring $k$ the
equivariant ordinary coarse homology $$\rmH\cX^{G}(-,A):G\BC\to \Mod(Hk)$$   with  coefficient in a $k$-module $A$ (see \eqref{nfgbgfbddgfber})  has been  introduced in \cite[Def. 7.2]{equicoarse}.
It is the obvious equivariant and dual version of the coarse cohomology invented by Roe \cite{MR1147350}.
It is represented by an explicit chain complex  valued functor. Versions of this chain complex in special cases were also considered in \cite{Yu_1995}, \cite{wulff_axioms}.

By periodizing we mean the application of the operation $\prod_{k\in \Z}\Sigma^{2k}$.
Applying this to $\rmH\cX^{G}(-,A)$ and restoring $u$-continuity we obtain
the   periodic equivariant coarse homology
functor  with  coefficients  in $A$ (\cref{kopgherthertgertgertget}) $$\PH \cX^{G}(-,A):G\BC\to \Mod(Hk)\ .$$
The functor $ \PH\cX_{G_{can,max}}^{G}(-,\C)$ in \eqref{ferwferfwrefw}
   is the specialization of $\PH \cX^{G}(-,A)$ to  $k=A=\C$, twisted by $G_{can,max}$.

We now assume in addition that the $C^{*}$-category $\bC$ has a $G$-invariant trace. It induces   
 a natural trace transformation \eqref{wregwerfvdfs}
\begin{equation}\label{sdbfdvwervdfvsfdvd}\tau:\PCH\cX^{G}_{\bC}\to  \PH \cX^{G}(-,\C):G\BC\to \Mod(H\C)\ .
\end{equation}
The map $$\beta: \PH\cX_{G_{can,max}}^{G}(-,\C)\to  \PH\cX_{G_{can,max}} \cX^{hG}(-,\C)$$
 is the Borelification map 
 from 
\eqref{bsldkjvopsdfvsfdvsfdv} and the map $$\pr_{X}:\PH\cX^{hG}_{G_{can,max}} (-,\C)\to \PH \cX^{hG}(-,\C)$$ is  induced by the  projection $(-)\otimes G_{can,max} \to (-)$.  

If $E:G\LCH^{+}\to \cC$ is a functor to  a   cocomplete stable $\infty$-category $\cC$, then  
 $$E \bP:G\BC\to \cC$$  denotes the result of its coarsification  introduced in \cref{lhertghetrgrtgertgt}.
 The construction of coarsification has been introduced in \cite[Sec. 5.5 ]{roe_lectures_coarse_geometry}, and in the background we use its homotopical version \cite{ass}.
If $E$ is a weak equivariant  Borel-Moore homology theory, then its coarsification $E\bP$
 is by \cref{lpkherferferfergertgertgetrg} an equivariant coarse homology theory.

In the non-equivariant case, for any ring spectrum $R$ we introduce a functor \cref{okgpgkerpwofwerferwfwrefwf}
 which send a $R$-module spectrum $M$ to  a Borel-Moore homology theory
 $$M_{\BM}:\LCH^{+}\to \Mod(R)$$ with $M_{\BM}(*)\simeq M$. 
 The formation of Borel-equivariant  homology theories allows to construct equivariant Borel-Moore homology theories from non-equivariant ones. In particular
 we can form the  Borel-equivariant Borel-Moore homology $$(M_{\BM})^{hG}:G\LCH^{+}\to \Mod(R)\ ,$$ 
 see \cref{okheprthertgertgetg}.
 The functor $(\prod_{k\in \Z} \Sigma^{2k}H\C_{\BM})^{hG}\bP$ in the lower left corner of \eqref{ferwferfwrefw}
is an instance of the composition of the last three constructions for $R=H\C$.

Staying in the non-equivariant case,
for any $k$-module $A$ we have the 
 Borel-Moore character  \eqref{iogopwrefwerfewrferwf}  $$\chi_{\BM}:\rmH\cX(-,A)\to HA_{\BM}\bP:\BC\to \Mod(Hk)$$
 relating the ordinary coarse homology with coefficients in $A$ and the coarsification of the Borel-More homology
 $HA_{\BM}$. Specializing to $k=A=\C$,
its periodization (see \eqref{gwergwergwrfregw1}) and Borelification is the transformation 
\eqref{fqwedwqedwfrefrdq}
  $$P(\chi_{\BM})^{hG}:\PH\cX^{hG}(-,\C)\to (\prod_{k\in \Z} \Sigma^{2k}H\C_{\BM}\bP)^{hG} :G\BC\to \Mod(H\C)$$
  appearing in \eqref{ferwferfwrefw}.

 The trace transformation \eqref{sdbfdvwervdfvsfdvd} and the composition $$P(\chi_{\BM})\circ \tau:\PCH\cX_{\bC}\to  (\prod_{k\in \Z} \Sigma^{2k}H\C_{\BM})\bP:\BC\to \Mod(H\C)$$
 from coarse periodic cyclic homology to a version of ordinary coarse homology theory 
 are the general homotopy theoretic analogues of characters considered, e.g., in \cite[Sec. 2]{Yu_1995}, \cite[Sec. 4.3]{caputi-diss}, \cite[(4.1.3)]{Ludewig:2025aa}, \cite[(2.2)]{Engel:2025aa}.

 Using the strong excision axiom, for
 any weak equivariant Borel-Moore theory $E^{G}$, in \cref{ojohpertgtrgretgtgg} we describe the topological transgression, a natural transformation
$$T^{G,\topp}:E^{G} \bP\to  \Sigma E^{G}\circ \partial_{h}:G\BC\to  \cC\ .$$
With different technical details the fact that  strong excision allows to define such a transgression 
 has first been observed in \cite{wulff_axioms}. 
 In the diagram \eqref{ferwferfwrefw} we use the Borel-equivariant version $(T^{\topp})^{hG}$ from  \eqref{iogjoiergjowerfwerfw}
 of the topological transgression.

In the non-equivariant case and for $\bC=\Hilb_{c}(\C)$  in 
\cref{hertgertgergtrgrtg} we show (the at least to the author surprising fact) that the analytic $K$-homology with coefficients in $A$ has 
a completely homotopy theoretic description:
\begin{prop} \label{kopertgetrgetrr}We have an equivalence of spectrum-valued functors
\begin{equation}\label{gerfwereg354eewrgw}K^{\an}_{A}\simeq  K(A)_{\BM}:\LCH^{+}\to \Sp\ .
\end{equation}  \end{prop}
   The proof of this result essentially   uses the recent classification of Borel-Moore cohomology theories  by D. Clausen
 (see \cite[Thm. 3.6.13]{NKP}, \cref{gojertpgwerfwerfwerf}).
 The \cref{kopertgetrgetrr} allows to construct
the homotopical Chern character $\rmch^{h}$
 in  \eqref{ferwferfwrefw} as   the  natural transformation (\cref{9gwerfreferwfwever})
 \begin{equation}\label{bdfspokpvdfvsdfvsdvsfdv}\rmch^{h}:K^{\an}\stackrel{\eqref{gerfwereg354eewrgw}}{ \simeq}   K(\C)_{\BM}\to  \prod_{k\in \Z} \Sigma^{2k}H\C_{\BM}:\LCH^{+}\to \Sp
\end{equation} obtained by applying $(-)_{\BM}$ to
 the map of spectra
\begin{equation}\label{vdsfpovkspdfvfdvsf}K(\C)\simeq KU\to KU\wedge H\C\stackrel{!}{\simeq} \prod_{k\in \Z} \Sigma^{2k}H\C\ .
\end{equation}  It depends on the choice of the identification of $H\C$-modules marked by $!$.

\begin{rem}
The homotopical Chern character  \eqref{bdfspokpvdfvsdfvsdvsfdv} as a natural transformation  from (non-equivariant) analytic $K$-homology to periodic complex Borel-Moore homology theory defined on all of $\LCH^{+}$ is new, and its unconditional existence was surprinsing to the author.  
Classical constructions, e.g. in  \cite{Connes_1985}, \cite{Connes_1990}, \cite{Connes_1995},  associated a Chern character with values in variants of cyclic homology   to special cycles
in analytic $K$-homology  of $X$ which can be represented by suitably summable Fredholm modules.
The existence of good cycles usually requires some sort of smoothness of $X$.
  For smooth manifolds the cyclic homology   can further be mapped to de Rham
homology and finally to periodic complex Borel-Moore homology.  Our homotopical Chern character therefore allows to  formulate index theorems
stating that the homotopical Chern character of some class is equal to an analytically defined Chern character of the same class in cases where the latter is defined. In the present paper will not pursue this direction.
But this information  would be important if one wanted to deduce the results of \cite{Engel:2025aa} from the present paper in a technical sense.
\hB \end{rem}

If $E:BG\to \Fun(\BC,\cC)$ is a non-equivariant coarse homology theory
with $G$-action in the sense of \cref{khopertgertgetrg}, then  according to \cref{oiog0wergerfwerferwfr} we can form   
 the Borel-equivariant coarse homology theory
$$E^{hG}:=(\lim_{BG}E)^{u}:G\BC\to \cC\ ,$$
where for $X$ in $G\BC$ the group $G$ acts on $E(X)$ diagonally  via its action on $E$ and the action on $X$ by functoriality, and $(-)^{u}$ indicates the operation of forcing $u$-continuity.  A natural transformation $\phi:E\to F$ of functors as above gives rise to a transformation $\phi^{hG}:E^{hG}\to F^{hG}$. In all our examples the data determining the equivariant coarse homology theory $E^{G}$ also determines
a  non-equivariant coarse homology theory
with $G$-action $E$, and these two are related by a canonical Borelification map $\beta:E^{G}\to E^{hG}$.
The details will be given in \cref{okhphherthertgeg9}.

This finishes the rough description of the functors and transformations featuring \eqref{ferwferfwrefw}.

 \subsection{Commutativity}\label{koipgwergwerfwerfrw}

 In this section we explain how we  construct step by step commutative diagrams related with
 \eqref{ferwferfwrefw}.
 The main results are \cref{kohpethtrgertgergertgetgeh}, \cref{iowerjgowegjweriogrwegwe} and \cref{opherthrtgertg} stated below. The proofs given in the present section will use the definitions of the functors and transformations  from \cref{orkhprtegrgerg} and \cref{jogopwefewrfewfd} and depend on various technical results shown these sections.
 In order to get an overview in a first reading  one  could rely on the descriptions given in \cref{gjkopwergerfwerf} and
 accept various results cited from the later sections as granted.

 We first simplify the right part of  \eqref{ferwferfwrefw} by projecting from the Higson corona to a simpler space.
   We consider the functor 
  $$\iota^{\topp}_{u}:G\UBC\to G\LCH^{+}$$ (see   \eqref{gewrfwerfrwefwe}) defined, using Gelfand duality, such that
    $C_{0}(\iota^{\topp}_{u}(Y))\simeq C_{u,0}(Y)$, where the right-hand side is the $G$-$C^{*}$-algebra of bounded uniformly continuous functions on $Y$ which become arbitrary small outside of sufficiently large bounded subsets. 
The Higson corona of a $G$-bornological  is, if not empty,  a complicated compact Hausdorff $G$-space. In order to simplify the lower right corner in  \eqref{ferwferfwrefw} we consider 
  $Y$ in $G\UBC$ and  a map $\pi:\partial_{h}X\to \iota^{\topp}_{u}(Y)$ in $G\LCH^{+}$.
    Then we can extend the diagram \eqref{ferwferfwrefw} as follows:
 {\scriptsize \begin{equation}\label{gwerwefwerfrefw}\hspace{-1cm}
\xymatrix{ K\cX_{G_{can,max}}^{G,\ctr}(X)\ar[rr]^{c^{G}} \ar[dd]^{\rmch^{G,\alg}}&&K\cX_{G_{can,max}}^{G}(X) \ar[rr]^{T^{G,\an}}&&\Sigma K^{G,\an}(\partial_{h}X)\ar[dddd]^{\beta}\ar[dl]^{\pi}  \\  
 &&&\Sigma K^{G,\an}(\iota^{\topp}_{u}(Y))\ar[ddd]^{\beta}& \\
  \PCH\cX^{G}_{G_{can,max}}(X)\ar[d]^{\tau^{G}}&& & &   \\ \ar[d]^{\beta}\PH\cX^{G}_{G_{can,max}}(X,\C) && &  &  \\ 
\PH\cX^{hG}_{G_{can,max}}(X,\C)\ar[d]^{\pr_{X}}  &&& \Sigma K^{\an,hG}( \iota^{\topp}_{u}(Y))\ar[dd]^{\rmch^{hG}}&\ar[l]^{\pi} \Sigma K^{\an,hG}(\partial X)\ar[ddd]^{\rmch^{hG} }    \\ 
 \PH\cX^{hG}(X ,\C)\ar[dd]^{P(\chi_{\BM})^{hG}}
&  &&  &   \\ 
&&&\Sigma  (\prod_{k\in \Z} \Sigma^{2k}H\C_{\BM})^{hG}(\iota^{\topp}_{u}(Y)) &\\ 
  (\prod_{k\in \Z} \Sigma^{2k}H\C_{\BM} \bP)^{hG} (X)\ar[rrrr]^{(T^{\topp})^{hG}} &&&&\Sigma  (\prod_{k\in \Z} \Sigma^{2k}H\C_{\BM})^{hG}(\partial_{h}X)  \ar[ul]^{\pi}}\ ,
\end{equation}}where the two right squares commute by naturality of $\beta$ and $\rmch^{hG}$. It is even closer to 
  the diagram \cite[(1.2)]{Engel:2025aa}.

%
%
%
We let $\phi $ be the  clockwise composition and $\psi$ be the  counterclockwise composition
\begin{equation}\label{}\phi ,\psi  :K\cX_{G_{can,max}}^{G,\ctr}(X)\to \Sigma    (\prod_{k\in \Z} \Sigma^{2k}H\C_{\BM})^{hG}(\iota^{\topp}_{u}(Y))\ .
\end{equation} 
Note that $\psi$ and $\phi$ depend on $X$ and the map $\partial_{h}X\to \iota^{\topp}_{u}(Y)$, but this fact is   not reflected in the notation.
Our goal is to find  
 an equivalence  $$\phi\simeq \psi\ ,$$
 i.e., a filler of the big cell.


 We use the cone-at-$\infty$ functor described in \cref{koopehrrtgetrgeg}.
 We have a natural transformation
\begin{equation}\label{ewrferfrfvsdfvfd} \pi:\partial_{h}\cO^{\infty}\to \iota^{\topp}_{u}:G\UBC\to G\LCH^{+}\end{equation}
which will be described in detail after \eqref{gwreggjweiog} in  \cref{htekoperhrtgrt}.
So inserting $\cO^{\infty}(Z)$ in place of $X$ and $  Z  $ in place of $Y$ the whole diagram \eqref{gwerwefwerfrefw} becomes the  diagram \eqref{gwerwtttttefwerfrefw} of natural transformations between functors defined on $G\UBC$.

 {\scriptsize \begin{equation}\label{gwerwtttttefwerfrefw}\hspace{-2cm}
\xymatrix{ K\cX_{G_{can,max}}^{G,\ctr}(\cO^{\infty}(Z))\ar[rr]^{c^{G}} \ar[dd]^{\rmch^{G,\alg}}&&K\cX_{G_{can,max}}^{G}(\cO^{\infty}(Z)) \ar[rr]^{T^{G,\an}}&&\Sigma K^{G,\an}(\partial_{h}\cO^{\infty}(Z))\ar[dddd]^{\beta}\ar[dl]^{\pi}  \\  
 &&&\Sigma K^{G,\an}(\iota^{\topp}_{u}(Z))\ar[ddd]^{\beta}& \\
  \PCH\cX^{G}_{G_{can,max}}(\cO^{\infty}(Z))\ar[d]^{\tau^{G}}&& & &   \\ \ar[d]^{\beta}\PH\cX^{G}_{G_{can,max}}(\cO^{\infty}(Z),\C) && &  &  \\ 
\PH\cX^{hG}_{G_{can,max}}(\cO^{\infty}(Z),\C)\ar[d]^{\pr_{\cO^{\infty}(Z)}}  &&& \Sigma K^{\an,hG}( \iota^{\topp}_{u}(Z))\ar[dd]^{\rmch^{hG}}&\ar[l]^{\pi} \Sigma K^{\an,hG}(\partial \cO^{\infty}(Z))\ar[ddd]^{\rmch^{hG} }    \\ 
 \PH\cX^{hG}(\cO^{\infty}(Z) ,\C)\ar[dd]^{P(\chi_{\BM})^{hG}}
&  &&  &   \\ 
&&&\Sigma  (\prod_{k\in \Z} \Sigma^{2k}H\C_{\BM})^{hG}(\iota^{\topp}_{u}(Z)) &\\ 
  (\prod_{k\in \Z} \Sigma^{2k}H\C_{\BM} \bP)^{hG} (\cO^{\infty}(Z))\ar[rrrr]^{(T^{\topp})^{hG}} &&&&\Sigma  (\prod_{k\in \Z} \Sigma^{2k}H\C_{\BM})^{hG}(\partial_{h}\cO^{\infty}(Z))  \ar[ul]^{\pi}}\ ,
\end{equation}}
We now proceed with the following steps.
 \begin{enumerate}
 \item We consider the case of the trivial group. Then the diagram  \eqref{gwerwtttttefwerfrefw}  simplifies to \eqref{gwerwefwerfrefffw}
 \begin{enumerate}
 \item We evaluate at $Z=*$.
 In this case we can calculate the compositions and normalize the identification \eqref{vdsfpovkspdfvfdvsf} such that \eqref{gwerwefwerfrefffw} commutes. The detailed argument is given in   \cref{kophertgertgterg}.
 \item \label{okprregertgergg}We then take advantage of the facts that $K\cX^{\ctr}(\cO^{\infty}(-))$ and $  (\prod_{k\in \Z} \Sigma^{2k}H\C_{\BM} \bP)  (\iota^{\topp}_{u}(-))$ are homotopy invariant and excisive functors on $\UBC$ (see \cref{hgwiueghergwergwe} for the first), and that  \eqref{gwerwefwerfrefffw} is a diagram of natural transformations of spectrum valued functors. We   conclude  by
  \cref{koprthrtwerferfwerfwhergrtge} that
   \eqref{gwerwefwerfrefffw} naturally commutes for $Z$ a finite CW-complex, see \cref{kohpethtrgertgergertgetgeh}.  \end{enumerate} 
\item We can now generalize to the equivariant case.
\begin{enumerate}
\item \label{oigperwerferfwer} If $G$ is  finite and if  $Z$ in $G\UBC$ is homotopy invariant to a finite $G$-CW-complex, then we obtain
  the commutativity of the  Borel-equivariant version   \eqref{gwerwefwerfreffffefefedededew}  of  \eqref{gwerwefwerfrefffw} as follows. First note  that  the Borel-equivariant coarse homology theories in this diagram are naturally the versions $E^{\{e\},hG}$, in contrast to $E^{hG}$ derived from    $E$ associated to $E^{G}$ as in \cref{khopertgertgrtgertge}. But in view of  \cref{biojrgogbfgbdb} we can replace $E^{\{e\},hG}$ by $E^{hG}$.
 By homotopy invariance of the functors in \eqref{gwerwefwerfrefffw} we can assume that $Z$ is a finite $G$-$CW$-complex. Then it is a finite CW-complex with a cellular $G$-action. Using the natural commutativity
obtained in \cref{okprregertgergg} the diagram \eqref{gwerwefwerfrefffw} becomes a   commutative diagram  in $\Fun(BG,\Sp)$. We then apply $\lim_{BG}$ and restore $u$-continuity (see \cref{kophprthgertrtgegrtg})
where necessary in order to get the commutative diagram \eqref{gwerwefwerfreffffefefedededew}.

 \item  \label{okhperrtgetge}We now use the comparison of the equivariant and Borel equivariant coarse homology theories in order to conclude that the outer part diagram of the diagram \eqref{gwerwefwefffrfrefffw} commutes provided $G$ is finite and $Z$  in $G\UBC$ is homotopy equivalent to a finite $G$-CW-complex.
In order to see this  we use that all squares involving $\beta$ commute. The big inner cell is precisely  \eqref{gwerwefwerfreffffefefedededew}
 which commutes as seen in \ref{oigperwerferfwer}.
 The squares     \textcircled{1}, \textcircled{2},  \textcircled{3}   can be refined to  compositions of two squares, one involving 
 $\pr$ and the other involving $\beta$. The first type of square commutes by naturality of the respective transformation.
 For the commutativity of the cells involving the Borelification $\beta$ we refer to 
 \begin{itemize} \item \cref{khopertgertgretgtrdng} and \eqref{gwergerffssgrw4} for \textcircled{1},  
 \item  \cref{oiuehrgweroferferwfweferfwrdb} and \cref{okhopekrtgertgethehrtgertg} 
 for \textcircled{2}, \item and \cref{jihgiorhertgertgt55} and \cref{lkprthergretgtr} 
 for \textcircled{3}. 
 \end{itemize}  In each case we must combine a square of the kind \eqref{jiogjoiwegrwerfwrefrwefwfrfreew4}
 with a square of the type \eqref{ijiojgoejrofjwerfwerf234}.
The commutativity of   \textcircled{4} is shown in
 \cref{ijhgiowwerferfrewfrwefwf}.  
 
   \item We finally introduce an additional geometric data leading to the motivic transgression.
In this case we get a diagram \eqref{gwerwefwerfrefw1}. We assume that $G$ is finite and that $Y$ in $G\UBC$ is homotopy invariant to a finite $G$-CW-complex.
The maps
$p^{G,\an}$ and $p^{G,\topp}$ are the analytical and topological Paschke maps \eqref{verfwerffvsfdv} and \eqref{hkeorptkgpertgertg}.
The Paschke maps are versions of $T^{G,\an}$ or $T^{G,\topp}$ for cones, respectively. So it is not surprising that 
by  \cref{khopperttrgeg} and \cref{kophrthertgtrgetrge} 
the squares  and triangles involving  Paschke maps commute.
The big inner cell in  \eqref{gwerwefwerfrefw1}  is precisely the outer cell of   \eqref{gwerwefwefffrfrefffw} which is already seen to commute by 
\cref{okhperrtgetge}. We conclude \cref{iowerjgowegjweriogrwegwe}.
  \end{enumerate}\end{enumerate}
 
  {\scriptsize\begin{equation}\label{gwerwefwerfrefffw} \hspace{-1cm}
\xymatrix{ K\cX^{\ctr}(\cO^{\infty}(Z))\ar[rr]^{c} \ar[dd]^{\rmch^{\alg}}&&K\cX(\cO^{\infty}(Z)) \ar[rr]^{T^{\an}}&&\Sigma K^{\an}(\partial_{h}\cO^{\infty}(Z))\ar[ddddddd]^{\ch^{h}}\ar[dl]^{\pi}  \\  
 &&&\Sigma K^{\an}(\iota^{\topp}_{u}(Z))\ar[ddddd]^{\ch^{h}}& \\
  \PCH\cX(\cO^{\infty}(Z))\ar[d]^{\tau}&& & &   \\ \ar[dddd]^{P(\chi_{\BM})}\PH\cX (\cO^{\infty}(Z),\C) && &  &  \\ 
  &&& &    \\ 
 &  &&  &   \\ 
&&&\Sigma  (\prod_{k\in \Z} \Sigma^{2k}H\C_{\BM}) (\iota^{\topp}_{u}(Z)) &\\ 
  (\prod_{k\in \Z} \Sigma^{2k}H\C_{\BM} \bP)  (\cO^{\infty}(Z))\ar[rrrr]^{T^{\topp}} &&&&\Sigma  (\prod_{k\in \Z} \Sigma^{2k}H\C_{\BM})(\partial_{h} \cO^{\infty}(Z))  \ar[ul]^{\pi}}\ .
\end{equation}}

 {\scriptsize\begin{equation}\label{gwerwefwerfreffffefefedededew} \hspace{-2cm}
\xymatrix{ K\cX^{\ctr,hG}(\cO^{\infty}(Z))\ar[rr]^{c^{hG}} \ar[dd]^{\rmch^{\alg,hG}}&&K\cX^{hG}(\cO^{\infty}(Z)) \ar[rr]^{(T^{\an})^{hG}}&&\Sigma K^{\an,hG}(\partial_{h}\cO^{\infty}(Z))\ar[ddddddd]^{\ch^{hG}}\ar[dl]^{\pi}  \\  
 &&&\Sigma K^{\an,hG}(\iota^{\topp}_{u}(Z))\ar[ddddd]^{\ch^{hG}}& \\
  \PCH\cX^{hG}(\cO^{\infty}(Z))\ar[d]^{\tau^{hG}}&& & &   \\ \ar[dddd]^{P(\chi_{\BM})^{hG}}\PH\cX^{hG}(\cO^{\infty}(Z),\C) && &  &  \\ 
  &&& &    \\ 
 &  &&  &   \\ 
&&&\Sigma  (\prod_{k\in \Z} \Sigma^{2k}H\C_{\BM})^{hG} (\iota^{\topp}_{u}(Z)) &\\ 
  (\prod_{k\in \Z} \Sigma^{2k}H\C_{\BM} \bP)^{hG}  (\cO^{\infty}(Z))\ar[rrrr]^{(T^{\topp})^{hG}} &&&&\Sigma  (\prod_{k\in \Z} \Sigma^{2k}H\C_{\BM})(\partial_{h} \cO^{\infty}(Z))^{hG}  \ar[ul]^{\pi}}\ .
\end{equation}}

{\tiny\begin{equation}\label{gwerwefwefffrfrefffw} \hspace{-2.68cm}
\xymatrix{ K\cX_{G_{can,max}}^{\ctr,G}(\cO^{\infty}(Z))\ar[ddd]^{\rmch^{G,\alg}}\ar[rr]^{c^{G}}\ar[dr]^{\pr\circ \beta} &&K\cX_{G_{can,max}}^{G}(\cO^{\infty}(Z))\ar[d]^{\pr\circ \beta}\ar[rr]^{T^{\an}}&&\Sigma K^{G,\an}(\partial_{h}\cO^{\infty}(Z))\ar[d]^{\beta}\ar[ddl]^{\pi}\\
&\ar@{}[u]^{\textcircled{1}}K\cX^{\ctr,hG}(\cO^{\infty}(Z))\ar[r]^{c^{hG}} \ar[dd]^{\rmch^{\alg,hG}}&K\cX^{hG}(\cO^{\infty}(Z)) \ar[rr]^{(T^{\an})^{hG}}\ar@{}[ur]^{\textcircled{4}}&&\Sigma K^{\an,hG}(\partial_{h}\cO^{\infty}(Z))\ar[ddddddd]^{\ch^{hG}}\ar[ddl]^{\pi}  \\  
 &&&\Sigma K^{G,\an}(\iota^{\topp}_{u}(Z))\ar[d]^{ \beta} & \\
\ar@{}[uur]^{\textcircled{3}}  \PCH\cX^{G}_{G_{can,max}}(\cO^{\infty}(Z))\ar[d]^{\tau^{G}}\ar[r]^{\pr\circ \beta}& \PCH\cX^{hG}(\cO^{\infty}(Z))\ar[d]^{\tau^{hG}}& &\Sigma K^{\an,hG}(\iota^{\topp}_{u}(Z))\ar[dddd]^{\ch^{hG}} &   \\  \PH\cX^{G}_{G_{can,max}}(\cO^{\infty}(Z),\C) \ar[r]^{\pr\circ \beta}\ar@{}[ru]^{\textcircled{2}}&\PH\cX^{hG}(\cO^{\infty}(Z),\C) \ar[dddd]^{P(\chi_{\BM})^{hG}} & &  &  \\ 
  &&& &    \\ 
 &  &&  &   \\ 
&&&\Sigma  (\prod_{k\in \Z} \Sigma^{2k}H\C_{\BM})^{hG} (\iota^{\topp}_{u}(Z)) &\\ &
  (\prod_{k\in \Z} \Sigma^{2k}H\C_{\BM} \bP)^{hG}  (\cO^{\infty}(Z))\ar[rrr]^{(T^{\topp})^{hG}} &&&\Sigma  (\prod_{k\in \Z} \Sigma^{2k}H\C_{\BM})(\partial_{h} \cO^{\infty}(Z))^{hG}  \ar[ul]^{\pi}}\ .
\end{equation}}

{\tiny \begin{equation}\label{gwerwefwerfrefw1}\hspace{-3cm}
\xymatrix{K\cX_{G_{can,max}}^{G,\ctr}(X)\ar[rr]^{c^{G}}\ar[dr]^{T^{\mot}} \ar[dd]^{\rmch^{G,\alg}}&&K\cX_{G_{can,max}}^{G}(X)\ar[d]^{T^{\mot} } \ar[rr]^{T^{G,\an}}&&\Sigma K^{G,\an}(\partial_{h}X)\ar[d]^{ \beta}\ar[dl]^{t} \\ 
&K\cX^{G,\ctr}_{G_{can,max}}(\cO^{\infty}(Y))\ar[d]^{\rmch^{G,\alg}}\ar[r]^{c^{G}}&K\cX_{G_{can,max}}^{G}(\cO^{\infty}(Y))\ar[d]^{T^{G,\an}}\ar[r]^{p^{G,\an}}&\Sigma K^{G,\an}(\iota^{\topp}_{u}(Y))\ar[d]^{ \beta} &\Sigma K^{\an,hG}(\partial_{h}X)\ar[dddddd]^{\ch^{hG}}\ar[dl]^{t} \\
\PCH\cX_{G_{can,max}}^{G}(X)\ar[d]^{\tau^{G}}\ar[r]^{T^{\mot} }&\PCH\cX_{G_{can,max}}^{G}(\cO^{\infty}(Y)) \ar[d]^{\tau^{G}}&    \Sigma K^{G,\an}(\partial_{h} \cO^{\infty}(Y)) \ar[ur]^{\pi}& \Sigma K^{\an,hG}(\iota^{\topp}_{u}(Y))\ar[dddd]^{\ch^{hG} } & \\
\PH\cX_{G_{can,max}}^{G}(X)\ar[d]^{ \pr\circ  \beta} \ar[r]^{T^{\mot }}& \PH\cX_{G_{can,max}}^{G}(\cO^{\infty}(Y))\ar[d]^{ \pr\circ  \beta}  &       &  &  \\
\PH\cX^{hG}(X)\ar[ddd]^{P(\chi_{\BM})^{hG} } \ar[r]^{T^{\mot }}& \PH\cX^{hG}(\cO^{\infty}(Y))\ar[dd]^{P(\chi_{\BM})^{hG} }&& & \\
&  &(\prod_{k\in \Z}  \Sigma^{2k} H\C_{\BM})^{hG}(\partial_{h}\cO^{\infty}(Y)) \ar[dr]^{\pi}& &  \\ 
& (\prod_{k\in \Z}  \Sigma^{2k} H\C_{\BM}\bP)^{hG}(\cO^{\infty}(Y))\ar[ur]^{(T^{\topp})^{hG}}\ar[rr]^{p^{G,\topp}}&&\Sigma (\prod_{k\in \Z}  \Sigma^{2k} H\C_{\BM})^{hG}(\iota^{\topp}_{u}(Y)) &\\
  \ar[ur]^{T^{\mot} }(\prod_{k\in \Z}  \Sigma^{2k} H\C_{\BM}\bP)^{hG}(X)\ar[rrrr]^{(T^{\topp})^{hG}}\ &&&&\ar[ul]^{t}\Sigma (\prod_{k\in \Z}  \Sigma^{2k} H\C_{\BM})^{hG}(\partial_{h}X) }
\end{equation}}

%

 \begin{prop}[{\cref{kophertgertgterg}}]
There is a unique  choice of the marked equivalence in \eqref{vdsfpovkspdfvfdvsf}   such that  the diagram \eqref{gwerwefwerfrefffw} commutes for $Z=*$.   \end{prop}

%
%

By  \cref{koprthrtwerferfwerfwhergrtge}  we get

\begin{kor}\label{kohpethtrgertgergertgetgeh}
If $Z$ in $\UBC$ is  a finite $CW$-complex, then
 the diagram \eqref{gwerwefwerfrefffw}   commutes naturally in $Z$.
\end{kor}

\begin{kor}\label{iowerjgowegjweriogrwegwe}
If $G$ is finite, $X$ is in $G\BC$, and $Y$ in $G\UBC$   is homotopy invariant to a finite $G$-CW-complex, and 
  $\pi:\partial_{h}X\to \iota^{\topp}_{u}(Y)$ is such such that the motivic transgression   $T^{\mot}$ exists, then \eqref{gwerwefwerfrefw} commutes. 
\end{kor}

The following is a consequence of \cref{opkherptokgpertgrtge} and \cref{iowerjgowegjweriogrwegwe}.
Note that if $Y$ is a compact Hausdorff space, then it is an object of $\UBC$ in the canonical way and
$\iota^{\topp}_{u}(Y)\cong Y$. 
\begin{kor}\label{opherthrtgertg}
Assume that $G$ is the trivial group.
If $X$ is a metrizable bornological coarse space and  $\bar X$ is  a  metrizable Higson-dominated compactification of $X$  by a non-empty finite CW-complex $Y$, then \eqref{gwerwefwerfrefw} commutes.   \end{kor}

See also \cref{hojpertrtgertgerth9} and \cref{hojpertrtgertgerth91} for more general conditions implying the existence of the motivic transgression.

\begin{ex}
  We consider $X=\R^{n}$ with the standard bornological coarse structure induced by the metric.
  It has a metrizable Higson-dominated compactification by the sphere $S^{n-1}$. Since $S^{n-1}$ has the structure of a finite CW-complex,  by  \cref{opherthrtgertg}
   the diagram  \eqref{gwerwefwerfrefw} commutes  for $X=\R^{n}$ and $\partial_{h}\R^{n}\to S^{n-1}=:Y$.
  \hB
    \end{ex}

%
%

\begin{rem} The square \eqref{gwerwefwerfrefw1} commutes as a diagram of functors on the category of triples $(X,Y,\partial_{h}X\to \iota^{\topp}_{u}(Y))$, where $Y$ is   a finite CW-complex.  \hB
\end{rem}

\begin{rem}
In \cite[Thm. B]{Engel:2025aa} 
it is assumed that $X$ is presented by a complete Riemannian manifold 
and $\bar X$
 is a metrizable 
 Higson-dominated compactification with boundary $Y$ such that $(\bar X,\iota^{\topp}_{u}(Y))$ is a finite $CW$-pair. In this case
 \cref{opherthrtgertg} applies.   So by our result  we   drop the smoothness assumption on $X$. Furthermore we only need that $Y$ is a finite CW-complex and have no finiteness assumption on $\bar X$.
%
%
%
%
%
  \hB
 \end{rem}

\section{Functors and natural transformations}\label{orkhprtegrgerg}

\subsection{Bornological coarse spaces and the Higson corona functor}\label{kopehrtgegrtgrtge}
 
A bornological coarse space is a set equipped with a coarse structure and a compatible bornology. A morphism between bornological coarse spaces is a map of sets which is controlled (images of coarse entourages are coarse entourages) and proper (preimages of bounded subsets are bounded). We let $\BC$ denote 
 the
  symmetric monoidal category of bornological coarse spaces. For details we refer to  \cite[Sec. 2]{buen}. 
  
 We let $\CH$ denote the category of compact Hausdorff spaces and  continuous maps.  By Gelfand duality it is equivalent to the opposite of the category $\Calg_{\comm}$ of commutative unital $C^{*}$-algebras and unital homomorphisms.

In the following we explain the Higson corona construction which provides a functor
$$\partial_{h}:\BC\to \CH\ .$$  We recall the notation   from  \cite[Def. 3.1]{Bunke:2024aa}.
For a set $X$ we consider the $C^{*}$-algebra $\ell^{\infty}(X)$ of bounded functions
$f:X\to \C$ with the supremum norm.
For  $f$  in $\ell^{\infty}(X)$, an entourage $U$ of $X$, and  a subset $Y$ we define the $U$-variation of $f$ on $Y$ by
$$\Var_{U}(f,Y):=\sup_{(x,x')\in U\cap Y\times Y }|f(x)-f(x')|\ .$$
If $X$ is a bornological coarse space with coarse structure $\cC$ and bornology $\cB$, then
we consider the unital $C^{*}$-algebra
\begin{equation}\label{gerwfwerfwerfwf}\ell^{\infty}_{\cB}(X):=\{f\in \ell^{\infty}(X)\mid  \forall U\in \cC\colon \lim_{B\in \cB} \Var_{U}(f,X\setminus B)=0\}
\end{equation}  of bounded functions with the property that for every coarse entourage $U$ of $X$ the $U$-variation  becomes arbitrary small
outside of sufficiently large bounded subsets.
The $C^{*}$-algebra
\begin{equation}\label{gerwfwerfwerfwf1}\ell^{\infty}(\cB):=\{f\in \ell^{\infty}(X)\mid \lim_{B\in \cB} \sup_{x\in X\setminus B} |f(x)|=0\}\end{equation}  
of bounded functions which themselves  become arbitrary small 
outside of sufficiently large bounded subsets is an ideal in $\ell^{\infty}_{\cB}(X)$.
We then form   the commutative unital $C^{*}$-algebra
\begin{equation}\label{brtberbfgbdgetb}C(\partial_{h}X):=\ell^{\infty}_{\cB}(X)/\ell^{\infty}(\cB)\ .
\end{equation}  \begin{ddd}\label{kiogwergewrfwrefw}  The Higson corona $\partial_{h}X$ in $\CH$  is defined to be the compact Hausdorff space corresponding to
  $C(\partial_{h}X)$  by  Gelfand duality .\end{ddd}
 Note that a morphism $X\to X'$  in $\BC$ functorially induces via pull-back a morphism of exact sequence
   of the associated commutative $C^{*}$-algebras 
  $$
\xymatrix{
0\ar[r]&\ell^{\infty}(\cB_{X'})\ar[d]\ar[r]&\ell^{\infty}_{\cB_{X'}}(X')\ar[r]\ar[d]&C(\partial_{h}X')\ar[r]\ar[d]&0\\
0\ar[r]&\ell^{\infty}(\cB_{X})\ar[r]&\ell^{\infty}_{\cB_{X}}(X)\ar[r]&C(\partial_{h}X)\ar[r]&0}\ .
$$
We get a functorially induced map between the Higson coronas
 $$\partial_{h}f:\partial_{h}X\to \partial_{h}X'\ .$$
 This finishes the construction of the Higson corona functor.
 
 \begin{lem} \label{kopbrwgregrefwr}The Higson corona functor is coarsely invariant in the sense that if $f$ is a coarse equivalence, then
 $\partial_{h}f$ is a homeomorphism.\end{lem}
 \begin{proof}
 If $f,g:X\to X'$ are close maps between bornological coarse spaces, then $\partial_{h}f=\partial_{h}g$.
 \end{proof}

Let $G$ be a group.
 Recall from  \cite{equicoarse} that a $G$-bornological coarse space is a bornological coarse space with an action of $G$ by automorphisms such that the coarse structure $\cC$ contains a cofinal subset  $\cC^{G}$ of invariant entourages. The symmetric monoidal category $G\BC$ of  $G$-bornological coarse spaces   
 forms a  full subcategory
\begin{equation}\label{vweujiofwvwefewfvf}\hat \Res^{G}:G\BC\hookrightarrow \Fun(BG,\BC)
\end{equation}  of the symmetric monoidal  category of   bornological coarse spaces with an action of $G$ and equivariant morphisms.
  By functoriality,    the Higson corona functor  induces a functor
 $$\partial_{h}:\Fun(BG,\BC)\to G\CH\ ,$$
 where $G\CH$ is
 the category of   compact Hausdorff spaces with $G$-action and equivariant  continuous maps. The latter is equivalent to the opposite of 
 $G\Calg_{\comm}$ by Gelfand duality.

\begin{ex}
The Higson corona of a bounded bornological coarse space is empty. 
Furthermore, the Higson corona of a discrete bornological coarse space
with the minimal bornology is the boundary of the Stone-\v{C}ech compactification
of the set considered as a discrete topological space. \hB
\end{ex}

\subsection{Equivariant coarse homology theories}\label{okhphherthertgeg9}

In this section we recall the notion of an equivariant coarse homology theory. We further introduce the associated non-equivariant coarse homology theory and the corresponding Borel-equivariant coarse homology theory.  Using the additional structure of transfers we construct the Borelification map.

  Let  $$E^{G}:G\BC\to \cC$$ be some functor. 
\begin{rem}
 We will often use a simplified notation $E$ also for functors defined on $G\BC$.
 The superscript $G$ becomes relevant for a consistent notation if, as  later in this section, we want to talk about the non-equivariant theory
 $E$ associated to $E^{G}$, the theory $E^{\{e\}}$ and
 the corresponding Borel equivariant theory $E^{hG}$. The latter then denote functors derived in a natural way from $E^{G}$.
  \hB
 \end{rem}
Following \cite{equicoarse} we adopt the following definition.     
\begin{ddd}\label{okprhertgrtge9}
The  functor $E^{G} :G\BC\to \cC$  is an equivariant coarse homology theory if $\cC$ is   a cocomplete stable $\infty$-category  and $E^{G}$   is coarsely invariant, excisive, $u$-continuous and vanishes on flasques. \end{ddd}

 \begin{rem} An equivariant coarse homology  is called strong if it in addition annihilates weakly flasques.
 We refer to  \cite{equicoarse} for a detailed explanation of these conditions.
 The property of  strongness is relevant when we consider the composition $E^{G}\circ \cO^{\infty}:G\UBC\to \cC$
since it ensures that this composition is not only homotopy invariant, excisive and $u$-continuous, but also vanishes on  flasques and is therefore a local homology theory in the sense of \cite{ass}, see \cref{hgwiueghergwergwe}. \hB
 \end{rem}

 Note that the conditions  for an equivariant coarse homology theory only involve colimits.

 \begin{kor}\label{oigpwerferfwfrfrfr} If $\phi:\cC\to \cC'$ is a colimit-preserving functor between cocomplete stable $\infty$-categories and $E^{G}:G\BC\to \cC$ is an equivariant  coarse homology theory, then the composition $\phi\circ E^{G}:G\BC\to \cC'$ is again an equivariant   coarse homology theory.
 \end{kor}
 
 If $E^{G}$ is strong, then so is  $\phi\circ E^{G}$.

 \begin{ex}\label{hokrptogkertpgerthe9}
 Recall from \cite[Sec. 4]{equicoarse} that there is a universal equivariant coarse homology theory
\begin{equation}\label{gweroihgwuierjfoiewrfiojwerofwerfw} \Yo^{G}:G\BC\to G\Sp\cX\ .
\end{equation}By design, for every cocomplete stable $\infty$-category $\cC$ pull-back along $\Yo^{G}$ induces an equivalence
 $$\Yo^{G,*}:\Fun^{\colim}(G\Sp\cX,\cC)\to \Fun^{\cX}(G\BC,\cC)\ ,$$
 where the superscript $\cX$ indicates the full subcategory of $\Fun(G\BC,\cC)$ of  equivariant coarse homology theories.
  If $E^{G}:G\BC\to  \cC$ is an equivariant  coarse homology theory, then we will use the same symbol for the corresponding colimit-preserving functor $E^{G}:G\Sp\cX\to \cC$. \hB
\end{ex}

\begin{rem} \label{owopgeerfwerfrewfw}Assume that $\cC$ is a cocomplete stable $\infty$-category.
The full subcategory $\Fun^{\cX}(G\BC,\cC)$ of $ \Fun(G\BC,\cC)$ of equivariant coarse homology theories is closed under colimits. If $\cC$ is presentable, then $ \Fun(G\BC,\cC)$ is also presentable and we 
 have an adjunction
$$\incl:\Fun^{\cX}(G\BC,\cC)\leftrightarrows \Fun(G\BC,\cC):R\ .$$
The functor $R$ thus provides a universal way to approximate any functor $E:G\BC\to \cC$ by an equivariant  coarse homology theory
$RE:G\BC\to \cC$. The counit of the adjunction is a natural transformation $RE\to E$ with the universal property that for any coarse homology $F:G\BC\to \cC$ with a natural transformation $F\to E$ we have a unique factorization
\begin{equation}\label{zwetrwertwret}
\xymatrix{&RE\ar[d]\\F\ar@{..>}[ur]\ar[r]&E}\ .
\end{equation}
In general it is complicated to control $RE$. But in the case that $E$ is already coarsely invariant, excisive, and vanishes on flasques, then $R$ just forces $u$-continuity, and in this case we have an explicit formula, see \cref{kophprthgertrtgegrtg}.
\hB
\end{rem}

\begin{ex}\label{khoeprthretgertge}
We consider the Higson corona functor $\partial_{h}:G\BC\to G\CH$ from \cref{kiogwergewrfwrefw}. Let $E:G\CH\to \cC$ be any functor to a presentable stable $\infty$-category $\cC$. Then by \cref{owopgeerfwerfrewfw} we can form an equivariant  coarse homology theory
$$R(E\partial_{h}):G\BC\to \cC\ .$$ 
It comes with a natural transformation
$R(E\partial_{h})\to E\partial_{h}$. In this generality it is difficult to understand the functor $R(E\partial_{h})$.

The analytic transgression considered in \cref{kophertggertgrtger} has a factorization \eqref{kophertggertgrtgehertgrtr}
which can be used to show that $R(K_{\bC,A,G_{can,max}}^{G,\an}\partial_{h})$ is non-trivial by showing that
$T^{G,\an}$ is non-trivial. 

Similarly we can use the factorization \eqref{gerfwerfweg542zze} of the topological transgression introduced in \cref{kogprwfregergerfewrf} in order to show that
$R(E\partial_{h})$ is non-trivial for a weak  equivariant Borel-Moore homology theory $E:G\LCH^{+}\to \cC$. 

It seems to be an interesting problem in coarse homotopy theory to understand the coarse homology theories of the form $R(E\partial_{h})$ in greater detail.
\hB
\end{ex}

 \begin{ex}\label{oihperthertgertgetrg}
 If $Y$ is a $G$-bornological coarse space and $E^{G}:G\BC\to \cC$ is an equivariant coarse homology theory, then using the symmetric monoidal structure of $G\BC$ we can define a new equivariant coarse homology theory
$$E^{G}_{Y}:=E^{G}(-\otimes Y):G\BC\to \cC\ , $$
called the twist of $E^{G}$ by $Y$. We refer to \cite{equicoarse} for justifications.  \hB
 \end{ex}

\begin{construction}  \label{kophprthgertrtgegrtg}{\em
Sometimes we want to post-compose coarse homology theories with exact functors which may not preserve filtered colimits. Then coarse invariance, excision and vanishing on  flasques or weakly flasques  is preserved, but not $u$-continuity.
In this case we can apply 
 the construction $F\mapsto F^{u}$ of forcing $u$-continuity  \cite[Sec. 6.1]{buen} in order to 
 obtain again a coarse homology theory. If $\cC$ is presentable, then this construction is special case of the application of the functor $R$ from \cref{owopgeerfwerfrewfw}.
  
    We form    the category $G\BC^{\cC}$ of pairs $(X,U)$ of $X$ in $G\BC$ and $U$ in $\cC_{X}^{G}$.
A map
$f:(X,U)\to (X',U')$ is a morphism $f:X\to X'$ in $G\BC$ such that $f(U)\subseteq U'$.
We consider 
  a functor $F:G\BC\to \cC$ to a cocomplete $\infty$-category.
We can then consider the diagram of  functors
$$\xymatrix{G\BC^{\cC}\ar[d]^{(X,U)\mapsto X}\ar[rr]^{(X,U)\mapsto X_{U}}&&G\BC\ar[r]^{F}&\cC \\ G\BC\ar@{..>}[urrr]_{F^{u}}&&&}\ ,$$
where $X_{U}$ is the $G$-bornological coarse space defined by equipping the underlying 
$G$-bornological space of $X$   with the coarse structure generated by $U$, and the functor
$F^{u}$ is defined by left Kan-extension.
We say that $F^{u}$ is obtained by forcing $u$-continuity.
By the point-wise formula we have
\begin{equation}\label{bsdfvdfvsfvsfdvs}F^{u}(X)\simeq \colim_{U\in \cC_{X}^{G}} F(X_{U})\ .
\end{equation}
The functor $F^{u}$ inherits  the following properties of $F$: coarse invariance, excision, vanishing on flasques.
In addition,  $F^{u}$ is $u$-continuous by design. 
The natural transformation
$$G\BC^{\cC}\ni (X,U)\mapsto (X_{U}\to X)\in \Fun(\Delta^{1},G\BC)$$
induces a natural transformation
 \begin{equation}\label{gwezghuowerfwerf} F^{u}\to F:G\BC\to \cC \end{equation}
which has the following universal property:  If $E$ is $u$-continuous, then the transformation
$$\Nat(E,F^{u})\to \Nat(E,F)$$ induced by \eqref{gwezghuowerfwerf}
is an equivalence.
 \hB}
\end{construction}

We will apply this construction as follows.

Assume that $\cC$ is a stable, complete and cocomplete $\infty$-category, and that
$E^{G}:G\BC\to \cC$ is an equivariant coarse homology theory.
\begin{ddd} \label{iugwierogjwe8r9uf98werfwerfwerwerf}The $2$-periodization of  $E^{G}$
is defined by 
$$P(E^{G}):=(\prod_{k\in \Z} \Sigma^{2k} E^{G})^{u}:G\BC\to \cC\ .$$
\end{ddd}
The functor $\prod_{k\in \Z} \Sigma^{2k} E^{G}$
is coarsely invariant, excisive and vanishes on flasques, but it is no longer $u$-continuous in general since the infinite product does not commute with filtered colimits.  For this reason we force $u$-continuity in the second step.
%
%

Let $G$ be a group and \begin{equation}\label{gwer0ir0dfvdvvdfvwerfwerf}E:BG \times \BC\to  \cC
\end{equation}  be a functor to a complete $\infty$-category.
Then we can form
$$\lim_{BG}E:G\BC\xrightarrow{\hat \Res^{G},\eqref{vweujiofwvwefewfvf}} \Fun(BG,\BC)\xrightarrow{E}\Fun(BG\times BG,\cC) \xrightarrow{\diag_{BG}^{*}}  \Fun(BG,\cC)\xrightarrow{\lim_{BG}}\cC\ .$$

\begin{ddd}\label{khopertgertgetrg} If the adjoint $BG\to \Fun(\BC,\cC)$ of   $E$  in \eqref{gwer0ir0dfvdvvdfvwerfwerf} takes values   in coarse homology theories,
then we say that $E$ is a coarse homology theory with $G$-action. \end{ddd} In this case
the functor $   \lim_{BG}E$ is  coarsely invariant, excisive and vanishes on flasques, but it is no longer $u$-continuous in general since the  limit over $BG$ does not commute with filtered colimits   unless $BG$ is  finite.  In order to retain an equivariant  coarse homology theory   we force $u$-continuity in the second step.

\begin{ddd}\label{oiog0wergerfwerferwfr}
We define the Borel-equivariant coarse homology theory associated to a coarse homology theory $E$ with $G$-action by
$$E^{hG}:=(\lim_{BG}E)^{u}:G\BC\to \cC\ .$$
\end{ddd}
If $$\phi:E\to F:BG\times G\BC\to \cC$$ is a natural transformation  between coarse homology theories with $G$-actions, then
we get a functorially induced natural transformation 
\begin{equation}\label{iojiojoivj dffv d}\phi^{hG}:E^{hG}\to F^{hG}:G\BC\to \cC
\end{equation}
between their associated  Borel-equivariant versions.

Recall that a morphism between $G$-bornological coarse spaces is a  weak coarse equivalence  if it becomes a coarse equivalence after forgetting the $G$-action. The following is an immediate consequence of the definition.

\begin{kor}
The Borel-equivariant coarse homology theory $E^{hG}$ associated to $E$  
sends weak coarse equivalences to equivalences.
\end{kor}

\begin{ex}\label{ojhperthertgrtge}The projection map
\begin{equation}\label{frefwrefwrefwesdfv}X\otimes G_{max,max}\to X
\end{equation}
is a weak coarse equivalence, but not a coarse equivalence unless the $G$-action on the underlying set of $X$ is free.

Let $$G\BC\to G\BC_{h}$$ be the quotient which identifies  close maps. By \cite{heissdiss} it represents the Dwyer-Kan localization of $G\BC$ at the coarse equivalences.
 The object $G_{max,max}$ with the projection $G_{max,max}\to *$ and the diagonal
$G_{max,max}\to G_{max,max}\otimes G_{max,max}$ is an idempotent commutative coalgebra in $G\BC_{h}$.
 This coalgebra induces a right Bousfield  localization of $G\BC_{h}$ which inverts the weak coarse equivalences.
The map \eqref{frefwrefwrefwesdfv} is the component of the  counit of this localization at $X$.

In Diagram \eqref{ferwferfwrefw} we twist some equivariant coarse homology theories with $G_{can,max}$. If $G$ is finite, then $G_{can,max}\cong G_{max,max}$. In this case twisting by $G_{can,max}$ is the universal way to
make the equivariant coarse homology theory weakly coarsely  invariant.
 \hB\end{ex}
 
 Let us now start from an 
  equivariant coarse homology theory
 $$E^{G}:G\BC\to \cC\ .$$   
 Then, using \cref{oihperthertgertgetrg}, we define a  coarse homology theory with $G$-action by
 $$E:BG\times \BC\to \cC\ , \quad X\mapsto  E(X):=E^{G}(\Res_{G}(X)\otimes G_{min,min})\ ,$$
 where the $G$-action comes from the right-action of $G$ on $G_{min,min}$ and $\Res_{G}$
 equips $X$ with the trivial $G$-action.

 \begin{ddd}\label{khopertgertgrtgertge}
 We call $E$ the non-equivariant coarse homology theory associated to $E^{G}$.
 \end{ddd}
 To $E$ we can further associate the Borel-equivariant coarse homology theory $E^{hG}$ according to \cref{oiog0wergerfwerferwfr}.
 In order to relate $E^{G}$ with $E^{hG}$
 we now assume that $E^{G}$ extends to   an equivariant coarse homology theory with transfers $E^{G}_{\tr}$  in the sense of \cite{trans}:\begin{equation}\label{gowiejhofwerfwerfwerfwew}\xymatrix{G\BC\ar[dr]\ar[rr]^{E^{G}}&&\cC\\&G\BC_{\tr}\ar@{..>}[ur]_{E^{G}_{\tr}}&}
\end{equation} Here $G\BC_{\tr}$ is the $(2,1)$-category of spans
 $$\xymatrix{&W\ar[dr]^{f}\ar@{-->}[dl]_{g}&\\X&&Y}\ ,$$
 where $f:W\to X $ is a morpism in $G\BC$ which is in addition bornological, and $g$ is a controlled and bornological map which is in addition a bounded coarse covering, see \cite{trans} for details. We use a dashed arrow in order to make clear that $g$ is in general not a morphism in $G\BC$ since it may not be proper.

 \begin{ex}
 If $S$ is a $G$-set, then the projection
 $$\xymatrix{X\otimes S_{min,min}\ar@{-->}[r]^{\pr_{X}}& X}$$ is a typical example of a bounded coarse covering. 
 In fact, locally on coarse components  any bounded coarse covering over $X$
 is of this form.

 In contrast, if  $G$ is non-trivial, then the morphism \eqref{frefwrefwrefwesdfv} is not a bounded coarse covering.
\hB 
 \end{ex}

 \begin{construction} \label{} {\em Under the assumption that $E^{G}$ has an extension $E^{G}_{\tr}$ as in \eqref{gowiejhofwerfwerfwerfwew}  we 
  construct a   natural transformation
\begin{equation}\label{irjgoiwerferfweferf} \beta:E^{G}\to E^{hG} \ .
\end{equation} 
Recall that $$E^{hG}\simeq (\lim_{BG}\circ   E\circ \hat \Res^{G})^{u}:G\BC\to \cC\ .$$
In view of the universal property of $(-)^{u}\to \id$ (see \cref{kophprthgertrtgegrtg})  and the fact that $E^{G}$ is $u$-continuous we must construct a map
$$ E^{G}\to E^{hG}:=\lim_{BG}\circ E\circ \hat \Res^{G}:G\BC\to \cC\ .$$ 
We will actually construct its adjoint
$$\underline{E^{G}}\to E\circ \hat \Res^{G}:G\BC\to \Fun(BG,\cC)\ .$$
 Its value at $X$ in $G\BC$ is  a map
\begin{equation}\label{iogjoewrfwerfwerfwerfwef} \underline{E^{G}(X)}\to E^{G}(\hat \Res^{G}(X)\otimes G_{min,min})\ .
\end{equation}
 In $\Fun(BG,G\BC)$ we have a canonical isomorphism
 \begin{equation}\label{gwerhfiuhwerifewrfwer}X\otimes G_{min,min}\stackrel{\cong }{\to}\hat \Res^{G}(X)\otimes G_{min,min}\ , \quad (x,g)\mapsto (g^{-1}x,g)\ ,
\end{equation}  where $G$ acts on  $G_{min,min}$ by right multiplication.
 The transfer is a map
 \begin{equation}\label{bsfdwergwregwergwevodfsdbdf}\tr:X\to X\otimes G_{min,min}
\end{equation} in $G\BC_{\tr}$ represented by the span $$\xymatrix{&X\otimes G_{min,min}\ar[dr]^{\id}\ar@{-->}[dl]_{\pr_{X}}&\\X&&X\otimes G_{min,min}}\ ,$$

 We claim that it refines to an equivariant map \begin{equation}\label{gweoiroifjoewfwerfwerferwferfrffw}(\tr,\rho):\underline{X}\to X\otimes G_{min,min}
\end{equation}
 in $\Fun(BG,G\BC_{\tr})$.
 For every $g$ in $G$ we have a two-isomorphism $\rho_{g}: \tr\to g\circ \tr$ given by 
 $$\xymatrix{&& X\otimes G_{min,min}\ar[dddd]^{\id\otimes -\cdot g^{-1}} \\&  \ar@{-->}[dl]_{\pr_{X}}X\otimes G_{min,min} \ar[ur]^{\id}& \\X& &\\&\ar@{-->}[ul]^{\pr_{X}}X\otimes G_{min,min}\ar[dr]_{\id}\ar[uu]_{\id\otimes -\cdot g^{-1}}& \\&&X\otimes G_{min,min}}\ .$$
    The family $ \rho=(\rho_{g})_{g\in G}$ satisfies a cocycle condition.
 Applying the extension $E_{\tr}^{G}$ of $E^{G}$ to the  map $$\underline{X} \xrightarrow{(\tr,\rho)} X\otimes G_{min,min} \stackrel{\eqref{gwerhfiuhwerifewrfwer}}{\simeq} \hat \Res^{G}(X)\otimes G_{min,min} $$ in $\Fun(BG,G\BC_{\tr})$ we get the desired map \eqref{iogjoewrfwerfwerfwerfwef}  in $\Fun(BG,\cC)$.
 This construction is natural in $X$.
 
 This finishes the construction of the  natural transformation \eqref{irjgoiwerferfweferf}.
\hB }  \end{construction}

  \begin{ddd}\label{jgoijowerfewrfewfwfwef}
  If $E^{G}$ is an equivariant coarse homology theory with a given extension $E^{G}_{\tr}$, then we call
  the map
  $$\beta:E^{G}\to E^{hG}$$ in  \eqref{irjgoiwerferfweferf}
  the Borelification map.
  \end{ddd}
  
  Note that  the Borelification map depends on the choice of the extension $E^{G}_{\tr}$.

 Let $$\phi^{G}:E^{G}\to F^{G}:G\BC\to \cC$$ be a natural transformation  between equivariant coarse homology theories.
 Then we get an induced natural transformation \begin{equation}\label{bjibjsdofvsdfvsdfv}\phi:E\to F:BG\times G\BC\to \cC
\end{equation} of associated coarse homology theories with $G$-action.
 We then let  \begin{equation}\label{boijiodfbfgbdfgbd}\phi^{hG}:E^{hG}\to F^{hG}:G\BC\to \cC
\end{equation} denote the induced transformation of Borel-equivariant coarse homology theories. The following is an immediate consequence of the constructions.
 
 \begin{kor}\label{ogjwioergjeworijuoiewferfwwrf}
 If $\phi^{G}$ extends to a  natural transformation $$\phi^{G}_{\tr}:E^{G}_{\tr}\to F^{G}_{\tr}:G\BC_{\tr}\to \cC$$  between equivariant coarse homology theories with transfers, then  the following square commutes:
   \begin{equation}\label{jiogjoiwegrwerfwrefrwefwfrfreew4}\xymatrix{E^{G}\ar[r]^{\beta_{E}}\ar[d]^{\phi} &E^{hG} \ar[d]^{\phi^{hG}} \\ F^{G} \ar[r]^{\beta_{F}} &F^{hG} } \ .
\end{equation} 
 \end{kor}
Note that the filler of this square may depend on the choice of $\phi^{G}_{\tr}$.

 \begin{rem}\label{biojrgogbfgbdb}
 In our examples  we have  for every group $G$ a family $(E^{H})_{H\subseteq G}$ of  coarse homology theories $E^{H} :H\BC\to \cC$ for all subgroups of $G$.
 We than have two versions of a non-equivariant coarse homology theory, namely $E$ as defined in \cref{khopertgertgrtgertge} and the member  $E^{\{e\}}$ of the family. In our  examples the coarse homology theories $E^{H}$    furthermore depend on the choice of some parameters on which $G$ acts. In particular, we get a $G$-action on $E^{\{e\}}$ from the action on the parameters. 

 In all  of our examples we will show that there is a canonical equivalence $$E^{\{e\}}\simeq E:BG\times \BC\to \cC\ .$$
The natural transformations in our examples  also  come in families
$(\phi^{H}:E^{H}\to F^{H})_{H\subseteq G}$ 
of natural transformations. The transformation $\phi^{G}$ induces $\phi:E\to F$ according to \eqref{bjibjsdofvsdfvsdfv}.
We will furthermore show that
the squares \begin{equation}\label{ijiojgoejrofjwerfwerf234}\xymatrix{E\ar[r]^{\phi}\ar[d]^{\simeq} &F \ar[d]^{\simeq} \\ E^{\{e\}}\ar[r]^{\phi^{\{e\}}} &F^{\{e\}} } 
\end{equation}
  commute.
  These facts allow to simplify the notation an not to distinguish between $E$ and $E^{\{e\}}$ anymore.
   \hB
  \end{rem}

   \subsection{Uniform bornological coarse spaces, topological bornological spaces,  and locally compact Hausdorff spaces}\label{kopwhwthrh}
 
  A  partially defined equivariant proper map $X\to Y$ between locally compact Hausdorff $G$-spaces is a pair $(U\subseteq X,U\to Y)$
  of an invariant open subset $U$ of $X$ and an equivariant 
   proper continuous map $ U\to Y$.
   The category $G\LCH^{+}$ of locally compact Hausdorff $G$-spaces and partially defined equivariant proper maps  
is by Gelfand duality equivalent to the opposite  of the category $G\nCalg_{\comm}$ of possibly non-unital commutative $G$-$C^{*}$-algebras. It contains the wide subcategory   $G\LCH^{\propp} $  of everywhere defined equivariant proper maps.
 Under Gelfand duality the latter corresponds to the wide subcategory of $G\nCalg_{\comm}$ of essential homomorphisms.

A $G$-uniform bornological coarse space is a $G$-bornological coarse space with an additional $G$-uniform structure which is compatible with coarse structure. A morphism
 between $G$-uniform  bornological coarse spaces is a morphism of $G$-bornological coarse spaces which is in addition uniformly continuous.  We have a symmetric monoidal category 
 $G\UBC$  of $G$-uniform bornological coarse spaces, see  \cite{buen}, \cite{ass}, \cite{equicoarse} for details.

 A $G$-bornological topological space is a $G$-topological space with an additional compatible $G$-invariant bornology. A morphism between $G$-bornological topological spaces is an equivariant continuous map which is in addition proper in the bornological sense.
We have the symmetric monoidal category 
  $G\BT$ of
  topological bornological coarse spaces. We refer to   \cite{buen} for details.

  These categories are related by functors 
  \begin{equation}\label{gewrfwerfrwefwe}\xymatrix{G\BC\ar[rd]^{\iota^{\topp}_{\ell^{\infty}}}&\\G\UBC\ar[u]^{\iota} \ar[d]_{\iota'}\ar[r]^{\iota_{u}^{\topp}}&G\LCH^{+}\\G\BT\ar[ur]^{\iota^{\topp}}} .
\end{equation}  
The functor
 $\iota$ forgets the uniform structure, and $\iota'$ forgets the coarse structure and takes the underlying topological space presented by the uniform space.
We will often omit the symbols $\iota$ and $\iota'$.
The other functors 
are  characterized via Gelfand duality as follows:
\begin{enumerate}
\item $C_{0}(\iota^{\topp}_{\ell^{\infty}}(X))=\ell^{\infty}(\cB_{X})$
\item $C_{0}(\iota^{\topp}_{u}(X))= C_{u}(X)\cap \ell^{\infty}(\cB_{X})=:C_{0,u}(X)$
\item $C_{0}(\iota^{\topp}(X))= C(X)\cap \ell^{\infty}(\cB_{X})=:C_{0}(X)$\ .
\end{enumerate} 
Here $C(Z)$ denotes the $C^{*}$-algebra of bounded continuous functions on the topological space $Z$ with the supremum norm. If $Z$ is a uniform space, then 
  $C_{u}(Z)$ denotes  the subalgebra of $C(Z)$  of uniformly continuous functions.
For   $Z$ in $G\LCH^{+}$ the $C^{*}$-algebra
$C_{0}(Z) $ denotes the subalgebra of the $C^{*}$-algebra $C(Z_{+})$ of  functions on the one-point compatification $Z_{+}$ which satisfy $f(+)=0$.
%
 
The triangles in  \eqref{gewrfwerfrwefwe} do not commute but are filled with natural transformations
$$ \iota^{\topp}_{\ell^{\infty}}\circ \iota \to \iota^{\topp}_{u} \ , \qquad \iota^{\topp} \circ \iota' \to \iota^{\topp}_{u}  $$
induced by the obvious inclusions of   $C^{*}$-algebras indicated by the morphisms in the left column in \eqref{fqewfqwefqwd} below.

For $X$ in $G\UBC$ we have canonical morphisms of exact sequences
\begin{equation}\label{fqewfqwefqwd} 
\xymatrix{0\ar[r]&C_{0,u}(X) \ar[d]\ar[r]&C_{u}(X)\cap  \ell^{\infty}_{\cB}(X)\ar[r]\ar[d]& \frac{ C_{u}(X)\cap  \ell^{\infty}_{\cB}(X)}{C_{0,u}(X) }\ar[r]\ar[d]^{!!} &0\\
0\ar[r]&C_{0}(X)\ar[d]\ar[r]&C(X)\cap  \ell^{\infty}_{\cB}(X)\ar[r]\ar[d]& \frac{ C(X)\cap  \ell^{\infty}_{\cB}(X)}{C_{0}(X)}\ar[r]\ar[d]^{!}&0\\
0\ar[r]&\ell^{\infty}(\cB)\ar[r]& \ell^{\infty}_{\cB}(X)\ar[r]&C(\partial_{h}X)\ar[r]&0}
 \end{equation} It has been shown in \cite[Rem 3.5]{Bunke:2024aa} that the   vertical map marked by $!$ is an isomorphism
provided $X$ in $G\UBC$  has a  paracompact underlying topological space.

\begin{kor}\label{pkhoprtorptgiporetgertgetrg}
If $X$ in $G\UBC$  has a  paracompact underlying topological space, then
we have a compactification $\overline{\iota^{\topp}X}$  of $\iota^{\topp}X$ by the coarse Higson corona
$\partial_{h}  X$.
\end{kor}
Here by definition $ \overline{\iota^{\topp}X}$ is the Gelfand dual of $C(X)\cap  \ell^{\infty}_{\cB}(X)$.
The  vertical map in \eqref{fqewfqwefqwd} marked by $!!$ is an  isomorphism only under  rather restrictive conditions, see   \cite[Prop. 3.22]{Bunke:2024aa}.

 We have a fully faithful functor
 \begin{equation}\label{rfwerferf24}\beta:G\LCH^{\propp}\to G\BT
\end{equation}
 which equips a locally compact Hausdorff space with the bornology of relatively compact subsets.
%
%
%

\subsection{Equivariant  topological  coarse $K$-homology}

In this section we recall the construction of the equivariant topological coarse $K$-homology functor. 
We further compare two versions of its Borelification.

Let $G$ be a group. We let 
 $\nCcat$ denote the category of possibly non-unital $C^{*}$-categories  and consider  the category $$G\nCcat:=\Fun(BG,\nCcat)$$ of  possibly non-unital $C^{*}$-categories with strict $G$-action and equivariant functors. We refer to \cite{crosscat} for details.

 We fix $\bC$ in $G\nCcat$ and assume that the underlying $C^{*}$-category is  
effectively additive and admits countable AV-sums \cite{cank}.  
Following   \cite{coarsek}  we can then  consider the functor
$$\bV^{G}_{\bC}:G\BC\to \Ccat$$
which associates to $X$ in $G\BC$ the $C^{*}$-category of equivariant locally finite $X$-controlled objects in $\bC$. 
\begin{rem}\label{ttijhoerthertertg} 
At various places in the present paper  we need some details of the 
 description of the functor $\bV^{G}_{\bC}$. Let $\bM\bC$ denote the multiplier category of $\bC$ (we refer to \cite{cank} for details)  and consider
 $X$  in $G\BC$. The objects of $\bV^{G}_{\bC}(X)$ are triples
$(C,\rho,\nu)$, where $C$ is an object of $\bC$, $\rho=(\rho_{g})_{g\in G}$ is a cocycle of unitaries 
$\rho_{g}:C\to gC$ in $\bM\bC$, and $\nu$ is an equivariant  finitely additive   multiplier projection-valued measure on $X$.
Thereby we require  that  for every $x$  in $X$ the projection  $\nu(\{x\})$   belongs to $\bC$ and has an image, that $\id_{C}=\sum_{x\in X}\nu(\{x\})$ in the strict topology on $\bM\bC$,
and that for every bounded subset $B$ of $X$ the set
$\{x\in B\mid\nu(\{x\})\not=0\}$ is finite.

A controlled morphism $A:(C,\rho,\nu)\to (C',\rho',\nu')$ is an equivariant  morphism $A:C\to C'$ in $\bM\bC$ 
with the property that there exists a coarse entourage $U$ of $X$ such that
$\nu(\{x'\})A\nu(\{x\})=0$ for all $(x,x')$ in $(X\times X)\setminus U$. 
The composition and the involution of controlled morphisms are inherited from $\bM\bC$.

 The objects described above and controlled morphisms form the  unital $\C$-linear $*$-category
$\bV^{G,\ctr}_{\bC}(X)$ which happens to be additive. The  norm on the morphism spaces of $\bM\bC$ induces a 
 norm on the morphism spaces of $\bV^{G,\ctr}_{\bC}(X)$ and we define the unital $C^{*}$-category 
  $ \bV^{G}_{\bC}(X)$ as the completion.

If $f:X\to X'$ is a morphism in $G\BC$, then we define functors
$f_{*}:\bV^{G,\ctr}_{\bC}(X)\to \bV^{G,\ctr}_{\bC}(X')$ and
$f_{*}:\bV^{G}_{\bC}(X)\to \bV^{G}_{\bC}(X')$
 on objects and morphisms by
 $$f_{*}(C,\rho,\nu):=(C,\rho,f_{*}\nu)\ , \quad f_{*}A:=A\ .$$  
For justifications and more details we refer to \cite{coarsek}.
We thus have functors
\begin{equation}\label{gweoihjgoiwerfwerfrwef}\bV^{G,\ctr}_{\bC}:G\BC\to \Add_{\C}\ , \qquad \bV^{G}_{\bC}:G\BC\to \Ccat\ .
\end{equation}
where in the first case we forget the $*$-structure and consider the values of  $\bV^{G,\ctr}_{\bC}$ as a $\C$-linear  additive categories.
\hB 
\end{rem}

We next recall the $E$ and $K$-theory functors for $C^{*}$-categories. By  \cite{joachimcat} we have an adjunction
$$A^{f}:\nCcat\leftrightarrows \nCalg:\incl$$
where $\incl$ views a $C^{*}$-algebra as a $C^{*}$-category with a single object.

 By  \cite{budu} we have a universal
homotopy-invariant, $K$-stable, exact, and filtered colimit-preserving functor
\begin{equation}\label{bswgfergrwgrerwe}\ee:\nCalg\to \EE
\end{equation}
to a stable $\infty$-category which turns out to be presentable.
 If we equip $\nCalg$ with the symmetric monoidal structure given by the maximal tensor, then
$\EE$ has a unique presentably symmetric monoidal structure such that $\ee$ admits a unique symmetric monoidal refinement.
The functor  in \eqref{bswgfergrwgrerwe} is an $\infty$-categorical version of the classical $E$-theory functor
of \cite{MR1068250}, \cite{zbMATH04182148}.

We have an adjunction \begin{equation}\label{gwergwerferfrew} b: \Mod(KU) \leftrightarrows \EE:K \end{equation}
where $K:=\map_{\EE}(\beins_{\EE},-)
$ is the $K$-theory functor, and
the essential image   of the fully faithful functor $b$ is called the bootstrap class.

Let $$H:\nCcat\to \cC$$ be some functor.
Following \cite{cank} we adopt the following definition.
\begin{ddd} The functor $H $
is called a  finitary homological functor 
 if $\cC$ is a cocomplete stable $\infty$-category and $H$ 
 sends unitary equivalences to equivalences, 
exact sequences to fibre sequences, and preserves filtered colimits.\end{ddd}

\begin{lem}\label{kopherthetrg}
The composition  
 \begin{equation}\label{fewfqwfdqwedqwdwdweq} \ee^{\nCcat}:\nCcat\xrightarrow{A^{f}} \nCalg\xrightarrow{\ee}\EE
\end{equation}  is a finitary homological  functor.
  \end{lem}
  \begin{proof}
  In order to cite results from \cite{KKG} we use the
  factorization
  $$\xymatrix{&\nCalg\ar[dl]_{\kk}\ar[dr]^{\ee}&\\ \KK\ar[rr]^{c}&&\EE}$$
  where $\kk:\nCalg\to \KK$ is the $\infty$-categorical version of Kasparov's $KK$ theory \cite{kasparovinvent}
  constructed in \cite{KKG} and $c$ is the canonical comparsion map induced by the universal property of $\kk$.
  We  therefore have an equivalence $\ee^{\nCcat}\simeq c\circ \kk^{\nCcat}$,
  where $\kk^{\nCcat}:=\kk\circ A^{f}$.
 By \cite[Thm. 1.32.2]{KKG} the functor $\kk^{\nCcat}$ sends unitary equivalences  to equivalences.
 Consequently $\ee^{\nCcat}$ does so.
 
 We have a functor $A:\nCcat_{i}\to \nCalg$ (see \cite[Sec. 6]{crosscat}), where
 $\nCcat_{i}$ is the wide subcategory of $\nCcat$ of functors which are injective on objects. We further have a natural transformation $$A^{f}_{|\nCcat_{i}}\to A:\nCcat_{i}\to \nCalg\ .$$ By
  \cite[Prop. 6.9]{KKG} the morphism $\kk(A^{f}(\bC))\to \kk(A(\bC))$ is an equivalence for every $C^{*}$-category $\bC$.  This implies that $\ee(A^{f}(\bC))\to \ee(A(\bC))$ is an equivalence, too.
  
   Since $A$ preserves exact sequences by  \cite[Prop. 8.9.2]{crosscat} and $\ee$ is exact by design we can conclude that
 $\ee\circ A$ sends exact sequences to fibre sequences. 
 Hence so does $\ee^{\nCcat}$.
 
 Finally, $\ee^{\nCcat}$ preserves filtered colimits, since $A^{f}$, being    a left-adjoint, preserves all colimits and
 $\ee$ preserves filtered colimits by design.
   \end{proof}

Recall the \cref{okprhertgrtge9} of an equivariant coarse homology theory.
  The following is a consequence of  \cref{kopherthetrg} and \cite[Thm. 7.3]{coarsek}.
\begin{kor} \label{i9gopwegergffrefwre}The composition
\begin{equation}\label{ohkeroptgrtegertgerg}\ee\cX^{G}_{\bC}:G\BC\stackrel{\bV^{G}_{\bC}}{\to} \nCcat \stackrel{\ee^{\nCcat}}{\to} \EE
\end{equation}
is  an $\EE$-valued equivariant coarse homology theory.\end{kor}
By   \cite[Thm. 7.4]{coarsek} it is in addition  strong.

We write $\EE(-,-):=\map_{\EE}(-,-)$ and
omit the symbols $\ee$ or $\ee^{\nCcat}$ if we insert $C^{*}$-algebras or $C^{*}$-categories.

Let $A$ be in $\EE$.
\begin{ddd}
The functor  \begin{equation}\label{gerwferfrwefw}K_{A}:=\EE(\C, -\otimes A):\EE\to \Mod(KU)
\end{equation} is called the $K$-theory with coefficients in $A$.
\end{ddd}

Since $\ee(\C)$ is a compact object in $\EE$ the functor $K_{A}$
 preserves colimits.
 The following definition is therefore justified by \cref{oigpwerferfwfrfrfr}.
   
\begin{ddd}\label{gu90erwfwef} We define
  the equivariant topological coarse $K$-homology with coefficients in $(\bC,A)$ by
$$
K\cX^{G}_{\bC,A}:G\BC\stackrel{\ee\cX^{G}_{\bC}}{\to} \EE \stackrel{K_{A}}{\to} \Mod(KU)\ .$$
\end{ddd}

\begin{rem}
If $A=\beins_{\EE}$, then we omit the subscript and write
$K\cX^{G}_{\bC}:=K\cX^{G}_{\bC,\beins_{\EE}}$ for the equivariant topological coarse $K$-homology with coefficients in $\bC$ from
 \cite{coarsek}. 
 If $\bC=\Hilb_{c}(\C)$ with the trivial action, then we further omit the subscript $\bC$ and call
 $$K\cX^{G}:=K\cX^{G}_{\Hilb_{c}(\C)}:G\BC\to \Mod(KU)$$ the equivariant topological coarse $K$-homology.
 Finally, if $G$ is the trivial group, then we omit the symbol $G$ from the notation and call
 \begin{equation}\label{hrtgertggbgdfbdgf}K\cX:=K\cX^{\{e\}}:\BC\to \Mod(KU)
\end{equation}  the topological coarse $K$-homology.

 For a comparison of $K\cX$ (or $K\cX^{G}$) with classical constructions as $K$-theory of (equivariant) Roe
 algebras associated to proper metric spaces (with $G$-action) we refer to \cite[Sec. 8.7]{buen}, \cite[Sec. 7]{Bunke:2017aa}.  \hB\end{rem}
 
 \begin{rem}\label{gkojpwregwerwef}
 The additional parameter $A$ in $\EE$ can be used to introduce coefficients.
 For example we could consider the object 
 $\beins_{\EE}\otimes H\Q:=b(KU\wedge H\Q)$
 with $b$ as in \eqref{gwergwerferfrew}. 
 The map $KU\to KU\wedge H\Q$ in $\Mod(KU)$ induces a morphism $\beins_{\EE}\to \beins_{\EE}\otimes H\Q$ in $\EE$ and 
 a rationalization map
 $$K\cX^{G}_{\bC}\to K\cX^{G}_{\bC,\Q}\ .$$
 But note that 
as long as $A$ belongs to the bootstrap class (the essential image of $b$), or, more generally, to the Künneth class in $\EE$, we have an equivalence
$$K\cX^{G}_{\bC,A}\simeq K\cX^{G}_{\bC}\otimes_{KU} K(A)$$
so that the additional parameter does not   really give new information.  
This is in contrast to the case of analytic $K$-homology considered in \cref{gwjiopwergergw}, see \cref{tkopwerggregwr}.

At the moment we do not have an interesting application for the case where $A$ does not belong to the  Künneth class. \hB
 \end{rem}

 We now develop the facts envisaged in \cref{biojrgogbfgbdb}.
  For the moment (in contrast to the convention adopted in \eqref{hrtgertggbgdfbdgf}) let $\ee\cX_{\bC}$ and  $K\cX_{\bC,A}$ denote the non-equivariant coarse homology theories associated to $\ee\cX^{G}_{ \bC}$ and $K\cX^{G}_{ \bC,A}$ according \cref{khopertgertgrtgertge}.  Further we let $\ee\cX^{\{e\}}_{\bC}$ and 
 $K\cX^{\{e\}}_{\bC,A}$ denote the specializations of \eqref{ohkeroptgrtegertgerg} and  \cref{gu90erwfwef} to the case of the trivial group. These functors carry $G$-actions induced by the $G$-action on $\bC$.
 
  In the following we show that these two versions of non-equivariant coarse homology theories are equivalent.
 
  Let $\Ccat_{2,1}$ denote the $2$-category of unital $C^{*}$-categories, functors and unitary equivalences. The functor  $\ell:\Ccat\to \Ccat_{2,1}$ presents the Dwyer-Kan localization of $\Ccat$ at the unitary equivalences \cite{startcats}.
 If $F:\Ccat\to \cC$ is a functor which sends unitary equivalences to equivalences, then
  it has a canonical factorization \begin{equation}\label{gwerfwerfrfdsvdf}\xymatrix{\Ccat\ar[dr]_{\ell}\ar[rr]^{F}&&\cC\\&\Ccat_{2,1}\ar@{..>}[ur]_{F_{2,1}}&}\ . 
\end{equation}

 \begin{prop}  \label{kohperthtregrtgertgbtertgrtgrertgertge}
 If $\bC$ admits all $AV$-sums, then we have canonical equivalences 
 \begin{equation}\label{vsfdvfdvfdvsdfcwervdsfvfdvfdvsdfvs}\ee\cX_{\bC}\simeq \ee\cX_{\bC}^{\{e\}}:BG\times \BC\to \EE 
\end{equation} and  \begin{equation}\label{vsfdvfdvfdvsdfcwervdsfvfdvfdvsdfvse}K\cX_{\bC,A}\simeq K\cX^{\{e\}}_{\bC,A}:BG\times \BC\to \Mod(KU)\ .\end{equation}
 \end{prop}
 \begin{proof}
 Let $F$ be one of $\ee^{\nCcat}_{|\Ccat}$ or $K_{A}\circ \ee^{\nCcat}_{|\Ccat}$.
 These functors send unitary equivalences to equivalences and therefore have 
 factorizations $ F_{2,1}$ as in \eqref{gwerfwerfrfdsvdf}, and the functors in questions can be written in the form
 $$F\circ \bV_{\bC}\simeq  F_{2,1}\circ \ell \bV_{\bC} \ , \quad F\circ \bV_{ \bC}^{\{e\}}\simeq      F_{2,1}\circ \ell \bV^{\{e\}}_{ \bC} \ .  $$
 Here $\bV_{ \bC}^{\{e\}}$ is the functor from \eqref{gweoihjgoiwerfwerfrwef} for the trivial group and the $G$-action induced from the action on $\bC$, and in analogy to \cref{khopertgertgrtgertge}
 \begin{equation}\label{}
  \bV_{\bC}:=\bV_{\bC}^{G}(X\otimes G_{min,min})
\end{equation}
with the $G$-action induced by the right action of $G$ on $G_{min,min}$.
  
 The asserted equivalences come from an  equivalence
 \begin{equation}\label{gwioeuorgfewrgwergre}\ell\bV_{\bC}\stackrel{\simeq}{\to} \ell\bV^{\{e\}}_{\bC}:BG\times \BC\to  \Ccat_{2,1}
\end{equation} which we now describe.
 For $X$ in $\BC$ we first  define a functor
\begin{equation}\label{fpoewrkfpowerferwfew}\phi_{X}:\bV_{\bC}(X)\to \bV^{\{e\}}_{ \bC}(X)\ .
\end{equation}  
 For every object $(C,\rho,\nu)$ of $\bV_{\bC}(X)=\bV_{\bC}^{G}(X\otimes G_{min,min})$ we choose an image $C_{e}\to C$
 of  the projection $\nu(X\times \{e\})$. It exists by the assumption that $\bC$ is effectively additive.
  We write $\nu_{e}$ for the restriction of $\nu$ to $C_{e}$ and $X\times \{e\}$. 
We then let $\phi_{X}$ in  \eqref{fpoewrkfpowerferwfew} send $(C,\rho,\mu)$  to  $(C_{e},\nu_{e})$ in $\bV^{\{e\}}_{\bC}(X)$.  
Furthermore, for
$A:(C,\rho,\nu)\to (C',\rho',\nu')$ we let
$A_{e}:C_{e}\to  C\xrightarrow{A} C'\to  C_{e}'$ be the restriction of $A$, where
the last map in this composition is the adjoint of the isometry $C_{e}'\to C'$.
 This is the image of $A$ under the functor $\phi_{X}$ in \eqref{fpoewrkfpowerferwfew}.
This finishes the description  of  $\phi_{X}$ from \eqref{fpoewrkfpowerferwfew}. In order to show that  it is  an equivalence
one can construct an inverse which  sends an object $(C,\nu)$ in $\bV^{\{e\}}_{\bC}(X)$ to an object $(\bigoplus_{g\in G}gC,\rho,\tilde \nu)$ in $\bV_{\bC}(X)=\bV_{\bC}^{G}(X\otimes G_{min,min})$, where
$\rho=(\rho_{g})_{g\in G}$ with $\rho_{g}:\bigoplus_{h\in G}hC\to g\bigoplus_{h\in G}hC$
sends the summand with index $h$ canonically to the summand with index $g^{-1}h$ of $g \bigoplus_{h\in G}hC$.
The measure $\tilde \nu$ is given by $\tilde \nu(\{(x,g)\}):=g\nu(\{x\})$ on the summand with index $g$.
It furthermore sends $A:(C,\nu)\to (C',\nu')$ to $\bigoplus_{g\in G} gA:\bigoplus_{g\in G}gC\to \bigoplus_{g\in G}gC'$.
Note that this construction requires the existence of AV-sums for index sets   of the cardinality of $G$.

Since taking the images of the projections  $\nu(X\times \{e\})$ involves a choice we can only extend the construction above to produce a natural transformation of functors \begin{equation}\label{gpoewrwkpferfwerfwregwerwer42}\bV_{\bC}\to \bV^{\{e\}}_{\bC}:\BC\to \Ccat_{2,1}\ .
\end{equation}
To this end, if $f:X\to X'$ is a morphism, $(C,\rho,\nu)$ is an object of $\bV_{\bC}(X)$
and $C_{e}\to C$ and $C_{e}'\to C$ are the images chosen for
$\nu(X\times \{e\}$ and $(f_{*}\nu)(X\times \{e\})$, then
we have a canonical unitary isomorphism $\sigma_{f,(C,\rho,\nu)}: C_{e}\to C_{e}'$.
The family  of unitaries $\sigma_{f}:=(\sigma_{f,(C,\rho,\nu)})_{(C,\rho,\nu)\in \bV^{G}_{\bC}(X)}$ is a natural  unitary isomorphism  
$$\sigma_{f}:f_{*}\circ \phi_{X}\stackrel{ \cong}{\to}  \phi_{X'}\circ f_{*}\ .$$
The pair $((\phi_{X})_{X},(\sigma_{f})_{f})$
is a natural transformation  of functors \eqref{gpoewrwkpferfwerfwregwerwer42}.

We finally identify the $G$-actions. 
The composition $\phi_{X}\circ g$ would send
$(C,\rho,\nu)$ to $(C_{g},\nu_{g})$, where
$C_{g}$ is the choice of an image of $\nu(X\times \{g\})$ and $\nu_{g}$ is the corresponding restriction. 
The composition $g\circ \phi_{X}$ sends
$(C,\rho,\nu)$  to $(gC_{e},g\nu_{e})$.
The maps
$$\kappa_{g,(C,\rho,\nu)}:C_{g}\to C\stackrel{\rho_{g}}{\to} gC\to gC_{e}$$ provide an isomorphism
$$\kappa_{g}:=(\kappa_{g,(C,\rho,\nu)})_{(C,\rho,\nu)\in \bV^{G}_{\bC}}(X):\phi_{X}\circ g\to g\circ \phi_{X}\ .$$
The family $\kappa:=(\kappa_{g})_{g\in G}$ satisfies a cycle condition and the pair 
$(\phi_{X},\kappa):\bV_{\bC}(X)\to \bV^{\{e\}}_{ \bC}(X)$
is a refinement of \eqref{fpoewrkfpowerferwfew} to a $G$-equivariant functor.
 For a morphism $(g,f)$ in $BG\times \BC$ the following square
$$\xymatrix{C_{g}\ar[rr]^{\sigma_{f,(C,\rho,\nu))}}\ar[d]^{\kappa_{g,(C,\rho,\nu)}} &&C'_{g} \ar[d]^{\kappa_{g,(C,\rho,f_{*}\nu)}} \\ gC_{e}\ar[rr]^{g\sigma_{f,(C,\rho,g_{*}\nu))}} &&gC_{e}' } $$
commutes which expresses the compatibility of $\sigma$ and $\kappa$. Alltogether this describes the transformation \eqref{gwioeuorgfewrgwergre}.
\end{proof}

  By \cite[Thm. 10.4]{coarsek}, if $\bC$ admits all $AV$-sums, then  the functor $\bV^{G}_{\bC}$ from \eqref{gweoihjgoiwerfwerfrwef} has an extension \begin{equation}\label{hjfhiquhqwuiehfew24}
  \xymatrix{G\BC\ar[r]^{\bV_{\bC}^{G}}\ar[dr]&\Ccat\ar[r]^{\ell}&\Ccat_{2,1}\\& G\BC_{\tr}\ar@{..>}[ur]_{\bV_{\bC,\tr}^{G}}&}   \end{equation}to the $(2,1)$-category of $G$-bornological coarse spaces with transfers.
 Postcomposing with the functors $(e^{\nCcat}_{|\Ccat})_{2,1}$
or
$K_{A}\circ  (e^{\nCcat}_{|\Ccat})_{2,1}$ we get extensions
 $$\xymatrix{G\BC\ar[rr]^{\ee\cX_{\bC}^{G}}\ar[dr]& &\EE\\& G\BC_{\tr}\ar@{..>}[ur]_{\ee\cX^{G}_{\bC,\tr}}&}\ , \quad \xymatrix{G\BC\ar[rr]^{K\cX_{\bC,A}^{G}}\ar[dr]& &\Mod(KU)\\& G\BC_{\tr}\ar@{..>}[ur]_{K\cX^{G}_{\bC,A,\tr}}&}\ .$$ 
 This allows to construct
 the Borelification maps
 \begin{equation}\label{erfwerfrfwreefr2334frewfr}\beta:\ee\cX^{G}_{\bC}\to \ee\cX^{hG}_{\bC} \ , \quad \beta:K\cX^{G}_{\bC,A}\to K\cX^{hG}_{\bC,A}  
\end{equation} according to 
 \cref{jgoijowerfewrfewfwfwef}.
 
 In the following we construct versions
 \begin{equation}\label{vsdfvsdfvr3vsfvdfvsdfv}  \beta':\ee\cX^{G}_{\bC}\to \ee\cX^{\{e\},hG}_{ \bC} \ , \quad \beta':K\cX^{G}_{\bC,A}\to K\cX^{\{e\},hG}_{ \bC,A} \end{equation}  of the Borelification maps and argue that they are equivalent to the maps in \eqref{erfwerfrfwreefr2334frewfr} via the equivalences from \cref{kohperthtregrtgertgbtertgrtgrertgertge}. This is relevant since we will show later that the analytic transgression is compatible with the Borelification using the version $\beta'$, see \cref{ijhgiowwerferfrewfrwefwf}.

 The maps $\beta'$ are induced by a natural transformation \begin{equation}\label{fqewjhbhjf788f714fjhrqf}
\beta'':\ell \bV_{\bC}^{G}\to  \lim_{BG}\ell \bV^{\{e\}}_{ \bC}:G\BC\to \Ccat_{2,1}\ .
 \end{equation}
   We then apply $F_{1,2}$ and post-compose with the canonical map  $F_{1,2}\lim_{BG}\to \lim_{BG}F_{1,2}$, where  $F$  is  one of $\ee^{\nCcat}_{|\Ccat}$ or $K_{A}\circ \ee^{\nCcat}_{|\Ccat}$, and lift against a map \eqref{gwezghuowerfwerf}.

In order to construct \eqref{fqewjhbhjf788f714fjhrqf} note
that objects of
$\lim_{BG}\ell \bV^{\{e\}}_{\bC}(X) $ are triples
$(C,\rho,\nu)$ where $\rho=(\rho_{g})_{g\in G}$ is a cocyle of unitary isomorphism  $\rho_{g}:(C,\nu)\to (gC,g_{*}\nu)$ in
$\bV^{\{e\}}_{\bC}(X)$, and the morphisms $A:(C,\rho,\nu)\to (C,\rho,\nu)$  of
$\lim_{BG}\ell \bV^{\{e\}}_{\bC}(X)$ are
 morphisms $A:(C,\nu)\to (C',\nu)$ in $\bV^{\{e\}}_{\bC}(X)$ such that $gA\rho_{g}=\rho'_{g}A
$ for all $g$ in $G$. This is almost the description of the $C^{*}$-category $\bV_{\bC}^{G}(X)$ given in \cref{ttijhoerthertertg}. The only difference is that the morphisms
$\rho_{g}$ for objects in $\bV_{\bC}^{G}(X)$ are controlled by $\diag(X)$, while
here they are just approximable by controlled ones.
 We therefore have a natural  fully faithful embedding inducing \eqref{fqewjhbhjf788f714fjhrqf}.
 Note that the adjoint of $\beta''$ is described explicitly below in  \eqref{vasfcaasdcasdadscq}.

  \begin{rem} Note that in \cite{startcats} we introduced marked $C^{*}$-categories in order to write
$\bV^{G}_{\bC}(X)$ as a limit of $\bV_{\bC}^{\{e\}}(X)$. To this end we mark the $\diag(X)$-controlled unitaries in the latter.
\hB \end{rem}

 \begin{lem} We assume that $\bC$ admits all $AV$-sums. Then the following triangle of natural transformations of functors $G\BC\to \Fun(BG,\Ccat_{2,1})$ commutes:
 \begin{equation}\label{}
 \xymatrix{&\underline{\ell\bV^{G}_{\bC}}\ar[dl]_{(\tr,\rho), \eqref{gweoiroifjoewfwerfwerferwferfrffw}}\ar[dr]^{ \eqref{vasfcaasdcasdadscq}} &\\    \ell \bV_{\bC}\ar[rr]_{\eqref{gwioeuorgfewrgwergre}}&&\ell \bV_{\bC}^{\{e\}} }
\end{equation}
 \end{lem}
 \begin{proof}
 The left-down map in the triangle is induced by applying $\ell\bV_{\bC}^{G}$ to the map  \eqref{gweoiroifjoewfwerfwerferwferfrffw}. The right-down map is the adjoint of \eqref{fqewjhbhjf788f714fjhrqf}.
  This is a direct $2$-categorical calculation unfolding all definitions. 
  \end{proof}

We apply the functors   $\ee^{\nCcat}_{|\Ccat}$ or $K_{A}\circ \ee^{\nCcat}_{|\Ccat}$, go over to adjoints,
and form the lifts under \eqref{gwezghuowerfwerf} using the $u$-continuity of the upper corners in order to get:

\begin{kor}\label{iguweorigwergwergw9}We assume that $\bC$ admits all $AV$-sums.
Then the following triangles commute:
 \begin{equation}\label{gtgetrgg345gertgrge}
 \xymatrix{& \ee\cX^{G}_{\bC} \ar[dl]_{\beta}\ar[dr]^{\beta'}&\\    \ee\cX^{hG}_{\bC}\ar[rr]^{\eqref{vsfdvfdvfdvsdfcwervdsfvfdvfdvsdfvs}^{hG}} && \ee\cX^{\{e\},hG}_{\bC} }\ , \quad \xymatrix{&  K\cX^{G}_{\bC,A}\ar[dl]_{\beta}  \ar[dr]^{\beta'}&\\  K\cX^{hG}_{\bC,A}\ar[rr]^{\eqref{vsfdvfdvfdvsdfcwervdsfvfdvfdvsdfvse}^{hG}}&&K\cX_{\bC,A}^{\{e\},hG}  }\ .
\end{equation}
\end{kor}


%
%
%

 \subsection{Equivariant analytic $K$-homology}\label{gwjiopwergergw}

 In this section we recall the  construction of the equivariant analytic $K$-homology. 
 We furthermore introduce its  Borel-equivariant version and the Borelification map.

%

Generalizing \eqref{bswgfergrwgrerwe} to the equivariant case
we let \begin{equation}\label{hertgrtgtergbgd}\ee^{G}:G\nCalg\to \EE^{G}
\end{equation} be the  universal homotopy invariant, $K_{G}$-stable, exact and filtered colimit  preserving    functor to a cocomplete stable $\infty$-category constructed in \cite{budu}, called the equivariant $E$-theory functor   
Again, $\EE^{G}$ is presentable, and if we equip $G\nCalg$ with the symmetric monoidal structure given by the maximal tensor, then
$\EE^{G}$ has a unique presentably symmetric monoidal structure such that $\ee^{G}$ admits a unique symmetric monoidal refinement. 
The functor in \eqref{hertgrtgtergbgd} is an $\infty$-categorical refinement of the functor introduced in \cite{Guentner_2000}.

We consider the equivariant generalization   \begin{equation}\label{erfwrefrfwrfwrfreefrefwerferfw}\ee^{G\nCcat}:G\nCcat\xrightarrow{A^{f}} G\nAlg \xrightarrow{\ee^{G}}\EE^{G}
\end{equation}
of \eqref{fewfqwfdqwedqwdwdweq} which again sends unitary equivalences to equivalences, is exact and  preserves filtered colimits by the equivariant versions of the  arguments given in the proof of \cref{kopherthetrg}. 
As these properties are not needed in the present paper we will not discuss the details further.

We let $\bC^{(G)}_{\std}$ in $G\nCcat$ denote the $G$-$C^{*}$-category constructed from $\bC$ in 
 \cite[Def. 2.15.(2)]{bel-paschke}.
 Its objects are pairs $(C,\rho)$ which extend to an object $(C,\rho,\nu)$  in $\bV_{\bC}^{G}(X)$  for some 
 $X$ in $G\BC$ with a free $G$-action. The morphisms
 $A:(C,\rho)\to (C',\rho')$ are morphisms $A:C\to C'$
 in $\bC$, and the involution and composition is inherited from $\bC$.
 Finally, the $G$-action fixes the objects, and $g$ in $G$ acts on $A$ by
 $A\mapsto g\cdot A:= \rho_{g}^{\prime,-1} \circ gA \circ \rho_{g}$.
 
 \begin{rem} \label{kophertgretgertgertg}In the classical approach to  Kasparov's $KK$-theory \cite{kasparovinvent}, for
  a $G$-$C^{*}$-algebra $A$ one considers the standard $A$-Hilbert module $L^{2}(G,A)$ with $G$-action
 and its $G$-$C^{*}$-algebra  $K( L^{2}(G,A))$ of compact operators.
 It fits into an exact sequence
 \begin{equation}\label{gerwferfwerfwre}0\to K( L^{2}(G,A))\to B( L^{2}(G,A)) \to Q( L^{2}(G,A))\to 0\ ,
\end{equation}  where $B( L^{2}(G,A))\cong M K( L^{2}(G,A))$ is the
 algebra of bounded adjointable
 operators on the standard $A$-Hilbert module and also the multiplier algebra of $K(L^{2}(G,A))$,  and $Q( L^{2}(G,A))$ is an equivariant and Hilbert $A$-module version of the classical Calkin algebra.

 The category $\bC^{(G)}_{\std}$ can be considered as the analogue the $G$-$C^{*}$-algebra $K( L^{2}(G,A))$
 in the realm of $G$-$C^{*}$-categories. In analogy  to \eqref{gerwferfwerfwre} it fits into an exact sequence
 \begin{equation}\label{hrtge3g43gtrgegr}0\to \bC^{(G)}_{\std}\to \bM\bC^{(G)}_{\std} \to \bQ^{(G)}_{\std} \to 0
\end{equation}
 of $G$-$C^{*}$-categories, see \cite[Def. 2.15.]{bel-paschke}, where $\bM\bC^{(G)}_{\std}$ is the multiplier  category of $\bC^{(G)}_{\std}$.
 \hB
 \end{rem}

 We write $\EE^{G}(-,-):=\map_{\EE^{G}}(-,-)$ and omit the symbols $\ee^{G}$ or $\ee^{G\nCcat}$
 if we insert $G$-$C^{*}$-algebras or categories. Implicitly we we also use the functor $\Res_{G}:\EE\to \EE^{G}$ in order to consider objects of $\EE$ as objects of $\EE^{G}$.

 In the following we consider the functor  $$C_{0}:G\LCH^{+}\stackrel{\simeq}{\to} (G\nCalg_{\comm})^{\op}\to (G\nCalg)^{\op}\ .$$
 Recall that $A$ belongs to $\EE$.
\begin{ddd}\label{kophertgrgrgerg}
The equivariant   analytic $K$-homology with coefficients in
$(\bC,A)$ is defined  by
$$K^{G,\an}_{\bC,A} :=\EE^{G}(C_{0}(-),\bC^{(G)}_{\std}\otimes A ):G\LCH^{+}\to \Mod(KU)\ .$$
\end{ddd}
 
In \cref{ogpwerferfewrfwef} we will show that the  equivariant   analytic $K$-homology with coefficients in
$(\bC,A)$  is an equivariant Borel-Moore homology theory.

\begin{rem}\label{tkopwerggregwr}
Since the functor $\ee^{G}$ is symmetric monoidal we have a natural transformation
$$K^{G,\an}_{\bC}\otimes_{KU}K(A)\to  K^{G,\an}_{\bC,A}\ ,$$
but this transformation is not an equivalence in general.
The functor $ K^{G,\an}_{\bC,A}$ sends disjoint  unions (also infinite ones) to products.
In contrast, if $K(A)$ is not dualizable in $\Mod(KU)$, then the functor $K^{G,\an}_{\bC}\otimes_{KU}K(A)$
does not have this property. So in the case of analytic $K$-homology we have, e.g.,  two  non-equivalent versions of a rationalization,
$$K^{G,\an}_{\bC}\otimes H\Q \to  K^{G,\an}_{\bC,\Q}\ .$$
We prefer to work with the second one since it is again an equivariant Borel-Moore homology   homology theory, while the first is not.
The same applies if we replace $\Q$ by $\C$.
\hB
\end{rem}
 
Recall the notation from \cref{gkojpwregwerwef}.
\begin{ddd}\label{okpgbvgrbgbdfbdfgb}
The natural transformation of functors 
\begin{equation}\label{gwerojopwerferfwerfwerfw}\rmch^{K}:K^{G,\an}_{\bC} \to K^{G,\an}_{\bC,\C}:G\LCH^{+}\to \Sp
\end{equation} induced by
$\beins_{\EE}\to \beins_{\EE}\otimes H\C$
is called the equivariant $K$-theoretic  Chern character.\end{ddd}
The $K$-theoretic Chern character will be further used in \cref{9gwerfreferwfwever} and \cref{9gwerfreferwfwever1}. 

Note that $G$ acts on the functor $K^{\an}_{\bC,A}:\LCH^{+}\to \Mod(KU)$ via its action on $\bC$.
Specializing \cref{okheprthertgertgetg} we adopt the following definition:
\begin{ddd}\label{okheprthertgegegwerfwerfwrtgetg}
We define the Borel-equivariant  analytic $K$-homology as the composition 
$$\hspace{-0.3cm}K_{\bC,A}^{\an,hG}:G\LCH^{+}\xrightarrow{K_{\bC,A}^{\an}}\Fun(BG\times BG, \Mod(KU))\xrightarrow{\diag_{BG}^{*}}   \Fun(BG, \Mod(KU))\xrightarrow{\lim_{BG}} \Mod(KU)\ .$$
\end{ddd}
It is again an equivariant Borel-Moore homology theory.

In the following  we construct  the Borelification map
\begin{equation}\label{gwerpoj0erjgeroigerwg}\beta:K^{G,\an}_{\bC,A}\to K^{\an,hG}_{\bC,A}\ .
\end{equation}\begin{construction}{\em 
Recall from the text before \eqref{gwerfwerfrfdsvdf} that $\ell:\Ccat\to \Ccat_{2,1}$ represents the Dwyer-Kan localization at the unitary equivalences. We further use the notation from  \cite[Def. 2.8]{bel-paschke}.
We first  construct a morphism 
\begin{equation}\label{fwqedwedqewdq} (\phi,\sigma): \ell \bC^{(G)}_{\std}\to  \ell \bC_{\std}
\ . \end{equation}
in $\Fun(BG,\Ccat_{2,1})$.
The functor $\phi: \bC^{(G)}_{\std}\to  \bC_{\std}$ sends
the object $(C,\rho)$ in $ \bC^{(G)}_{\std}$ to $C$ in $\bC_{\std}$, and the morphism
$A:(C,\rho)\to (C',\rho')$  in $ \bC^{(G)}_{\std}$  to $A:C\to C'$  in $\bC_{\std}$.
The cocycle $\sigma=(\sigma_{g})_{g\in G}$ consists of natural transformations
$\sigma_{g}: \phi \circ g\to g\circ \phi$.  We define the components by
$\sigma_{g,(C,\rho)}:=\rho_{g}$. Indeed,
$\phi(g(C,\rho))= C$ and $g(\phi(C,\rho))=gC$ and
$\sigma_{g,(C,\rho)}:=\rho_{g}:C\to gC$ is a unitary as  required.
In order to see naturality, consider $A:(C,\rho)\to (C',\rho')$ in $  \bC^{(G)}_{\std}$.
Then $\phi(g(A))=\rho^{\prime,-1}_{g} gA\rho_{g}$ and $g(\phi(A))=gA$. The  resulting  relation
$\sigma_{g,(C',\rho')}   \phi(g(A)) = g(\phi(A)) \sigma_{g,(C,\rho)} $
expresses the naturality of $\sigma_{g}$. The cocycle condition for the $\rho$'s induces the cocycle
condition for $\sigma$. This finishes the construction of the map \eqref{fwqedwedqewdq}.

We have a canonical functor
$$\hat \Res^{G} :\EE^{G}\to \Fun(BG,\EE)\ .$$
It is determined by the property that it sends $\ee^{G}(A)$  for $A$ in $G\nCalg$ to
$\ee( A )$ in $\Fun(BG,\EE)$  for $A$ considered just as a $C^{*}$-algebra, and with the $G$-action induced by the $G$-action on $A$ and functoriality of $\ee$.
We write $\underline{\EE}$ for the  morphism $KU$-module in $\Fun(BG,\EE)$. For $A$ and $B$ in $G\nCalg$ or $\EE^{G}$ we have
$\underline{\EE}(A,B)\simeq \lim_{BG}\EE(A,B)$, where on the right-hand side
$G$ acts by conjugation  via its actions on $A$ and $B$, and we omitted to write the application of  $\hat \Res^{G}$ to $A$ and $B$.

Note that $\ee^{G\nCcat}_{|G\Ccat}$ sends unitary equivalences to equivalences and therefore has a natural factorization 
$ (\ee^{G\nCcat}_{|G\Ccat})_{2,1}$ as in \eqref{gwerfwerfrfdsvdf}.
We can therefore safely omit the symbol $\ell$ when we insert unital  $C^{*}$-categories into $\EE(-,-)$, or unital $G$-$C^{*}$-categories into   $\EE^{G}(-,-)$.
Furthermore, as above,  it is understood implicitly that when we insert a  unital $G$-$C^{*}$-category or a $G$-$C^{*}$-algebra into $\underline{\EE}$, then we apply $\hat \Res^{G}$ beforehand.
We now define the component of the Borelification map  \eqref{gwerpoj0erjgeroigerwg} at $X$ in $G\BC$ as the composition
\begin{eqnarray*}
K^{G,\an}_{\bC,A}(X)&\stackrel{def}{\simeq}  &
\EE^{G}(C_{0}(X),A\otimes \bC^{(G)}_{\std})\\&\xrightarrow{\hat \Res^{G}}&
\underline{\EE}(  C_{0}(X) ,A\otimes \bC^{(G)}_{\std} )\\&\xrightarrow{\eqref{fwqedwedqewdq}} &
\underline{\EE}( C_{0}(X) ,A\otimes   \bC_{\std})\\
&\stackrel{def}{\simeq}&K^{\an,hG}_{\bC,A}(X)\ .
\end{eqnarray*}
It is easy to see that this actually describes a natural transformation \eqref{gwerpoj0erjgeroigerwg} of functors 
from $G\LCH^{+}$ to $\Mod(KU)$.

 }
\hB
\end{construction}


\subsection{The analytic transgression}\label{kophertggertgrtger}

In this section we construct the analytic transgression, a natural transformation from the equivariant topological
coarse $K$-homology of a $G$-bornological coarse space  to the equivariant analytic  $K$-homology of its Higson corona. We further discuss its compatibility with the Borelification maps.

 We let $G_{can,max}$ in $G\BC$ be the group $G$ equipped with the canonical coarse structure and the maximal bornology.  Following   \cref{oihperthertgertgetrg}
we  consider the twist 
$$K\cX_{\bC,A,G_{can,max}}^{G}:G\BC\to \Mod(KU)$$
of the equivariant coarse topological $K$-homology from \cref{gu90erwfwef}.

We now construct the
  analytic transgression as a natural transformation
\begin{equation}\label{wegwerfrefwerfvfs}T^{G,\an}:K\cX_{\bC,A,G_{can,max}}^{G}\to \Sigma K_{\bC,A}^{G,\an}\circ \partial_{h}:G\BC\to \Mod(KU)\ .
\end{equation}
We closely  follow the construction from  \cite[Sec. 6]{bel-paschke}. We will use the details of the construction
of the  functor $\bV_{\bC}^{G}$ from \cite{coarsek} and recalled in \cref{ttijhoerthertertg}.
 We construct a multiplication morphism
\begin{equation}\label{werfwerfwerfwrbgwg}\mu_{X}:C(\partial_{h}X)\otimes \bV_{\bC}^{G}(X\otimes G_{can,max})\to \bQ_{\std}^{(G)}\ ,
\end{equation} where  $ \bQ_{\std}^{(G)}$ in $G\nCcat$ is as in \cite[Def. 2.15.(2)]{bel-paschke}, see also \cref{kophertgretgertgertg}. The twist by $G_{can,max}$ is introduced  to make the $G$-action on 
 $G$-bornological coarse spaces  free.
This freeness is used to see that the multiplication $\mu$ takes values in $\bQ^{(G)}_{\std}$.

\begin{construction}\label{ugwerigwerfrefw}{\em Here are the details of the construction of $\mu_{X}$ in  \eqref{werfwerfwerfwrbgwg}. Let $Y$ be in $G\BC$. If $f$ is in $\ell^{\infty}(Y)$, then for any object $(C,\rho,\nu)$ of $\bV^{G}_{\bC}(Y  )$ we can consider the element \begin{equation}\label{giouweiorjgoiwejfioeferwfw43}\nu(f):=\sum_{y\in Y} f(y)\nu(\{y\})  
\end{equation}(strict convergence) in $\End_{\bM\bC}(C )$.
Let $A:(C,\rho,\nu)\to (C',\rho',\nu')$ be a morphism in $\bV^{G}_{\bC}(Y  )$. Note that $A$ is in particular a morphism from $C$ to $C'$ in the multiplier category $\bM\bC$ of $\bC$.
But if  $f$ is in $\ell^{\infty}(\cB_{Y})$, then   $\nu'(f)A\in \bC$ by local finiteness of the objects.
Furthermore, if $f$ is in $\ell^{\infty}_{\cB_{Y}}(Y)$, then also 
$\nu'(f)A-A\nu(f)\in \bC$ by  \cite[Cor. 5.5]{bel-paschke}.  
We therefore get a multiplication map (see \cite[Def. 2.8]{bel-paschke} for notation)
$$C(\partial_{h}Y)\otimes \bV^{G}_{\bC}(Y  )\to \bM\bC^{(G)}/\bC^{(G)}\ , \quad (C,\rho,\nu)\mapsto (C,\rho)\ , \quad [f]\otimes A\mapsto [\nu'(f)A]\ .$$
 The domain of this multiplication map  is a maximal tensor product of $C^{*}$-categories. 
Since 
$C(\partial_{h}Y)$, considered as a $C^{*}$-category, has a single object the objects of the domain are in canonical bijection with the objects of $\bV_{\bC}^{G}(Y)$.

We have a projection $\pr:X\otimes G_{can,max}\to X$ in $G\BC$ which induces a morphism
$\pr^{*}:C(\partial_{h}X)\to C(\partial_{h}(X\otimes G_{can,max}))$.
The multiplication morphism $\mu_{X}$ in \eqref{werfwerfwerfwrbgwg} sends the object $(C,\rho,\nu)$ to the object $(C,\rho)$ in $\bQ^{(G)}_{\std}$. Furthermore it sends $[f]\otimes A:(C,\rho,\nu)\to (C',\rho',\nu')$ to the morphism $[\nu(\pr^{*}f)A]:(C,\rho)\to (C',\rho')$ in $\bQ^{(G)}_{\std}$.
As argued above this describes a well-defined functor of $C^{*}$-categories.
%
%
%
}\hB
\end{construction}

We further have a diagonal morphism
\begin{equation}\label{vwiejviowevdsfvsdfv}\hspace{-0.8cm}\delta_{\partial_{h}X}:=C(\partial_{h}X)\otimes -:\EE(\C,  \bV_{\bC}^{G}(X\otimes G_{can,max})\otimes A)\to \EE^{G}(C(\partial_{h}X ),C(\partial_{h}X)\otimes 
\bV_{\bC}^{G}(X\otimes G_{can,max})\otimes A)\ .
\end{equation} \begin{ddd} \label{rguweruigowerferfrfwr}We   define
the analytic transgression for $X$ as the composition
\begin{align*}
T^{G,\an}_{X}:K\cX^{G}_{\bC,A,G_{can,max}}(X)&\stackrel{def}{=}\EE(\C,\bV_{\bC}^{G}(X\otimes G_{can,max})\otimes A)\\&\stackrel{\delta_{\partial_{h}X}}{\to }
\EE^{G}(C(\partial_{h}X),C(\partial_{h}X)\otimes 
\bV_{\bC}^{G}(X\otimes G_{can,max})\otimes A)\\&\stackrel{\mu_{X}}{\to} \EE(C(\partial_{h}X),\bQ^{(G)}_{\std}\otimes A)\stackrel{\simeq,\partial}{\to} \Sigma
 \EE(C(\partial_{h}X),\bC^{(G)}_{\std}\otimes A)\stackrel{def}{=}\Sigma K^{G,\an
}_{\bC,A}(\partial_{h}X)\ .\end{align*}
\end{ddd}
Here $\partial$ is the boundary map associated to the sequence \eqref{hrtge3g43gtrgegr}. It
 is an equivalence since $ \bM\bC^{(G)}_{\std}$ is flasque \cite[Lem. 2.21]{bel-paschke}.
\begin{theorem}
The family $(T^{G,\an}_{X})_{X\in G\BC}$ are the components of 
a natural transformation $T^{G,\an}$ as in \eqref{wegwerfrefwerfvfs}
\end{theorem}
\begin{proof}
The argument is non-trivial, but  completely analogous to the one for the Paschke transformation given in  \cite[Sec. 7]{bel-paschke}.
\end{proof}

\begin{rem}
Using the composition in $\EE^{G}$ and the definitions $K_{G}(\partial_{h}X):=\EE^{G}(\C,C(\partial_{h}X))$ and
$K_{A}^{G}(\bC^{(G)}_{\std}):=\EE^{G}(\C,\bC^{(G)}_{\std}\otimes A)$
we can define a pairing
$$-\cap^{\cX}-:K_{G}(\partial_{h}X)\otimes  K\cX^{G}_{\bC,A,G_{can,max}}(X)\to \Sigma K^{G}_{A}(\bC^{(G)}_{\std})$$
which generalizes the coarse corona pairing from \cite[Def. 3.35]{Bunke:2024aa} to the equivariant case and coefficients  (see also \cref{ergerwfwerfwrf} for an application). \hB
\end{rem}


We now discuss the compatibility of the analytic transgression with the Borelification map.
 For $X$ in $G\BC$ let   $\pr:X\otimes G_{can,max}\to X$ denote the projection. 
\begin{prop}\label{ijhgiowwerferfrewfrwefwf}
We assume that $\bC$ admits all AV-sums.
The following diagram commutes  
\begin{equation}\label{fqwefqwedqwedqwedqwed}
\xymatrix{&\ar[dl]_{\pr\circ \eqref{erfwerfrfwreefr2334frewfr}}^(0.6){\pr\circ \beta}K\cX^{G}_{\bC,A,G_{can,max}}\ar[r]^{T^{G,\an}}\ar[d]_{\pr\circ  \eqref{vsdfvsdfvr3vsfvdfvsdfv}}^{\pr\circ \beta'}&K^{G,\an}_{\bC,A}\circ \partial_{h}  \ar[d]_{\eqref{gwerpoj0erjgeroigerwg}}^{\beta} \\K\cX^{hG}_{\bC,A }\ar[r]^{\simeq}_{\eqref{gtgetrgg345gertgrge}}& K\cX^{\{e\},hG}_{\bC,A }\ar[r]^{(T^{\an})^{hG}}&K^{\an,hG}_{\bC,A}\circ \partial_{h} }\ .
\end{equation}
\end{prop}
 \begin{proof}
   We start with constructing a natural transformation 
 \begin{equation}\label{vasfcaasdcasdadscq} (\psi,\kappa):\underline{\ell \bV^{G}_{\bC }} \to \ell \bV^{\{e\}}_{\bC}:G\BC\to \Cat_{2,1}\end{equation}
representing the adjoint of $\beta''$ in   
  \eqref{fqewjhbhjf788f714fjhrqf}.
   For $X$ in $G\BC$ the functor $\psi_{X}$ sends $(C,\rho,\nu)$ in $\bV^{G}_{\bC}(X)$ to
 $(C,\nu)$ in $\bV_{\bC}(X)$, and $A:(C,\rho,\nu)\to (C',\rho',\nu')$ in $  \bV^{G}_{\bC}(X)$ to
 $A:(C,\nu)\to (C',\nu')$ in $\bV^{\{e\}}_{\bC}(X)$. The  cocycle 
 $\kappa_{X}=(\kappa_{X,g})_{g\in G}$ of natural transformations
 $\kappa_{X,g}:\psi_{X}\to g\circ \psi_{X}$ has the components
 $\kappa_{X,g,(C,\rho,\nu)}:= \rho_{g}:(C,\nu) \to g(C,\nu)=(gC,g_{*}g\nu)$.
 The equation  $  g_{*} g\nu \circ \rho_{g}= \rho_{g} \circ \nu$ expresses the fact that  
 $\kappa_{X,g,(C,\rho,\nu)}$ is $\diag(X)$-controlled.
 The equality $\rho_{g}\circ A=gA\circ \rho_{g}$ expresses naturality of $\kappa_{X,g}$.
 Finally, the family
 $(\psi_{X},\kappa_{X})_{X\in G\BC}$ is the desired natural transformation \eqref{vasfcaasdcasdadscq}.
 
 The natural transformation \eqref{fwqedwedqewdq} induces a natural transformation
 \begin{equation}\label{fwiojowqdqwedeqweeee} \ell \bQ^{(G)}_{\std}\to \ell \bQ_{\std}\ .
  \end{equation}
  in $\Fun(BG,\Ccat_{2,1})$.
  For $X$ in $G\BC$ we consider the diagram 
 \begin{equation}\label{gwerfewrfrewfwe}\hspace{-2cm}
\xymatrix{\EE(\C,\bV^{G}_{\bC}(X\otimes G_{can,max})\otimes A)\ar[r]^-{\delta^{G}}\ar[d]^{\hat \Res^{G}\circ \Res_{G}}&\EE^{G}(C(\partial_{h}X),\bV^{G}_{\bC}(X\otimes G_{can,max})\otimes C(\partial_{h}X)\otimes A) \ar[r]^-{\mu^{G}}\ar[d]^{\hat \Res^{G}}& \EE^{G}(C(\partial_{h}X),\bQ^{(G)}_{\std}\otimes A)\ar[d]^{\hat \Res^{G}}\\ \ar[d]^{\pr\circ \eqref{vasfcaasdcasdadscq}}
\underline{\EE}(\underline{\C}, \underline{\bV^{G}_{\bC}(X\otimes G_{can,max})}\otimes A) \ar[r]^-{\widehat{\delta}}& \underline{\EE}( C(\partial_{h}X) ,\underline{\bV^{G}_{\bC}(X\otimes G_{can,max})}\otimes  C(\partial_{h}X) \otimes A)\ar[r]^-{\hat \mu} \ar[d]^{\pr\circ \eqref{vasfcaasdcasdadscq}}& \underline{\EE}(  C(\partial_{h}X)  , \bQ^{(G)}_{\std} \otimes A)\ar[d]^{ \eqref{fwiojowqdqwedeqweeee}}\\ 
\underline{\EE}(\underline{\C},  \bV^{\{e\}}_{\bC}(X )\otimes A)\ar[r]^-{ \delta}  & \underline{\EE}( C(\partial_{h}X) , \bV^{\{e\}}_{\bC}(X )\otimes  C(\partial_{h}X)\otimes A)\ar[r]^-{\mu}&\underline{\EE}( C(\partial_{h}X) , \bQ_{\std}\otimes A)}\ ,
\end{equation}
in $ \Mod(KU)$.
The map $\hat \mu$ is induced by the same functor \eqref{werfwerfwerfwrbgwg}  as $\mu^{G}$.
The upper and lower horizontal compositions reflect the definitions of the transgression maps $ T^{G,\an}$ and $(T^{\{e\},\an})^{hG}$.
The upper two and the lower left squares commute obviously.
One checks that the lower right square commutes since
\begin{equation}\label{}  \xymatrix{\underline{\ell \bV^{G}_{\bC}(X\otimes G_{ can,max})}\otimes \ C(\partial_{h}X) \ar[r]^-{\hat \mu}\ar[d]^{\pr\circ \eqref{vasfcaasdcasdadscq}} &  \ell \bQ^{(G)}_{\std}  \ar[d]^{\eqref{fwiojowqdqwedeqweeee}} \\ \ell \bV^{\{e\}}_{\bC}(X)\otimes  C(\partial_{h}X) \ar[r]^-{\mu} & \ell \bQ_{\std} }  \end{equation} 
commutes in $\Fun(BG,\Ccat_{2,1})$. We  get a commutative square \begin{equation}\label{}\xymatrix{ &\ar[dl]_{\pr\circ \beta'}K\cX^{G}_{\bC,A,G_{can,max}}\ar[r]^{T^{G,\an}}\ar[d] & K_{\bC,A}^{G,\an}\circ \partial_{h}\ar[d]_{\eqref{gwerpoj0erjgeroigerwg}}^{\beta} \\ K\cX^{\{e\},hG}_{\bC,A}\ar[r]^-{\eqref{gwezghuowerfwerf} }& \lim_{BG} K\cX^{\{e\}}_{\bC,A}\ar[r]^{!} & K_{\bC,A}^{\an,hG}\circ \partial_{h} } 
\end{equation}
of natural transformations between functors from $G\BC$ to $\Mod(KU)$, where the marked arrow is induced by 
the lower line of \eqref{gwerfewrfrewfwe}. The precise argument for naturality is again as in  \cite[Sec. 7]{bel-paschke}.
We get the right square in \eqref{fqwefqwedqwedqwedqwed}.
The left triangle follows from 
  \cref{iguweorigwergwergw9}. 
\end{proof}

%
%

\begin{rem}\label{jhtrhoeriotgertgetrhehhfj}
By specializing \eqref{zwetrwertwret} we obtain  a factorization
\begin{equation}\label{kophertggertgrtgehertgrtr}
\xymatrix{&R(K^{G,\an}_{\bC,A}\partial_{h})\ar[d]\\\ar@{..>}[ur]K\cX^{G}_{\bC,A,G_{can,max}}\ar[r]^{T^{G,\an}}&K^{G,\an}_{\bC,A}\partial_{h}} \ .\end{equation}
The dotted arrow can be considered as the best approximation of   the analytic  transgression  by  a natural transformation between
equivariant coarse homology theories. In examples one can show that $T^{G,\an}$ is non-trivial which implies that 
$R(K^{G,\an}_{\bC,A}\partial_{h})$ is not the zero functor. \hB
\end{rem}

\subsection{Equivariant algebraic coarse   $K$-homology}

 In this section we introduce the equivariant algebraic coarse $K$-homology and the comparison map to equivariant topological coarse $K$-homology. We discuss conditions ensuring that this comparison map is an equivalence.
 We furthermore introduce the Borel equivariant  algebraic coarse $K$-homology and the Borelification map using the general procedure from \cref{okhphherthertgeg9}. We argue that the comparison map is compatible with the Borelification.

For a commutative ring $k$ we let $\Cat_{k}$  denote the category of $\Mod(k)$-enriched categories.
We let $$K^{\Cat_{\Z}}H:\Cat_{\Z}\to \Sp$$ denote the
algebraic homotopy $K$-theory functor for $\Z$-linear categories described, e.g., in \cite{L_ck_2016},  \cite[Sec. 13.6]{Bunke:2025aa}.
Following \cite[Sec. 13.8]{Bunke:2025aa} its twist by the non-unital algebra of trace class operators is defined by
\begin{equation}\label{gwergfwrwfg5gf}K^{\Cat_{\Z}}H_{\cL^{1}}:=\Fib(K^{\Cat_{\Z}}H(\cL^{1,+}\otimes_{\C} -)\to K^{\Cat_{\Z}}H(  -)):\Cat_{\C}\to \Sp
\end{equation} 
where we omitted to write the forgetful functor $\Cat_{\C}\to \Cat_{\Z}$, and  the map  in \eqref{gwergfwrwfg5gf} is induced by 
the  canonical projection
$\cL^{1,+}\to \C$ from the unitalization $\cL^{1,+}$ of 
$\cL^{1}$.

By \cite[Def. 13.23]{Bunke:2025aa} we have   a natural comparison transformation
\begin{equation}\label{vfojoivdsfvfdvswre}c^{\cL^{1}}:K^{\Cat_{\Z}}H_{\cL^{1}}\to K^{\nCcat} :\Ccat\to \Sp\ ,
\end{equation}
where
$$K^{\nCcat}:\nCcat\xrightarrow{\ee^{\nCcat},\eqref{erfwrefrfwrfwrfreefrefwerferfw}} \EE\xrightarrow{K} \Mod(KU)$$
is the usual $K$-theory functor for $C^{*}$-categories. Here in the domain of \eqref{vfojoivdsfvfdvswre} we omitted the   functor which 
interprets a unital $C^{*}$-category as a $\C$-linear category, and in the target we omitted to write the restriction of the $K$-theory functor to unital $C^{*}$-categories. 
 By   \cite[Cor. 13.25]{Bunke:2025aa} the component
 $c^{\cL^{1}}_{\Hilb^{\fin}(\C)}$ of \eqref{vfojoivdsfvfdvswre} is an equivalence.
 Using that any finite-dimensional $C^{*}$-algebra  is Morita equivalent to some finite coproduct of copies of $\Hilb^{\fin}(\C)$ we conclude, using Morita invariance and additivity of the domain and target, that the 
 components 
 $c^{\cL^{1}}_{\Hilb^{\fg,\proj}(A)}$ and $c^{\cL^{1}}_{A}$  for finite-dimensional  $C^{*}$-algebras $A$
are  equivalences. Here we consider $A$ as a one-object $C^{*}$-category.

For any commutative ring $k$ let 
 $\Add_{k}\subseteq \Cat_{k}$ be the full subcategory  $k$-linear additive categories.
The following is taken from   \cite[Def. 6.2]{Bunke:2025aa}.
\begin{ddd}\label{kophtrherthetregrtg}
A functor $F:\Add_{k}\to \cC$ is homological if $\cC$ is a cocomplete stable $\infty$-category
and $F$ preserves equivalences, filtered colimits, sends Karoubi filtrations to fibre sequences,
and annihilates flasques.
\end{ddd}

By  \cite[Cor. 13.22]{Bunke:2025aa} the restriction of $K^{\Cat_{\Z}}H$ to $\Add_{\Z}$ is homological. 
This property is inherited by the twisted version
$K^{\Cat_{\Z}}H_{\cL^{1}}$ on $\Add_{\C}$.
Recall the functor $\bV^{G,\ctr}_{\bC}$ from  \eqref{gweoihjgoiwerfwerfrwef}.
\begin{ddd}\label{kopgwegwerfrfwref}
We define  the equivariant algebraic coarse $K$-homology functor $$K\cX^{G,\ctr}_{\bC }:G\BC\xrightarrow{\bV^{G,\ctr}_{\bC}} \Add_{\C}\xrightarrow{K^{\Cat_{\Z}}H_{\cL^{1}}}\Sp\ .$$
\end{ddd}

 Since $K^{\Cat_{\Z}}H_{\cL^{1}}$ is homological we conclude by   \cite[Thm. 7.2]{Bunke:2025aa} that the functor $K^{G,\ctr}_{\bC }$ is indeed an    equivariant coarse homology theory.
 \begin{ddd}\label{kopgwegwerfrfwref1}
 We define the comparison transformation
\begin{equation}\label{gtjwneiofgerfewrfwerfref}c^{G}:K\cX^{G,\ctr}_{\bC }\to K\cX^{G}_{\bC}:G\BC\to  \Sp
\end{equation}
as the composition
\begin{equation}\label{}K\cX^{G,\ctr}_{\bC }=K^{\Cat_{\Z}}H_{\cL^{1}}\circ \bV^{G,\ctr}_{\bC}\xrightarrow{!}
 K^{\Cat_{\Z}}H_{\cL^{1}}\circ \bV^{G}_{\bC} \xrightarrow{c^{\cL^{1}},\eqref{vfojoivdsfvfdvswre}}K^{\nCcat} \circ \bV^{G}_{\bC} =K\cX^{G}_{\bC} .
 \end{equation}  \end{ddd}
  
  Here the marked arrow is induced by the inclusion
  $ \bV^{G,\ctr}_{\bC }\to  \bV^{G}_{\bC}$.

 We now consider conditions ensuring that the comparison map is an equivalence. 
  Recall that we omit the subscript $\bC$    in the case that   $\bC=\Hilb_{c}(\C)$ with the trivial action.
 Let $H$ be a subgroup of $G$.  
\begin{lem}\label{hkrtopgkoprtogkrtgretgert}
 The component  \begin{equation}\label{gwergerrewferfwerf}c^{G}_{(G/H)_{min,min}}:K\cX^{G,\ctr}((G/H)_{min,min})\stackrel{\simeq}{\to} K\cX^{G}((G/H)_{min,min}) \ .
 \end{equation}
 of the comparison transformation at $(G/H)_{min,min}$ in $G\BC$ is an equivalence.
 \end{lem}
\begin{proof}
 As in \cite[Lem. 8.20]{equicoarse}      observe that  both, 
 $$\bV^{G,\ctr}((G/H)_{min,min})\ , \qquad  \bV^{G}((G/H)_{min,min})$$ are  equivalent to the  categories of finite-dimensional unitary representations of $H$, considered as $\C$-linear additive category or as $C^{*}$-category.
 They  are  equivalent to the additive completions of 
 $\coprod_{\hat H^{u}_{\fin}} \Hilb^{\fin}(\C)$ in $\Cat_{\C}$ or $\Ccat$, respectively, 
where $  \hat H^{u}_{\fin}$ denotes the set of isomorphism classes of finite-dimensional unitary representations of $G$.
Since the functors  $K^{\Ccat_{\Z}}H_{\cL^{1}}$ and
$K^{\nCcat}$
are Morita invariant and preserve coproducts, and \eqref{vfojoivdsfvfdvswre} is an equivalence on
$\Hilb^{\fin}(\C)$,  we conclude that
 \eqref{gwergerrewferfwerf}  is an equivalence.
 \end{proof}
 
 \begin{lem}\label{lpfrferferertgrtgrgerg}
If $H$ is a finite subgroup of $G$, then the component
 \begin{equation}\label{}
 K\cX^{G,\ctr}((G/H)_{min,max}\otimes G_{can,min})\to K\cX^{G}((G/H)_{min,max}\otimes G_{can,min})\end{equation}
 of the comparison transformation \eqref{gtjwneiofgerfewrfwerfref} at $(G/H)_{min,max}\otimes G_{can,min}$  in $G\BC$ is an equivalence.
 \end{lem}
 \begin{proof}
 For any subgroup $H$ of $G$
 the categories $\bV^{G,\ctr}((G/H)_{min,max}\otimes G_{can,min})$ and 
 $\bV^{G}((G/H)_{min,max}\otimes G_{can,min})$ are Morita equivalent to the group ring   $\C[H]$ 
 and the reduced group $C^{*}$-algebra $C^{*}(H)$, respectively, 
 see \cite[Prop. 8.24]{equicoarse}, \cite[Prop. 9.12]{coarsek}. If $H$ is finite, then $\C[H]\to C^{*}(H)$
 is an isomorphism of finite-dimensional $C^{*}$-algebras. 
 We now use that the functors $K^{\Ccat_{\Z}}H_{\cL^{1}}$ and
$K^{\nCcat}$ are Morita invariant and that \eqref{vfojoivdsfvfdvswre} is an equivalence on finite-dimensional $C^{*}$-algebras.
 \end{proof}

 Recall the universal equivariant coarse homology theory described in  \cref{hokrptogkertpgerthe9}.
 Let \begin{equation}\label{fqwewefedqedew}G\Sp\cX\langle \Fin_{min,max}\otimes G_{can,min}\cup \All_{min,min}  \rangle\subseteq G\Sp\cX
\end{equation} be the  localizing subcategory   
 generated by the objects $\Yo^{G}((G/H)_{min,min})$ for all subgroups $H$ of $G$ and $\Yo^{G}(  (G/H)_{min,max}\otimes G_{can,min})$ forall finite subgroups $H$ of $G$.
 Recall from \cref{hokrptogkertpgerthe9} that we denote coarse homology theories and the corresponding colimit preserving functors on motivic spectra by the same symbol.
\cref{hkrtopgkoprtogkrtgretgert} and \cref{lpfrferferertgrtgrgerg} imply:
 \begin{kor}\label{gkopwreferfwerfrwef} On the   localizing subcategory   \eqref{fqwewefedqedew} of $G\Sp\cX$  the comparison map $c^{G}:K\cX^{G,\ctr}\to K\cX^{G}$ is an equivalence.  
 \end{kor}

 \begin{ex}\label{plkthertgtggetrg}
 If $Y$ is a $G$-compact  proper  Hausdorff $G$-space which is  a homotopy retract of a finite $G$-$CW$-complex
 with finite stabilizers,
 then  $$\Yo^{G}(\cO^{\infty}(Y_{\cB_{\max}})\otimes G_{can,min}) \in \Sp\cX\langle \Fin_{min,max}\otimes G_{can,min} \rangle\ .$$
 Here we equip $Y$ with the  canonical $G$-uniform  and coarse structures (see \cite[Lem. 4.10]{bel-paschke}), and $(-)_{\cB_{\max}}$ stands for replacing the bornology by    the maximal bornology.
 Note that spaces of the form $\cO^{\infty}(Y_{\cB_{\max}})\otimes G_{can,min}$ for $G$-simplicial complexes $Y$ with finite stabilizers  occur in the domain of the forget-control map, see 
  \cite[Sec. 8]{Bunke_20202}.

Assume now that $Z$  in $G\UBC$
a  homotopy retract of a $G$-finite $G$-simplicial complex with the 
 structures induced by the sperical path metric. 
 Then $$\Yo^{G}(\cO^{\infty}(Z))\in \Sp\cX\langle \All_{min,min}  \rangle\ .$$
  
  In both cases, using homotopy invariance, we  reduce to the case of a $G$-$CW$-complex or $G$-simplicial complex. We then argue by induction over the cells using excision
 and homotopy invariance. The start of the inductions are given by  
  \cref{hkrtopgkoprtogkrtgretgert} or \cref{lpfrferferertgrtgrgerg}, respectively.
  \hB
   \end{ex}

   We define the associated non-equivariant
     algebraic coarse $K$-homology $$K\cX^{\ctr}_{\bC}:BG\times \BC\to \Sp$$ and  the corresponding 
    Borel equivariant algebraic coarse $K$-homology 
    $$K\cX^{\ctr,hG}_{\bC}:G\BC\to \Sp$$ by specializing the definitions from 
    \cref{okhphherthertgeg9} to $K\cX^{G,\ctr}_{\bC}$. 
    
   Let $ \Add_{\C}\to \Add_{\C,2,1}$ be the $(2,1)$-category of $\C$-linear additive categories, functors and equivalences.
  The functor  \begin{equation}\label{fwedqwdeqdqewdqwd} \ell:\Add_{\C}\to \Add_{\C,2,1} \end{equation}  represents the Dwyer-Kan localization at the equivalences \cite{Bunke:ab}.
    If we assume that $\bC$ admits all AV-sums, then by restriction to wide subcategories   in \eqref{hjfhiquhqwuiehfew24},
    then by the  arguments of \cite[Sec. 10]{coarsek}   the functor $\ell \bV_{\bC}^{G,\ctr}$ has an extension
  with transfers
 \begin{equation}\label{foiwjeiorfwerfwefewerfw}\xymatrix{G\BC\ar[r]^{\ell } \ar[dr]&\Add_{\C}\ar[r]^{ \bV_{\bC}^{G,\ctr}}&\Add_{\C,2,1}\\&G\BC_{\tr}\ar@{..>}[ur]_{ \bV_{\bC,\tr}^{G,\ctr}}}\ .
\end{equation} 
Applying $K^{\Cat_{\Z}}H_{\cL^{1}}$  (which factorizes over \eqref{fwedqwdeqdqewdqwd} since it sends equivalences to equivalences) we get an extension   
   $$\xymatrix{G\BC\ar[rr]^{K\cX_{\bC}^{G,\ctr} } \ar[dr]& &\Sp\\&G\BC_{\tr}\ar@{..>}[ur]_{ K\cX_{\bC,\tr}^{G,\ctr}}}\ .$$
   The latter allows to define the Borelification map
   \begin{equation}\label{sdacasdcrgqrgrewge} \beta: K\cX^{G,\ctr}_{\bC} \to K\cX^{\ctr,hG}_{\bC}:G\BC\to \Sp
    \end{equation}
  by   \cref{jgoijowerfewrfewfwfwef}
    and
    the Borelification of the comparison map 
      \begin{equation}\label{fwefefwerfwrefrwfre} c^{hG}: K\cX^{\ctr,hG}_{\bC} \to K\cX^{hG}_{\bC}:G\BC\to \Sp\ .
    \end{equation}

It immediately follows from the construction of the transfers  in \cite[Sec. 10]{coarsek}
that the natural transformation
$$\bV^{G,\ctr}_{\bC}\to \bV^{G}_{\bC}:G\BC\to \Add_{\C}$$ given by the inclusion extends to a natural transformation
$$\bV^{G,\ctr}_{\bC,\tr}\to \bV^{G}_{\bC,\tr}:G\BC_{\tr}\to \Add_{\C,2,1}\ .$$
This in turn gives an extension of the comparison map \eqref{gtjwneiofgerfewrfwerfref} to a map
$$c^{G}_{\tr}:K\cX^{G,\ctr}_{\bC,\tr}\to K\cX^{G }_{\bC,\tr}:G\BC_{\tr} \to \Sp\ .$$
From \cref{ogjwioergjeworijuoiewferfwwrf} we conclude:
\begin{kor}\label{khopertgertgretgtrdng} We assume that $\bC$ admits all AV-sums.
The square of  natural transformations between equivariant coarse homology theories
\begin{equation}\label{}\xymatrix{K\cX^{G,\ctr}_{\bC}\ar[d]_{\eqref{gtjwneiofgerfewrfwerfref}}^{c^{G}}\ar[r]^{\beta}_{\eqref{sdacasdcrgqrgrewge}}& K^{\ctr,hG}_{\bC}\ar[d]_{\eqref{fwefefwerfwrefrwfre}}^{c^{hG}}\\ K\cX^{G}_{\bC}\ar[r]^{\beta}_{\eqref{erfwerfrfwreefr2334frewfr}}&K\cX^{hG}_{\bC}} \end{equation}
commutes
\end{kor}

  Note that as in the case of  equivariant  coarse  topological $K$-homology
  we have a non-equivariant  version $K\cX^{\{e\},\ctr}_{\bC}$ and a Borelification map
  $\beta':K\cX^{G,\ctr}_{\bC}\to K\cX^{\{e\},\ctr,hG}_{\bC}$ defined similarly as the map with the same symbol in  \eqref{vsdfvsdfvr3vsfvdfvsdfv}
  Analoguousy to \cref{kohperthtregrtgertgbtertgrtgrertgertge} and \cref{iguweorigwergwergw9} one can show, as announced in \cref{biojrgogbfgbdb}:
    \begin{prop}\label{kopbgbdgbdfgbgd} We assume that $\bC$ admits all AV-sums.
  There is a natural equivalence \begin{equation}\label{gjiweorjofwewerfwerfw}K\cX^{\ctr}_{\bC}\simeq K\cX^{\{e\},\ctr}_{\bC}:BG\times G\BC\to \Sp\end{equation}  
under which the Borelification maps $\beta,\beta'$ and the comparison maps $c^{hG},c^{\{e\},hG}$ become equivalent, i.e., 
 \begin{equation}\label{ghweruihfiuwehrfuiehurfiwerfewrfwerfwerf}\xymatrix{&K\cX^{G,\ctr}_{\bC}\ar[dr]^{\beta'}\ar[dl]_{\beta}&\\ K\cX_{\bC}^{\ctr,hG}\ar[rr]_{\eqref{gjiweorjofwewerfwerfw}^{hG}}^{\simeq}&&K\cX_{\bC}^{\{e\},\ctr,hG}}
\end{equation}  
  and \begin{equation}\label{gwergerffssgrw4}\xymatrix{ K\cX_{\bC}^{\ctr,hG}\ar[rr]_{\eqref{fwefefwerfwrefrwfre}}^{c^{hG}}\ar[d]_{\eqref{gjiweorjofwewerfwerfw}^{hG}} && \ar[d]^{\eqref{vsfdvfdvfdvsdfcwervdsfvfdvfdvsdfvse}^{hG}}K\cX_{\bC}^{hG} \\ K\cX_{\bC}^{\{e\},\ctr,hG}\ar[rr]^{c^{\{e\},hG}} && K\cX_{\bC}^{\{e\},hG}} 
\end{equation} commute.
    \end{prop}


%
%
%
%
%
%

\subsection{Equivariant  coarse    periodic cyclic homology}

In this section we introduce the equivariant coarse periodic cyclic homology theory and the corresponding  Borelification.  The equivariant coarse periodic cyclic homology theory will be defined by composing the periodic cyclic homology functor for $\C$-linear additive categories with the functor 
$\bV^{G,\ctr}_{\bC}$ from \eqref{gweoihjgoiwerfwerfrwef} and forcing $u$-continuity. 
In principle we could map $\C$-linear additive categories, using bounded chain complexes represening the stable hull \cite{zbMATH07160436}, to $\C$-linear
stable $\infty$-categories and then apply the general periodic cyclic homology functor defined by taking
$S^{1}$-Tate-fixed points of the Hochschild homology functor for $\C$-linear stable $\infty$-categories. 
But in order to construct trace maps from the
equivariant coarse periodic cyclic homology theory to periodic coarse homology later
we rely on the concrete classical chain complex model for periodic cyclic homology of $\C$-linear additive categories which we now recall.

Recall from \cite{Kassel_1987} that a mixed complex over a commutative ring $k$ is a triple $(C_{*},d,B)$, where   
$(C_{*},d)$ is a chain complex over $k$ and $B:C_{*}\to C_{*+1}$ is an additional degree-$1$ map  such that $dB+Bd=0$ and $B^{2}=0$.
A morphism between mixed complexes is a map of chain complexes  over $k$ which  intertwines the additional degree-$1$ maps. We get a category $\Mix_{k}$ of mixed complexes over $k$ and morphisms.

A morphism   of mixed complexes is called a weak equivalence   if it induces a 
  quasi-isomorphism of underlying chain complexes. We let \begin{equation}\label{dfvdsfvwerv} \ell:\Mix_{k}\to \Mix_{k,\infty}
\end{equation} 
denote the Dwyer-Kan localization at the weak equivalences.
The $\infty$-category $\Mix_{k,\infty}$ turns out to be cocomplete and stable \cite[Prop. 3.1.10]{caputi-diss}.

We have a functor
\begin{equation}\label{bgfbdfgbdgfbdf} \rmMix_{k}:  \Cat_{k}\to \Mix_{k,\infty}
\end{equation}  which sends a $k$-linear category $\bA$ to the  mixed complex
associated to the cyclic nerve of $\bA$ considered as an object of $\Mix_{k,\infty}$.

\begin{rem} \label{kwgopergweferfwef}For $\bA$ in $\Cat_{k}$
we have \begin{equation}\label{ijfwoervwewer}\rmMix_{k}(\bA)_{n}\cong   \bigoplus_{(A_{0},\dots,A_{n})\in \Ob(\bA)^{n+1}}  \Hom_{\bA}(A_{0},A_{n})\otimes_{k} \Hom_{\bA}(A_{1},A_{0})\otimes_{k}\dots  \otimes_{k}  \Hom_{\bA}(A_{n},A_{n-1})\ .
\end{equation} 
We let $d_{i}:\rmMix_{k}(\bA)_{n}\to \rmMix_{k}(\bA)_{n-1}$ be the $k$-linear map  induced by the composition
at $A_{i}$. Thereby $d_{n}(f_{0}\otimes \dots \otimes f_{n})=(-1)^{n}f_{n}f_{0}\otimes f_{1}\otimes \dots \otimes f_{n-1}$.
Then $$d:=\sum_{i=0}^{n}(-1)^{i} d_{i}:\rmMix_{k}(\bA)_{n}\to \rmMix_{k}(\bA)_{n-1} $$ is the differential of the mixed complex in degree $n$. The map $$B:\rmMix_{k}(\bA)_{n}\to \rmMix_{k}(\bA)_{n+1}$$
is given by
$(-1)^{n+1} (1-t)sN$, where $t:\Mix_{k}(\bA)_{n+1}\to \Mix_{k}(\bA)_{n+1}$   is  the cyclic permutation sending 
$f_{0}\otimes \dots \otimes f_{n+1}$ to $(-1)^{n+1}f_{n+1}\otimes f_{0}\otimes \dots\otimes f_{n}$,
 $N:=\sum_{i=0}^{n+1} t^{i}: \Mix_{k}(\bA)_{n+1}\to \Mix_{k}(\bA)_{n+1}$,
and the map $s :\Mix_{k}(\bA)_{n}\to \Mix_{k}(\bA)_{n+1}$ sends $f_{0}\otimes \dots \otimes f_{n}$ to
$\id_{A_{n}}\otimes  f_{0}\otimes \dots \otimes f_{n}$.
\hB
\end{rem}

While $\Mix_{k}$ is defined on all of $\Cat_{k}$ its restriction to $\Add_{k}$ is has good  properties.
Recall \cref{kophtrherthetregrtg}.
\begin{prop}\label{hoperthetrgerger}
If $k$ is a field, then
the  functor $\rmMix_{k}:  \Add_{k}\to \Mix_{k,\infty}$ is homological.
\end{prop}
\begin{proof}
This assertion is surely folklore. We sketch the argument  from \cite{caputi-diss}.

The cyclic nerve construction, forming the associated mixed complex, and the localization $\ell$  in \eqref{dfvdsfvwerv}
all preserve filtered colimits. Therefore $\rmMix_{k}$ does so.

By  \cite[Lem. 3.4.4]{caputi-diss} the functor $\rmMix_{k}$ is equivalent to the restriction 
 Keller's   mixed complex functor \cite{Keller_1999}  for exact $k$-linear categories to $k$-linear additive categories equipped with the split exact structure. In more detail, 
 recall that Keller's functor associates to an exact category $\bA$ the  object
 \begin{equation}\label{bsdoijviosdfvsdfvsfdvsf}\Cofib(\rmMix_{k}(\Acyc^{b}(\bA'))\to \rmMix_{k}(\Ch^{b}(\bA')))
\end{equation}
  where   $\bA'\to \bA$ is a resolution by an exact flat $dg$-category,
 $\Ch^{b}(\bA')$ is the category of bounded chain chain complexes over $\bA'$, and
 $ \Acyc^{b}(\bA')$ is the full subcategory of $\Ch^{b}(\bA') $ of acylic chain complexes.
 In this formula  $\rmMix_{k}$ is   the obvious extension of the mixed complex functor from $k$-linear to $dg$-categories over $k$.
  
 First of all, since we assume that $k$  is a field, it is not necessary
 to  go over to flat resolutions. Since we consider additive categories as exact categories with the split-exact structure
 the mixed complex of the $dg$-category of bounded acyclic chain complexes over the additive category is acyclic 
 \cite[Lem. 3.4.3]{caputi-diss}. We therefore do not need to consider the cofibre construction above. 
 Finally the map $\bA\to \Ch^{b}(\bA)$ sending the objects of $\bA$ to chain complexes concentrated in degree $0$ induces a weak equivalence  $\rmMix_{k}(\bA)\to \rmMix_{k}(\Ch^{b}(\bA))$ by \cite[Lem. 3.4.3(i)]{caputi-diss}.
 
 We therefore know that the functor $\rmMix_{k}$ from \eqref{bgfbdfgbdgfbdf} has the
homological properties of Keller's functor stated in  \cite[Thm. 1.5.]{Keller_1999}.

An equivalence of additive categories $\bA\to \bA'$  induces a weak equivalence  of
 $\infty$-categories of bounded chain complexes  $\Ch^{b}(\bA)_{\infty}\to \Ch^{b}(\bA')_{\infty}$ (see \cite[Lem. 2.15]{Bunke:2017aa})
 and therefore, by the exactness result \cite[Thm. 1.5.(c)]{Keller_1999}, an equivalence
 $\rmMix_{k}(\bA)\to \rmMix_{k}(\bA')$.

 By \cite[Ex. 1.8, Prop. 2.6]{Schlichting_2004} a Karoubi filtration  $\bI\subseteq \bA$ induces a Karoubi sequence 
 $$    \Ch^{b}(\bI)_{\infty}\to \Ch^{b}(\bA)_{\infty} \to \Ch^{b}(\bA/\bI)_{\infty}$$
 which again by  
  \cite[Thm. 1.5.(c)]{Keller_1999} induces a fibre sequence
  $$\rmMix_{k}(\bI)\to \rmMix_{k}(\bA)\to \rmMix_{k}(\bA/\bI)\ .$$
%
%
%
The exactness property just shown implies that $\rmMix_{k}$ preserves sums and therefore  flasques. Since the target
category is stable, flasques therein are zero objects and we can conclude that $\rmMix_{k}$ annihilates flasques.
\end{proof}

 We have a bifunctor
 $$\Alg_{k}\times \Cat_{k}\to \Cat_{k}\ , \qquad (R,\bA)\mapsto R\otimes_{k}\bA\ .$$
 The category $R\otimes_{k}\bA$ has the same objects as $\bA$, and
 the morphism spaces are given by \begin{equation}\label{iujgiowerfwerfwerf}\Hom_{ R\otimes_{k}\bA}(A,A'):=R\otimes_{k}\Hom_{\bA}(A,A')\ .
\end{equation}
 The composition involves the product of $R$. Applied to a non-unital algebra this construction would yield a non-unital category. In order to avoid this problem we adopt the following construction.
 
Let $R$ be a possibly non-unital $k$-algebra and $R^{+}\to k$ be the canonical morphism from the unitalization $R^{+}$. Then we define the $R$-twist of the mixed complex functor by
\begin{equation}\label{bsdbsdfd}\rmMix_{R}:=\Fib(\rmMix_{k}(R^{+}\otimes_{k}-)\to \rmMix_{k}(-)):\Add_{k}\to \Mix_{k,\infty}\ . 
\end{equation} 

The following is a consequence of \cref{hoperthetrgerger}.

\begin{kor}\label{hoperthetrgerger1}  If $k$ is field, then 
the functor $\rmMix_{R}:\Add_{k}\to \Mix_{k,\infty}$ is homological.
\end{kor}

Assume that $\tr:R\to k$ is a $k$-linear trace.

\begin{lem}\label{kopgerwefwergrtg}
The trace induces a natural transformation of functors
\begin{equation}\label{hrtgfgbdgfdb}\Tr:\rmMix_{R}\to \rmMix_{k}:\Cat_{k}\to \Mix_{k,\infty}\ .
\end{equation}
\end{lem}
\begin{proof} The underlying $k$-module of $R^{+}$ is $R\oplus k$ whose elements will be denoted by $(r,\lambda)$
with $r$ in $R$ and $\lambda$ in $k$.
We can extend the trace to a trace $\tr^{+}:R^{+}\to k$ by setting
$\tr^{+}(r,\lambda):=\tr(r)$.  The transformation \eqref{hrtgfgbdgfdb}  is induced by a chain-level construction $\Tr^{+}$.  As explained in \cref{kwgopergweferfwef} and using \eqref{iujgiowerfwerfwerf} we have  
$$\rmMix_{R^{+}}(\bA)_{n}\cong \bigoplus_{(A_{0},\dots,A_{n})\in \Ob(\bA)^{n+1}} (R^{+}\otimes_{k} \Hom_{\bA}(A_{0},A_{n}))\otimes_{k}\dots \otimes_{k}(R^{+}\otimes_{k} \Hom_{\bA}(A_{n},A_{n-1}))\ .$$
The transformation $\Tr^{+}$ sends
$$(r_{0}\otimes f_{0})\otimes \dots \otimes (r_{n}\otimes f_{n})$$ to
$$\tr^{+}(r_{0}\dots r_{n})   f_{0}\otimes_{k} \dots \otimes_{k} f_{n}$$ in \eqref{ijfwoervwewer}.
One checks that this map is compatible with the differentials  $d$ and $B$. The canonical homomorphism $R^{+}\to k$ induces the second map in the
   exact sequence of  mixed complexes  $$0\to  F\to \rmMix_{R^{+}}(\bA) \to \rmMix_{k}(\bA)_{n} \to 0 $$
 (this time considered in $\Mix_{k}$), where according to \eqref{bsdbsdfd} the kernel $F$ of the second map represents $\rmMix_{R}(\bA)$.
 The restriction of $\Tr^{+}$ to $F$ is a map $\Tr:\rmMix_{R}(\bA)\to \rmMix_{k}(\bA)$ in $\Mix_{k,\infty}$ 
 which represents the desired natural transformation.
 \end{proof}

The following is a version of \cite[Thn. 3.3.2]{caputi-diss} and a formal consequence of \cite[Thm. 7.2]{Bunke:2025aa}
and \cref{hoperthetrgerger1}.
Let $R$ be a $\C$-algebra, e.g., the algebra $\cL^{1}$ of trace class operators.
\begin{theorem}
The composition
$$\rmMix\cX^{G}_{\bC,R}:G\BC\xrightarrow{\bV_{\bC}^{G,\ctr}}\Add_{\C}\xrightarrow{\rmMix_{R}} \Mix_{\C,\infty}$$
is an equivariant   coarse homology theory.
\end{theorem}

If $R=\C$, then we omit this symbol from the notation.

\begin{rem}
One can check that $\rmMix\cX^{G}_{\bC,R}$ is in addition strong and continuous. As these properties are not needed
in the present paper we will not give the argument. \hB
\end{rem}

If $R$ has a trace $\tr:R\to \C$, then  by \cref{kopgerwefwergrtg} we get a natural transformation of coarse homology theories
\begin{equation}\label{vsdfvbvbsdbdfb}\Tr:\rmMix\cX^{G}_{\bC,R}\to\rmMix\cX^{G}_{\bC}:G\BC\to \Mix_{\C,\infty}\ .
\end{equation}

We have a functor
\begin{equation}\label{bsdfvsdfvsfsfvwr}\PCH':\Mix_{k}\to \Ch_{k}
\end{equation}
which sends a mixed complex to the associated periodic cyclic homology complex.
The functor is finitely additive and preserves exact sequences. 

\begin{rem} \label{koprethertgetgegtrg}If $(C_{*},d,B)$ is a mixed complex over $k$, then the periodic cyclic homology complex
is given by $$\PCH'(C_{*},d,B)_{n}:=\prod_{k\in \nat} C_{n+2k}$$ with the differential
$\hat d+\hat B$, where $\hat d$ is the factorwise application of $d$ and
$\hat B$ is given by $B:C_{*+2k}\to C_{*+2k+1}$, where the target is considered as a factor with index $k+1$
(so that $\hat B$ has degree $-1$).   \hB \end{rem}

The functor  in \eqref{bsdfvsdfvsfsfvwr}
preserves weak equivalences and therefore descends to a functor
$$\PCH:\Mix_{k,\infty}\to   \Mod(Hk)\ .$$
This functor $ \PCH$  does not preserve infinite filtered colimits  because of 
  the infinite products involved in its chain-level construction.

    We  define the functor
 $$\widetilde{\PCH\cX^{G}_{\bC,R}}:G\BC\xrightarrow{\rmMix\cX^{G}_{\bC,R} }  
 \Mix_{\C,\infty}\xrightarrow{\PCH }\Mod(H\C)\ .$$
Since $ \PCH$ is exact this functor inherits the following  properties from $\rmMix\cX^{G}_{\bC,R}$:
   coarse invariance,
 excision, and
  vanishing on flasques.
 But since $\PCH$ does not preserve filtered colimits this composition is no longer $u$-continuous and therefore not a coarse homology theory.
 But we can force $u$-continuity by  \cref{kophprthgertrtgegrtg}.
 \begin{ddd}\label{kopgewerfrefewf}
We define the equivariant coarse periodic cyclic homology theory 
with coefficients in $(\bC,R)$ $$\PCH\cX^{G}_{\bC,R}:= (\widetilde{\PCH\cX^{G}_{\bC,R}})^{u}:G\BC\to \Mod(H\C)$$ 
as the functor obtained  by forcing $u$-continuity on $\widetilde{\PCH\cX^{G}_{\bC,R}}$.
\end{ddd}

If $R=\C$, then we omit this symbol from the notation.

\begin{rem}
One can show in addition that $\PCH\cX^{G}_{\bC,R}$ is strong. \hB
\end{rem}

 In the presence of a trace $\tr:R\to \C$ 
 the transformation \eqref{vsdfvbvbsdbdfb} induces a transformation of equivariant coarse homology theories
 \begin{equation}\label{gerwfewrfefwerf}\Tr:\PCH\cX^{G}_{\bC,R}\to \PCH\cX^{G}_{\bC}:G\BC\to \Mod(H\C) \ .
\end{equation}

 Recall that
 $$\PCH\cX^{G}_{\bC,R}:=(\PCH\circ \rmMix_{R}\circ \bV_{\bC}^{G,\ctr})^{u}:G\BC\to \Mod(H\C)\ .$$
If we assume that $\bC$ admits all AV-sums, then the extension $\bV^{G}_{\bC,\tr}$ from \eqref{foiwjeiorfwerfwefewerfw}  provides an extension  \begin{equation}\label{iozgweuirfwerfwerfwref}\xymatrix{G\BC\ar[rr]^{\PCH\cX^{G}_{\bC,R} } \ar[dr]& &\Mod(H\C)\\&G\BC_{\tr}\ar@{..>}[ur]_{\PCH\cX^{G}_{\bC,R,\tr}}}\ .
\end{equation}
We can therefore define the Borel equivariant version
\begin{equation}\label{gwregwer}
 \PCH\cX^{hG}_{\bC,R}:G\BC \to \Mod(H\C) 
\end{equation}
of the equivariant coarse periodic cyclic homology and
the Borelification map
 \begin{equation}\label{werfwerfwerf} 
\beta:\PCH\cX^{G}_{\bC,R}\to \PCH\cX^{hG}_{\bC,R}:G\BC \to \Mod(H\C) 
\end{equation}
 by specializing the constructions from 
    \cref{okhphherthertgeg9}.

  Note that as in the case of  equivariant coarse algebraic $K$-homology
  we have a non-equivariant  version $\PCH\cX^{\{e\}}_{\bC,R}$ and a Borelification map
  $\beta':\PCH\cX^{G}_{\bC,R}\to \PCH\cX^{\{e\},hG}_{\bC,R}$ defined similarly as the map with the same symbol in  \eqref{vsdfvsdfvr3vsfvdfvsdfv}
  Analoguousy to \cref{kohperthtregrtgertgbtertgrtgrertgertge} and \cref{iguweorigwergwergw9} one can show, as announced in \cref{biojrgogbfgbdb}:
    \begin{prop}\label{kopbgbdgberferferdfgbgd} We assume that $\bC$ admits all AV-sums.
  There is a natural equivalence \begin{equation}\label{gjiweorjofwewerfwerfw111}\PCH\cX_{\bC,R}\simeq \PCH\cX^{\{e\}}_{\bC,R}:BG\times G\BC\to \Mod(H\C)\end{equation}  
under which the Borelification maps $\beta$ and $\beta'$ become equivalent, i.e.,
 \begin{equation}\label{gweorjfiojiwerfwerfwerfw}\xymatrix{&\PCH\cX^{G}_{\bC,R}\ar[dr]^{\beta'}\ar[dl]_{\beta}&\\ \PCH\cX_{\bC,R}^{hG}\ar[rr]_{\eqref{gjiweorjofwewerfwerfw111}^{hG}}^{\simeq}&&\PCH\cX_{\bC,R}^{\{e\},hG}}\ .
\end{equation}  
commutes.
    \end{prop}

%
%

\subsection{The algebraic coarse Chern character}

In this section we introduce the algebraic coarse Chern character
as a natural transformation from  equivariant algebraic coarse $K$-homology to equivariant  coarse periodic cyclic homology. We further show that it is compatible with Borelification.

Let $k$ be a field of characteristic $0$. The periodic cyclic homology functor $\PCH_{k}$ has a canonical extension to
$\Cat_{\Z}$ and we 
 have a natural transformation
\begin{equation}\label{werfwerferwfwrefref}\rmch^{GJ}:K^{\Cat_{\Z}}H \to  \PCH_{k}: \Add_{k}\to \Mod(Hk)\ .
\end{equation}

\begin{rem}
The superscript $GJ$ stands for Goodwillie-Jones and indicates that this transformation
is derived from the Goodwillie-Jones trace. The Chern character map \eqref{werfwerferwfwrefref} is well-known and widely used in the literature, but the detailed construction is difficult to locate in the literature.
For completeness we provide a sketch.

Note that the domain and target of \eqref{werfwerferwfwrefref} are defined on $\Cat_{k}$, and that the 
Chern character can be extended to this bigger domain. But for simplicity, in the following we restrict to additive categories since we defined $\rmMix_{k}$ using the naive mixed complex functor in contrast to \eqref{bsdoijviosdfvsdfvsfdvsf}.

%
%
Let $$\rmCH^{-}:\Mix_{k,\infty}\to \Mod(Hk)$$ be the negative cyclic homology complex functor.
We have a natural transformation  of functors \begin{equation}\label{bdfsbvfdvsfdv}\rmCH^{-}\to \PCH:\Mix_{k,\infty}\to \Mod(Hk)
\end{equation} 
induced by the inclusion of the negative cyclic homology complex into the periodic one.

By Morita invariance, the classical non-connective  algebraic $K$-theory functor for rings over $k$ is equivalent to the composition
 $$K^{\Ring_{k}}:\Ring_{k}\xrightarrow{\incl} \Cat_{k} \xrightarrow{ K^{\Cat_{\Z}}}  \Sp\ .$$
 Similarly we define
 the functor
 $$\rmMix_{k}^{\Ring}:\Ring_{k}\xrightarrow{\incl}   \Cat_{k}\xrightarrow{\rmMix_{k}} \Mix_{k,\infty}\ .$$
 Furthermore we define   the negative cyclic homology for rings by 
\begin{equation}\label{bdfgbdfgbdfgbdgfb}\rmCH_{k}^{\Ring,-}:\Ring_{k} \xrightarrow{\rmMix^{\Ring}_{k}} \Mix_{k,\infty}\xrightarrow{\rmCH^{-}}\Mod(Hk)
\end{equation} 
and 
   the periodic cyclic homology for rings by 
$$\PCH_{k}^{\Ring}:\Ring_{k} \xrightarrow{\rmMix^{\Ring}_{k}} \Mix_{k,\infty}\xrightarrow{\PCH}\Mod(Hk)\ .$$
 The transformation \eqref{bdfsbvfdvsfdv} induces a transformation
$$\rmCH_{k}^{\Ring,-}\to \PCH_{k}^{\Ring}:\Ring_{k}\to \Mod(Hk)\ .$$
Since $\Mix_{k}$ is the naive mixed complex of the cyclic nerve of the $k$-linear category
the functors $\rmMix_{k}^{\Ring}$, $\rmCH_{k}^{\Ring,-}$, $\PCH^{\Ring}_{k}$ are represented by the explicit chain complex-valued constructions described, e.g., in \cite{zbMATH01093754}.

We start from the classical Goodwillie-Jones trace 
$$c^{GJ}:K^{\Ring_{k}}\to \rmCH_{k}^{\Ring,-}\ .$$
We let
$$h:K^{\Ring_{k}}\to K^{\Ring_{k}}H:\Ring_{k}\to \Sp$$ be the universal homotopy invariant functor
under $K^{\Ring}$ given by  Weibel's homotopy $K$-theory \cite{zbMATH04095731}. 
We then consider the square of natural transformations   $$\xymatrix{K^{\Ring_{k}}\ar[r]^{h}\ar[d]^{c^{GJ}} & K^{\Ring_{k}}H\ar@{..>}[d]^{\rmch^{GJ,\Ring}} \\ \rmCH_{k}^{\Ring,-}\ar[r] &  \PCH_{k}^{\Ring}} 
$$ of functors from $\Ring_{k}$ to $\Mod(Hk)$. Since $ \PCH_{k}^{\Ring}$ is homotopy invariant by  \cite{Goodwillie_1985}  \cite{Goodwillie_1986}  
we get the dotted arrow $\rmch^{GJ,\Ring}$ from the universal property of $h$.

In order to get $\rmch^{GJ}$ in \eqref{werfwerferwfwrefref} we must now extend $\rmch^{GJ,\Ring}$ from rings over $k$ to $k$-linear additive categories.
For more details about the following constructions see \cite[Sec. 13.4]{Bunke:2025aa}.
Let  $\Cat_{k,i}$ denote the wide subcategory of $\Cat_{k}$ of functors which are injective on objects. We have a functor $$A^{\alg}:\Cat_{k,i}\to \Ring_{k}$$
which sends a $k$-linear category $\bA$ to the ring $ A^{\alg}(\bA)$ whose underlying
  $k$-module  is $$\bigoplus_{A,B\in \Ob(\bA)} \Hom_{\bA}(A,B)\ ,$$ and whose composition is the obvious matrix
multiplication.  The functor $A^{\alg}$ can be restricted to $\Add_{k}$.
 Using that    $K^{\Ring_{k}}H$ and $\rmMix_{k}^{\Ring} $  are matrix stable and preserve filtered colimits one can check that 
 the compositions
 $$K^{\Ring_{k}}H\circ A^{\alg}:\Add_{k,i}\to  \Sp$$
 and $$\rmMix_{k}^{\Ring}\circ A^{\alg}:\Add_{k,i}\to \Mix_{k,\infty}$$
 preserve equivalences, see \cite[Lem. 13.20]{Bunke:2025aa} for an argument. Hence also  $\PCH_{k}^{\Ring}\circ A^{\alg}:\Add_{k}\to \Mod(Hk)$
 preserves equivalences. Note that the argument of  \cite[Lem. 13.20]{Bunke:2025aa} 
 does not directly apply to periodic cyclic homology since this functor does not preserve filtered colimits.
 We get factorizations
 $$\xymatrix{\Add_{k,i}\ar[r]^{K^{\Ring_{k}}H\circ A^{\alg}}\ar[d]&\Sp\\ \Add_{k,i}[eq^{-1}]\ar@{..>}[ur]_{\hat K^{\Add_{k}}H}&}\ ,
 \xymatrix{\Add_{k,i}\ar[r]^{\rmMix_{k}^{\Ring}\circ A^{\alg}}\ar[d]&\Mix_{k,\infty}\\ \Add_{k,i}[eq^{-1}]\ar@{..>}[ur]_{\hat \rmMix_{k } }&}  \ , \xymatrix{\Add_{k,i}\ar[r]^{\PCH_{k}^{\Ring}\circ A^{\alg}}\ar[d]&\Mod(Hk)\\ \Add_{k,i}[eq^{-1}]\ar@{..>}[ur]_{\widehat \PCH_{k}}&}\ ,$$
 where the vertical arrows are the Dwyer-Kan localizations at the equivalences.
 The transformation $\rmch^{GJ,\Ring}$ induces a transformation of functors 
 \begin{equation}\label{bdfgbdfbdfgbdfgb}\hat \rmch^{GK}:\hat K^{\Add_{k}}H\to \widehat{\PCH_{k}}:\Add_{k,i}\to \Sp\ .
\end{equation}
 We finally use that the inclusion 
 $\Add_{k,i}\to \Add_{k}$ induces an equivalence
 $\Add_{k,i}[eq^{-1}]\stackrel{\simeq}{\to} \Add_{k}[eq^{-1}]$.
 We have squares
 $$\xymatrix{\Add_{k}\ar[r]^{K^{\Cat_{\Z}}H}\ar[d] &\Sp   \\ \Add_{\Z}[eq^{-1}]&\ar[l]^{\simeq} \Add_{k,i}[eq^{-1}]\ar[u]^{\hat K^{\Add_{k}}H} }\ , \quad   
 \xymatrix{\Add_{k}\ar[r]^{\rmMix_{k} }\ar[d] &\Sp   \\ \Add_{k}[eq^{-1}]&\ar[l]^{\simeq} \Add_{k,i}[eq^{-1}]\ar[u]^{\hat \rmMix_{k} } }$$ and
 $$\xymatrix{\Add_{k}\ar[r]^{\PCH_{k} }\ar[d] &\Mod(Hk)   \\ \Add_{k}[eq^{-1}]&\ar[l]^{\simeq} \Add_{k,i}[eq^{-1}]\ar[u]^{\widehat{\PCH_{k}}} } \ .$$
 For the commutativity of the left square we refer to  \cite[Lem. 13.21]{Bunke:2025aa}.
 The argument only depends the fact that the functors preserve filtered colimits and that
 $K^{\Cat_{\Z}}H_{|\Add_{k}}$ is Morita invariant.  
 The same holds true for $\rmMix_{k} $ (see \cite[Thm. 1.5.(a)]{Keller_1999} for Morita invariance)
 which implies commutativity of the second square.  This directly implies the commutativity of the 
 third.

   The natural transformation \eqref{bdfgbdfbdfgbdfgb} eventually
 induces the desired natural transformation \eqref{werfwerferwfwrefref}. 
 \hB
 \end{rem}
 
 \begin{rem}
 For explicit calculations one often needs a formula
 for the composition
 $$ K^{\Cat_{\Z}}\to  K^{\Cat_{\Z}}H \xrightarrow{\rmch^{GJ}}  \PCH_{k} :\Add_{k}\to \Sp\ .$$   To this end one observes that    
 this transformation extends to a transformation of functors defined on $\mathbf{dgCat}_{k}$.
 The induced transformation
 $$\pi_{0}K^{\Cat_{\Z}}\to \pi_{0}  \PCH_{k}:\mathbf{dgCat}_{k}\to \Ab$$
 is characterized by universal properties in \cite{Tabuada_2011}.
 This characterization leads to explicit formulas
 as  in \cite[4.3.4]{Ludewig:2025aa}.
 \hB
  \end{rem}

In view of \cref{kopgwegwerfrfwref} and \cref{kopgewerfrefewf} the transformation
\eqref{werfwerferwfwrefref} induces a natural transformation of equivariant coarse homology theories \begin{equation}\label{oihbuirthbidgbdfg}\rmch^{\alg}_{\cL^{1}}:K\cX^{G,\ctr}_{\bC }\to \PCH\cX^{G}_{\bC,\cL^{1}}:G\BC\to \Sp\ .
\end{equation}

We can post-compose with the transformation
\eqref{gerwfewrfefwerf} induced by the trace $\tr:\cL^{1}\to \C$.
\begin{ddd}\label{iopgwregfrwefwerfrefw}
We define the algebraic Chern character transformation
as the composition
$$\rmch^{G,\alg}:K\cX^{G,\ctr}_{\bC }\xrightarrow{\rmch^{\alg}_{\cL^{1}},\eqref{oihbuirthbidgbdfg}} \PCH\cX^{G}_{\bC,\cL^{1}}\xrightarrow{\Tr,\eqref{gerwfewrfefwerf}}
 \PCH\cX^{G}_{\bC}:G\BC\to \Sp\ .$$
\end{ddd}

\begin{rem}\label{kophwegergwerwf}
We have constructed the bold part of
\begin{equation}\label{okhpowhwegrhthtr}\xymatrix{K\cX_{\bC}^{G,\ctr}\ar[rr]^{c}\ar[d]^{\rmch^{G,\alg}}&&K\cX_{\bC}^{G}\ar@{.>}[dll]^{\rmch^{G,\topp}}\\\PCH\cX^{G}_{\bC}}
\end{equation}
and one could ask whether there exists  a factorization indicated by the dotted arrow.
Recall the localizing subcategory  \eqref{fqwewefedqedew} 
 of $G\Sp\cX$
and \cref{gkopwreferfwerfrwef}.
\begin{kor}We have a natural transformation
   $$\rmch^{G,\topp}:K\cX^{G}_{|\eqref{fqwewefedqedew} }\to \PCH\cX^{G}_{|\eqref{fqwewefedqedew} }\ .$$ 
\end{kor}
At the moment we do not have an argument that $\rmch^{G,\topp}$ does not exist in general on all of $G\BC$. But even partial existence beyond \eqref{fqwewefedqedew} would have strong consequences.

  Recall \cref{gkopwreferfwerfrwef} and  \cref{plkthertgtggetrg}.
 The complex version of the Novikov conjecture asserts that  the assembly map 
 $$\asmbl^{K^{G}}:\colim_{G_{\Fin}\Orb} K^{G}\wedge H\C\to K^{G}(*)\wedge H\C$$
 is  injective on homotopy groups, where 
 $$K^{G}:=K\cX^{G}((-)_{min,max}\otimes G_{can,min}):G\Orb\to \Mod(KU)\ .$$
 We define the functors $$K^{G,\ctr}:=K\cX^{G,\ctr}((-)_{min,max}\otimes G_{can,min}):G\Orb\to \Sp$$
 and $$\PCH^{G}:=\PCH\cX^{G}((-)_{min,max}\otimes G_{can,min}):G\Orb\to \Mod(H\C)\ .$$
  We now assume that the extension $\rmch^{G,\topp}$ exist  on the  image of the functor \begin{equation}\label{rewferfwefwerfwrf}
 (-)_{min,max}\otimes G_{can,min}:G\Orb\to G\BC\ . \end{equation} 
  Then we have  the  commutative diagram in $\Sp$
 $$\xymatrix{ \ar@/^-3cm/[dd]_{\rmch^{G,\topp}}^{\simeq}\colim_{G_{\Fin}\Orb} K^{G}\wedge H\C\ar[rrr]^{\asmbl^{K^{G}}\wedge H\C}  &&& K^{G}(*)\wedge H\C\ar[dd]^{\rmch^{G,\topp}}  \\\colim_{G_{\Fin}\Orb} K^{G,\ctr}\wedge H\C\ar[d]^{\rmch^{G,\alg}}\ar[u]^{\simeq}_{c^{G}\wedge H\C}&&\\
\colim_{G_{\Fin}\Orb} \PCH^{G}\ar[rrr]^{\asmbl^{\PCH^{G}}}&&&\PCH^{G}(*)
  }\ .$$ By  \cref{gkopwreferfwerfrwef}
the comparison map $c^{G}$  is an equivalence.  Since the targets of $\rmch^{G,\alg}$ and $\rmch^{G,\topp}$ are 
  $H\C$-modules      we have   canonically induced $H\C$-module maps also denoted by $\rmch^{G,\alg}$ and $\rmch^{G,\topp}$ as indicated.
 For a finite group $H$ the algebra  $C^{*}(H)$ is a finite sum of matrix algebras. Since
 $$K^{\nCalg}(\C)\wedge H\C\stackrel{\simeq}{\leftarrow}  K^{\Cat_{\Z}}H(\cL^{1})\stackrel{\rmch^{GJ}}{\to} \PCH_{\C}(\cL^{1})\stackrel{\Tr}{\to} \PCH(\C)$$  induces an equivalence (domain and target are both equivalent to
 $\prod_{k\in \Z}\Sigma^{2k} H\C$ and the map induces some equivalence)
 we conclude by additivity  of the functors that also the left vertical composition is an equivalence as indicated.
 
 By  \cite[Thm. 1.7 \& 1.8]{L_CK_2006}, if 
  $E_{\Fin}G$ has a $G$-compact $CW$-model,   then $\asmbl^{\PCH^{G}}$ is split injective. 
  We could conclude that
  $\asmbl^{K^{G}}\wedge H\C$ is split injective.

\begin{kor}
Under the assumption   that $E_{\Fin}G$ has a $G$-compact $CW$-model the existence of the extension  $\rmch^{G,\topp}
$ in \eqref{okhpowhwegrhthtr} on the image of the functor \eqref{rewferfwefwerfwrf} implies the complex version of the Novikov conjecture for $G$.
\end{kor}
\hB
\end{rem}

The algebraic Chern character is induced from a natural transformation
of functors
\begin{equation}\label{hrthertgertgertgetrg}K^{\Cat_{\Z}}H_{\cL^{1}}\stackrel{ \rmch^{GL}_{\cL^{1}}}{\to} \PCH_{\cL^{1}}\stackrel{\Tr}{\to}  \PCH
\end{equation}applied to $\bV^{G,\ctr}_{\bC}$.   These functors are defined on $\Add_{\C}$ and send equivalences to equivalences. If $\bC$ admits all AV-sums, then we have the extension 
\eqref{hjfhiquhqwuiehfew24} which in turn  induces an extension
$$\rmch^{G,\alg}_{\tr}:K\cX^{G,\ctr}_{\bC,\tr}\to \PCH\cX^{G}_{\bC,\tr}:G\BC_{\tr}\to \Sp$$
of $\rmch^{G,\alg}$ from \cref{iopgwregfrwefwerfrefw}.
By \cref{ogjwioergjeworijuoiewferfwwrf} we conclude:

\begin{kor}\label{jihgiorhertgertgt55} We assume that $\bC$ admits all AV-sums. The following square commutes: 
\begin{equation}\label{briojbiofdjgdo}\xymatrix{ K\cX^{G,\ctr}\ar[rr]^{\rmch^{G,\alg}}\ar[d]^{\beta} && \PCH\cX^{G}_{\bC} \ar[d]^{\beta} \\ K\cX^{\ctr,hG}_{\bC} \ar[rr]^{\rmch^{\alg,hG} } &&  \PCH\cX^{hG}_{\bC} } 
\end{equation} \end{kor}

The equivalence \eqref{gwioeuorgfewrgwergre} restricts to an equivalence
\begin{equation}\label{gwioeuorgfewrgwrrrrergre}\ell\bV^{\ctr}_{\bC}\stackrel{\simeq}{\to} \ell\bV^{\{e\},\ctr}_{\bC}:BG\times \BC\to  \Add_{2,1}
\end{equation} 
Applying \eqref{hrthertgertgertgetrg}, taking $\lim_{BG}$ and restoring $u$-continuity we conclude:
\begin{kor}\label{lkprthergretgtr} We assume that $\bC$ admits all AV-sums.
We have a commutative square
\begin{equation}\label{gwerfwerfgtrg425}\xymatrix{ K\cX^{\ctr,hG}_{\bC}\ar[rr]^{\rmch^{hG}}\ar[d]^{\simeq} &&\PCH^{hG}_{\bC} \ar[d]^{\simeq} \\ K\cX^{\{e\},\ctr,hG}_{\bC}\ar[rr]^{\rmch^{\{e\},hG}} &&\PCH_{\bC}^{\{e\},hG} } 
\end{equation}
\end{kor}
The vertical equivalences in \eqref{gwerfwerfgtrg425} are induced by \eqref{gwioeuorgfewrgwrrrrergre}, and the horizontal maps by  \eqref{hrthertgertgertgetrg}.

\subsection{Equivariant periodic coarse homology}

In the section we recall the construction of the equivariant coarse ordinary homology theory, introduce its periodization and its Borelification.

We first recall the construction of the equivariant coarse ordinary homology theory from \cite[Sec. 7]{equicoarse}.
Let $k$ be a commutative ring and $A$ be a $k$-module.
 We consider the functor
 \begin{equation}\label{iugioejorfcerfwerf}C\cX^{G}(-,A):G\BC\to \Ch(k)
\end{equation} 
 which associates to $X$ in
  $G\BC$ the chain-complex 
 $(C\cX^{G}(X,A),\partial)$  of  $G$-invariant locally finite  controlled $A$-valued chains on $X$.
 An element $\phi$ in $C\cX^{G}_{n}(X,A)$ is a function
 $\phi:X^{n+1}\to A$ which is
 \begin{enumerate}
 \item $G$-invariant: $\phi(gx_{0},\dots,gx_{n})=\phi(x_{0},\dots,x_{n})$ for all $g$ in $G$ and $(x_{0},\dots,x_{n})$ in $X^{n+1}$.
 \item locally finite: For every bounded subset $B$ of $X$  and $i$ in $\{0,\dots ,n\}$ the set
 $\pr_{i}(\supp(\phi)) \cap B$ is finite.
 \item controlled: There exists an entourage $U$ in $\cC_{X}^{G}$ such that $\phi$ is $U$-controlled, i.e., 
 for all $(x_{0},\dots,x_{n})$ in $\supp(\phi)$ and all $i,j$ in $\{0,\dots,n\}$ we have
 $(x_{i},x_{j})\in U$.
 \end{enumerate}
 Here $\supp(\phi):=\{(x_{0},\dots,x_{n})\in X^{n+1}\mid \phi(x_{0},\dots,x_{n})\not=0\}$ is the support of $\phi$,
 and $\pr_{i}:X^{n+1}\to X$ is the projection onto the $i$'th component.
 The differential $\partial:C\cX_{n}^{G}(X,A)\to C\cX_{n-1}^{G}(X,A)$ is given by
 $$\partial \phi:=\sum_{i=0}^{n}(-1)^{i}\partial_{i}\ ,$$
 where
 $$(\partial_{i}\phi)(x_{0},\dots,x_{n-1}):=\sum_{x\in X}\phi(x_{0},\dots,x_{i-1},x,x_{i},\dots,x_{n-1})$$
 takes the sum over the fibre of $\pr_{i}$.
 If $\phi$ is $U$-controlled, then
 the summand for $x$ is non-zero only if $x\in \bigcap_{i=0}^{n-1} U[x_{i}]$. Since
 this set if bounded and $\phi$ is locally finite the sum is finite.
For $X$ in $G\BC$ we have  an isomorphism  \begin{equation}\label{gwerfwervsfvfd}C\cX^{G}(X,A)\cong \colim_{U\in \cC_{X}^{G}} C\cX_{U}^{G}(X,A)\ ,
\end{equation}  where $C\cX_{U}^{G}(X,A)$ is the subcomplex of $C\cX^{G}(X,A)$ of $U$-controlled chains.

Let \begin{equation}\label{bsdfpojkvopsdfvsdfvsdfvfdv}\ell:\Ch_{k}\to   \Mod(Hk)
\end{equation}be the Dwyer-Kan localization at quasi-isomorphisms.
\begin{theorem}[{\cite[Thm. 7.3]{equicoarse}}]
The functor \begin{equation}\label{nfgbgfbddgfber}\rmH\cX^{G}(-,A):=\ell C\cX^{G}(-,A):G\BC\to \Mod(Hk)\end{equation} 
is an equivariant coarse homology theory.
\end{theorem}

The functor $\rmH\cX^{G}(-,A)$ is called the equivariant coarse ordinary homology theory  with coefficients in $A$. It is furthermore 
   continuous by \cite[Lem. 7.9]{equicoarse}, 
  strong by \cite[Lem. 7.10]{equicoarse}, and 
 strongly additive by \cite[Lem. 7.11]{equicoarse}.

%


Recall the   \cref{iugwierogjwe8r9uf98werfwerfwerwerf} of $2$-periodization.
\begin{ddd}\label{kopgherthertgertgertget}
We define the equivariant  coarse periodic  homology theory with coefficients in $A$ by 
$$ \PH\cX^{G}(-,A):= P(\rmH\cX^{G}(-,A))
:G\BC\to \Mod(Hk)\ .$$ 
\end{ddd}

By \eqref{bsdfvdfvsfvsfdvs} for $X$ in $G\BC$ we have
\begin{equation}\label{gwerferfwfer}\PH\cX^{G}(X,A)\simeq \colim_{U\in \cC^{G}_{X}} \prod_{k\in \nat} \Sigma^{2k}\rmH\cX^{G}(X_{U},A)\ .
\end{equation}  
\begin{lem}\label{kopethtregtet}
We have $$\PH\cX^{G}(X)\simeq \ell  \colim_{U\in \cC_{X}^{G}}  \prod_{k\in \nat}  \colim_{n\in \nat }C\cX^{G}_{U^{n}} (X,A)[2k]\ .$$
\end{lem}
\begin{proof}
We use  \eqref{gwerfwervsfvfd}, 
\eqref{gwerferfwfer} and the fact that $(U^{n})_{n\in \nat}$ is cofinal in $\cC^{G}_{X_{U}}$ provided $U$ conains the diagonal.
We also use that $\ell$ commutes with filtered colimits and products.
\end{proof}

The order of the product and the colimits in this formula is relevant.
Using \cref{kopethtregtet} one can show:
\begin{prop}
$\HP\cX^{G}(-,A)$ is strong.
\end{prop}
\begin{proof}
As this fact is not needed in the present paper and
the proof requires to unfold   definitions not recalled here  we leave it as an exercise to the reeader.
%
%
%
%
\end{proof}

It has been shown in \cite[Sec. 3.3]{trans} that the equivariant coarse ordinary homology theory
has an extension
\begin{equation}\label{huihgiwehgerfwerfwerf}\xymatrix{G\BC\ar[dr]\ar[rr]^{\rmH\cX^{G}(-,A)}&&\Mod(Hk)\\&G\BC_{\tr}\ar[ur]_{\rmH\cX^{G}_{\tr}(-,A)}&}\ .
\end{equation}
This extension induces an an extension $\PH\cX^{G}_{\tr}(-,A)$ of the periodic version.
We can therefore define the Borel-equivariant version
$$\PH\cX^{hG}(-,A):G\BC\to \Mod(Hk)$$
and the Borelification map
\begin{equation}\label{bsldkjvopsdfvsfdvsfdv}\beta:\PH\cX^{G}(-,A)\to \PH\cX^{hG}(-,A):G\BC\to \Mod(Hk)
\end{equation}
according to \cref{jgoijowerfewrfewfwfwef}.

We finally make the points from \cref{biojrgogbfgbdb} explicit.

\begin{prop}
We have a canonical equivalence
\begin{equation}\label{bsdfvsdfvsfdvsfdvsfv}\PH\cX(-,A)\simeq \PH\cX^{\{e\}}(-,A):BG\times \BC\to \Mod(Hk)
\end{equation}
  under which the Borelification maps $\beta$ and $\beta'$ become equivalent, i.e., 
 \begin{equation}\label{bsdifub980u90bdfgb}\xymatrix{&\PH\cX^{G}(-,A)\ar[dr]^{\beta'}_{\eqref{vsdfokvpsosdfv}}\ar[dl]^{\eqref{bsldkjvopsdfvsfdvsfdv}}_{\beta}&\\ \PH\cX^{hG}(-,A)\ar[rr]_{\eqref{bsdfvsdfvsfdvsfdvsfv}^{hG}}^{\simeq}&&\PH\cX^{\{e\},hG}(-,A)}
\end{equation}  
commutes.
\end{prop}
\begin{proof}
The map $\beta'$ will be construced in the proof below.
For  $X$  in $\BC$ we define a natural isomorphism
\begin{equation}\label{bdfsoijoiwgsdbdfb}C\cX(X,A)\stackrel{\cong}{\to} C\cX^{G}(X\otimes G_{min,min},A)
\end{equation}
which sends a chain
$\phi$  in $C\cX_{n}(X,A)$  to the chain
$$((x_{0},g_{0}),\dots,(x_{n},g_{n}))\mapsto \left\{\begin{array}{cc} \phi(x_{0},\dots,x_{n})&g_{0}=\dots =g_{n}\\0 &else  \end{array} \right.$$
in $C\cX^{G}_{n}(X\otimes G_{min,min},A)$. Applying $\ell$ from \eqref{bsdfpojkvopsdfvsdfvsdfvfdv}
this isomorphism induces the equivalence
\begin{equation}\label{jbiosdfjoivdsfvsdfvsdvdfvsdfv}\rmH \cX(-,A)\simeq \rmH\cX^{\{e\}}(-,A):BG\times \BC\to \Mod(Hk)\ ,
\end{equation} where $G$ acts trivially on both sides.
Periodization gives the equivalence \eqref{bsdfvsdfvsfdvsfdvsfv}.

For $X$ in $G\BC$ we have  canonical inclusion
\begin{equation}\label{bsodjkfpvsdfvsfdvdf}\underline{C\cX^{G}(X,A)}\to C\cX(\hat \Res^{G}(X),A)
\end{equation} of chain complexes with $G$-action.
Applying $\ell$ and taking adjoints we get the Borelification map
$$\tilde \beta:\rmH\cX^{G}(-,A)\to \rmH\cX^{hG}(-,A):G\BC\to \Mod(Hk)\ .$$
Its periodization finally induces
 \begin{equation}\label{vsdfokvpsosdfv} \beta':\PH\cX^{G}(-,A)\to \PH\cX^{hG}(-,A):G\BC\to \Mod(Hk)\ .
\end{equation}
The following triangle commutes:
$$\xymatrix{&\underline{C\cX^{G}(-,A)}\ar[dl]_{\tr} \ar[dr]^{\eqref{bsodjkfpvsdfvsfdvdf}} &\\ C\cX^{G}(\hat \Res^{G}(-)\otimes G_{min,min},A) &&\ar[ll]^{\eqref{bdfsoijoiwgsdbdfb}}_{\cong}C\cX^{\{e\}}(\hat \Res^{G}(-),A)  }\ .$$
Here $\tr$ is the chain level transfer inducing the transfer \eqref{bsfdwergwregwergwevodfsdbdf}  in equivariant coarse ordinary homology. We apply $\ell$, take adjoints,   and apply $P(-)$ and  in order to get \eqref{bsdifub980u90bdfgb}.
\end{proof}

\subsection{The trace transformation}

In this section we introduce the trace transformation from equivariant coarse periodic cyclic homology  
to equivariant periodic ordinary coarse homology theory. We further discuss its compatibility with Borelification.

Recall our standing assumption that  $\bC$   in $G\nCcat$ is effectively additive and admits countable $AV$-sums. 
In the present subsection   we assume in addition  that the full subcategory $\bC^{u}$ of unital objects admits a $G$-invariant trace $\tau_{\bC}=(\tau_{C})_{C\in \Ob(\bC^{u})}$, see \cite[Def. 9.1 \& 9.8]{Bunke:2025aa}. The trace associates to
every
object $C$ in $\bC^{u}$ a $\C$-linear map
$$\tau_{C}:\End_{\bC}(C)\to \C$$ such that
for every two morphisms $f:C\to C'$ and $f':C'\to C$
between objects of $\bC^{u}$
we have
$\tau_{C'}(f\circ f')=\tau_{C}(f'\circ f)$ (trace property), and for every $g$ in $G$ and $h$ in $\End_{\bC}(C)$ we have
$\tau_{C}(h)=\tau_{gC}(gh)$ (invariance).

\begin{ex}
The category $\Hilb_{c}(\C)^{u}$ with the trivial action has a $G$-invariant trace given by the usual matrix trace.
More generally, if $A$ is a unital  $G$-$C^{*}$-algebra with a $G$-invariant trace, then $\Hilb_{c}(A)^{u}$ has an induced  $G$-invariant trace. \hB
\end{ex}

In this section we construct a   transformation of equivariant coarse homology theories 
\begin{equation}\label{wregwerfvdfs}\tau^{G}:\PCH\cX^{G}_{\bC}\to \HP\cX^{G}(-,\C):G\BC\to \Mod(H\C)\ . 
\end{equation}
It is a modification of the trace transformation from equivariant  coarse Hochschild or cyclic homology to coarse homology
constructed in \cite{caputi-diss}.   Variants also appeared, e.g.,  in \cite{Engel:2025aa}, \cite{Ludewig:2025aa}.

%
%
%
%


\begin{construction}\label{ojorherthgetrgtr}{\em
 The trace transformation will be determined by a chain-level construction. 
Recall the explicit description of the periodic cyclic homology complex of a mixed complex from 
\cref{koprethertgetgegtrg} and the description of the mixed complex of a $k$-linear category from \cref{kwgopergweferfwef}.
The trace transformation \eqref{wregwerfvdfs} will be determined by the natural  chain map 
$$\tilde \tau^{G}: \colim_{U\in \cC_{X}^{G}}\prod_{k\in \Z}  \rmMix_{\C}(\bV_{\bC}^{G,\ctr}(X_{U}))[2k]\to \colim_{U\in \cC^{G}_{X}} \prod_{k\in \Z} \colim_{n\in \nat} C\cX_{U^{n}}^{G}(X,\C)[2k]\ ,$$ see \cref{kopgewerfrefewf} for the domain and
 \cref{kopethtregtet} for the target, where the $\Z$-graded group in the domain has the differential of the periodic 
 cyclic homology complex.
By definition, the map  $\tilde \tau^{G}$ will  preserve the factors of the products.

Let $U$ in $\cC_{X}^{G}$ contain the diagonal.
We consider the map between the  factors with index $k$. Let $(f_{0}\otimes\dots \otimes f_{n+2k})$ be in $ \Mix(\bV_{\bC}^{G,\ctr}(X_{U}))[2k]_{n}$. For every pair of points $x,y$ in $X$ the matrix coefficient $f_{i,x,y}=\mu(\{x\})f_{i}\mu(\{y\})$ is a morphism in $\bC^{u}$.
Hence the whole composition below is a morphism in $\bC^{u}$ and its trace is defined.
We define
$$\tilde \tau^{G}(f_{0}\otimes\dots \otimes f_{n+2k})(x_{0},\dots,x_{n+2k}):=\tau_{A_{n+2k}}(f_{0,x_{n+2k},x_{0}} \circ \dots  \circ f_{n+2k,x_{n-1+2k},x_{n+2k}})\ .$$
Using that $\tau_{\bC}$ is $G$-invariant one checks that the function $\tilde \tau^{G}(f_{0}\otimes \dots \otimes f_{n+2k})$ on $X^{n+2k+1}$ is $G$-invariant.  

If $n_{i}$ in $\nat$ is such that $f_{i}$ is $U^{n_{i}}$-controlled for all $i$ in $\{0,\dots,n+2k\}$, then we have 
$\tilde \tau^{G}(f_{0}\otimes\dots \otimes f_{n+2k})\in C\cX^{G}_{U^{\sum_{i=0}^{n+2k}n_{i}}}(X,\C)$.   In order to see this assume that $j<l$ and that 
  $(x_{j},x_{l})\not\in U^{\sum_{i=0}^{n+2k}n_{i}}$. Then
  $f_{j+1,x_{j},x_{j+1}}\circ \dots \circ  f_{l,x_{l-1},x_{l}}=0$ since this composition propagates at  most as  $U^{\sum_{i=j}^{l}n_{i}}$ and $U^{\sum_{i=j}^{l}n_{i}}\subseteq U^{\sum_{i=0}^{n+2k}n_{i}}$ since
  $\sum_{i=j}^{k}n_{i}\le \sum_{i=0}^{n+2k}n_{i}$.

We check local finiteness.
If $B$ is a bounded subset of $X$, then for every $i$ in $I$ the set of $(x_{0},\dots,x_{i-1},\hat x_{i},x_{i+1},\dots,x_{n+2k}) $ in $X^{n+2k}$ such that $$\tilde \tau^{G}(f_{0}\otimes\dots \otimes f_{n+2k})(x_{0},\dots,x_{i-1},x,x_{i+1},x_{n+2k})\not=0$$ and $x\in B$  is finite since
only  tuples $(x_{0},\dots,x_{i-1},x_{i+1},x_{n+2k})$ with $x_{j}\in U^{\sum_{i=0}^{n+2k}n_{i}}[B]$  for all $j\in \{0,\dots,n+2k\}\setminus \{i\}$  can contribute non-trivially.

Naturality is an straightforward calculation.

One checks that $\tau\circ d_{i}=\partial_{i}\circ \tau$.
This implies that
$\tau\circ d=d\circ \tau$.
Using the cyclicity of the trace one checks that $\tau\circ B=0$.
This implies that
$\tau$ is a chain map.} \hB
\end{construction}

\begin{ddd}
We call $\tau^{G}$ from \eqref{wregwerfvdfs} the trace transformation.
\end{ddd}

We now state the compatibility with the Borelification:

\begin{prop}\label{okhopekrtgertgethehrtgertg} We assume that $\bC$ admits all AV-sums.
The following square commutes:
\begin{equation}\label{gwergwerg2345gtw}\xymatrix{ \PCH\cX^{G}_{\bC}\ar[r]^{\tau^{G}}\ar[d]_{\eqref{werfwerfwerf}}^{\beta} & \PH\cX^{G}(-,\C) \ar[d]_{\eqref{bsldkjvopsdfvsfdvsfdv}}^{\beta} \\   \PCH \cX_{\bC}^{hG}(-,\C)\ar[r]^{\tau^{hG}} &\PH^{hG}(-,\C) } 
\end{equation}
\end{prop}
\begin{proof}
One checks by unfolding definitions that
the trace map $\tau^{G}$ extends to a transformation
$$\tau^{G}_{\tr}: \PCH\cX^{G}_{\bC,\tr}\to  \PH \cX^{G}_{\tr}(-,\C):G\BC_{\tr}\to \Mod(H\C)\ ,$$
see \eqref{huihgiwehgerfwerfwerf} for the target and \eqref{iozgweuirfwerfwerfwref} for the domain.
To this end one must check that the chain complex level construction 
is compatible with transfers along bounded coarse coverings and that $2$-morphisms of spans act as equalities.
The assertion then follows from \cref{ogjwioergjeworijuoiewferfwwrf}.
\end{proof}

\begin{prop}\label{oiuehrgweroferferwfweferfwrdb} We assume that $\bC$ admits all AV-sums.
The square
\begin{equation}\label{}\xymatrix{ \PCH\cX^{hG}_{\bC}\ar[r]^{\tau^{hG}}\ar[d]^{\eqref{gjiweorjofwewerfwerfw111}^{hG}} &  \PH\cX^{hG}(-,\C)\ar[d]^{\eqref{jbiosdfjoivdsfvsdfvsdvdfvsdfv}^{hG}} \\  \PCH\cX^{\{e\},hG}_{\bC} \ar[r]^{\tau^{\{e\},hG}} & \PH\cX^{\{e\},hG}(-,\C)} 
\end{equation}
commutes.
\end{prop}
\begin{proof}
One checks commutativity on  the chain level before applying $\ell$ and $\lim_{BG}$.
\end{proof}

%
%
%
 
\subsection{Equivariant  Borel-Moore homology theories}\label{jiooggwerweferwf}

Let   
$$E:G\LCH^{+,\op}\to \cC$$ be  a functor to a cocomplete  stable target.
We will consider the following properties   such a functor could have
 \begin{ddd} \label{ehpzhrtgtrg} \mbox{}
 \begin{enumerate}
 \item homotopy invariance: The map $E (X)\to E([0,1]\times X)$ induced by the projection $[0,1]\times X\to X$ is an equivalence for every $X$ in $G\LCH^{+}$.
 \item\label{lkopgrthertgrtgerrrr} strong excision: $E(\emptyset)\simeq 0$ and
 for every $X$ in $G\LCH^{+}$ and  inclusion $i:Y\to X$ of a closed invariant subspace with complement $j:U\to X$
 \begin{equation}\label{giwjeriojfowerfwerfwref}E(U)\xrightarrow{X\stackrel{j}{\supseteq} U\stackrel{\id_{U}}{\to}U} E(X)\xrightarrow{i} E(Y)
\end{equation} is  a fibre sequence.
 \item\label{ojohpertgrtgtrget} prodescent:  For every cofiltered family $(X_{i})_{i\in I}$ in $G\CH$
  the canonical map
 $$ \colim_{i\in I^{\op}} E(X_{i})\stackrel{\simeq}{\to}   E(\lim_{i\in I} X_{i}) $$ is an equivalence.
   \end{enumerate}
    \end{ddd}
    
    Note that in the situation of  \cref{ehpzhrtgtrg}.\ref{lkopgrthertgrtgerrrr}. we have a factorization
   $$\xymatrix{E(U)\ar[r]\ar[dr]_{ \emptyset\to U} &E(X)\ar[r]&E(Y)\\&E(\emptyset)\ar[ur]_{  Y\supseteq \emptyset\to \emptyset}&}$$
   showing that the composition \eqref{giwjeriojfowerfwerfwref} is canonically zero.
   In \cref{ehpzhrtgtrg}.\ref{ojohpertgrtgtrget}. we consider the category $G\CH$ of compact Hausdorff $G$-spaces and equivariant continuous (everywhere defined) maps  as a non-full subcategory of $G\LCH^{+}$.
    
    Let $E:G\LCH^{+,\op}\to \cC$ be a functor to a cocomplete  stable target.
\begin{ddd}\label{kophertgrtgrtgerrtgertg}
The functor $E$ is an equivariant  Borel-Moore cohomology theory if it is homotopy invariant, strongly excisive and satisfies prodescent.
\end{ddd}

    Let $\phi:\cC\to \cD$ be a colimit-preserving functor between cocomplete stable $\infty$-categories.
    
    \begin{kor}\label{okiophkeprthrtgretgertg}
    If $E:G\LCH^{+,\op}\to \cC$ is an equivariant Borel-Moore cohomology theory, then so is
    $\phi\circ E:G\LCH^{+,\op}\to \cD$.
    \end{kor}

The target of  the equivariant $E$-theory functor \eqref{hertgrtgtergbgd}
 is cocomplete and stable. 
 \begin{ddd}\label{okopkgpowrtgertgetrg}
 We define the   functor \begin{equation}\label{gerwfreferfwefwer}\ee^{G}C_{0}:G\LCH^{+,\op}\stackrel{C_{0}}{\simeq}G \nCalg_{\comm}\to G\nCalg\xrightarrow{\ee^{G}}\EE^{G}\ .
\end{equation} 
 \end{ddd}

\begin{lem} \label{kopherthertgertgertr}The functor
$\ee^{G}C_{0}:G\LCH^{+,\op}\to  \EE^{G}$ 
is an equivariant Borel-Moore cohomology theory. \end{lem}\begin{proof}
The functor $\ee^{G}C_{0}$
  is homotopy invariant
   since $C_{0}$ translates homotopies in $G\LCH^{+}$ to homotopies in $G\nCalg$ and the functor   $\ee^{G}$  is homotopy invariant.
   
   We have $C_{0}(\emptyset)\cong 0$ and hence $e^{G}C_{0}(\emptyset)\simeq 0$.
    If $U$ is an invariant open subset of $X$ in $G\LCH^{+}$ with closed complement $Y$, then we have an exact sequence
$$0\to C_{0}(U)\to C_{0}(X)\to C_{0}(Y)\to 0$$
in $G\nCalg_{\comm}$,
where the first map is the extension by $0$ and the second is the restriction to $Y$.
By exactness of $\ee^{G}$ we get the desired fibre sequence
$$\ee^{G}C_{0}(U)\to \ee^{G}C_{0}(X)\to \ee^{G}C_{0}(Y)\ .$$
Finally, if $(X_{i})_{i\in I}$ is a cofiltered system in $G\CH$ with $X\cong \lim_{i\in I}X_{i} $, then we get by Gelfand duality an isomorphism
$$\colim_{i\in I^{\op}} C_{0}(X_{i})\stackrel{\cong} {\to} C_{0}(X)$$
in $G\nCalg_{\comm}$.
Since $\ee^{G}$ preserves filtered colimits we conclude that
$$\colim_{i\in I^{\op}} \ee^{G}C_{0}(X_{i}) \stackrel{\cong} {\to} \ee^{G}C_{0}(X)\ .$$
\end{proof}
By the same argument as above,  the functor $\ee^{G}C_{0}$ actually sends all cofiltered limits in $G\LCH^{+}$ to colimits.


Let $E:G\CH^{\op}\to \cC$ be a functor. The notion of homotopy invariance and prodescent for $E$  are defined  as above.
In order to define the notion of an   equivariant   Borel-Moore cohomology for functors defined on $G\CH^{\op}$ we must reformulate the excision axiom.  
 A pair $(X,Y)$ in $G\CH$ is an object $X$ in $G\CH$ with a closed invariant subspace $Y$. We set 
\begin{equation}\label{werfwerfwerfewe}E(X,Y):=\Fib(E(X)\to E(Y))\ .
\end{equation}
 A map of pairs $f:(Z,W)\to (X,Y)$ is a map $f:Z\to X$ in $G\CH$ with $f(W)\subseteq Y$.
 It  naturally  induces a map
 $E(X,Y)\to E(Z,W)$.

  \begin{ddd}
$E$ is  invariant under relative homeomorphisms if 
 for every map   $f:(Z,W)\to (X,Y)$  of pairs in $G\CH$ that  induces a homeomorphism $Z\setminus W\to X\setminus Y$ the induced map
 $ E(X,Y)  \to E(Z,W)$  is an equivalence.
\end{ddd}

Let $E:G\CH^{\op}\to \cC$ be a functor to a cocomplete stable target.
\begin{ddd}
The functor $E$ is   a compact   equivariant  Borel-Moore cohomology if it is homotopy invariant,  invariant under relative homeomorphisms and satisfies prodescent.
\end{ddd}

  We let $\Fun^{\BM}(G\LCH^{+,\op},\cC)$ (or  
$\Fun^{\BM}(G\CH^{\op},\cC)$) denote the full subcategories of 
$\Fun (G\LCH^{+,\op},\cC)$ (or  
$\Fun (G\CH^{\op},\cC)$) of equivariant (compact)
 Borel-Moore cohomology theories.

Let $\cC$ be a cocomplete stable $\infty$-category and consider the restriction
$r:\Fun (G\LCH^{+,\op},\cC)\to \Fun (G\CH^{\op},\cC)$ along the inclusion $G\CH\to G\LCH^{+}$.
\begin{lem}\label{hjrtigjiortgtrgertg}
The restriction functor 
 induces equivalence $$r: \Fun^{\BM}(G\LCH^{+,\op},\cC)\stackrel{\simeq}{\to} \Fun^{\BM}(G\CH^{\op},\cC)\ .$$   
\end{lem}
\begin{proof}
We let $G\CH_{*/}$ denote the category of pointed compact Hausdorff $G$-spaces. Then
we have an equivalence
of categories
\begin{equation}\label{fgqefdqewdqwedqewd}G\CH_{*/}\stackrel{\simeq}{\to}G \LCH^{+}\ .
\end{equation}
The functor in \eqref{fgqefdqewdqwedqewd} sends the object $(X,*_{X})$ in $G\CH_{*/}$ to the object  $X\setminus *_{X}$ in $G\LCH^{+}$, and  the morphism $f:(X,*_{X})\to (Y,*_{Y})$ in $G\CH_{*/}$  to the morphism
$$X\setminus *_{X}\supseteq X\setminus f^{-1}(*_{Y})\xrightarrow{f_{|X\setminus f^{-1}(*_{Y})}} Y\setminus *_{Y}$$
in $G\LCH^{+}$.

The inverse functor sends the object 
$X$ in $G\LCH^{+}$ to the one-point compactification $(X_{+},+)$ in $G\CH_{*/}$,   and the morphism
$X\supseteq U\xrightarrow {f}Y$ in $G\LCH^{+}$ to the map $(X_{+},+_{X})\to (Y_{+},+_{Y})$ in $G\CH_{*/}$
that sends all points of $X_{+}\setminus U$ to $+_{Y}$ and points $u$ in $U$ to $f(u)$.

It is clear that the restriction sends  equivariant 
Borel-Moore cohomology theories to
compact  equivariant  Borel-Moore cohomology theories.
We have a functor  
$$\phi:\Fun(G\CM^{\op},\cC)\to \Fun(G\CM^{\op}_{*/},\cC)\ , \quad E\mapsto ( (X,*)\mapsto E(X,*)) \ .$$
Using  \eqref{fgqefdqewdqwedqewd}  for the second equivalence we form the functor
$$s:\Fun(G\CM^{\op},\cC)\stackrel{\phi}{\to}   \Fun(G\CM^{\op}_{*/},\cC)\simeq  \Fun(G\LCH^{+,\op},\cC)\ .$$ 
In the following we exhibit $s$ as an inverse of $r$ at the level of equivariant Borel-Moore
cohomology theories.

We have a natural transformation of functors
$$(-)_{+}\to \id:G\LCH^{+}\to G\LCH^{+}\ , \qquad X\mapsto (X^{+}\supseteq X\stackrel{\id_{X}}{\to}X)\ .$$
It induces a natural transformation
of functors
\begin{equation}\label{gerwg524ffwerf}E\mapsto ( E\to E((-)_{+})):\Fun(\LCH^{+,\op},\cC)\to \Fun(\LCH^{+,\op},\cC)\ .
\end{equation} We have a natural commutative  diagram
$$\xymatrix{+\ar[rr]\ar[drr]^{+\supseteq \emptyset\to \emptyset}&&X_{+}\ar[rr]^{X_{+}\supseteq X\stackrel{\id_{X}}{\to}X} &&X\\&&\emptyset\ar[urr]&&}$$
which  after applying    a functor $E:G\LCH^{+}\to \cC$ with $E(\emptyset)\simeq 0$ (we say that $E$ is reduced)  gives a natural commutative diagram
$$\xymatrix{E(X)\ar[rr]\ar[drr]&&E(X_{+})\ar[rr]&&E(+) \\&&0\ar[urr]&&}\ .$$
Consequently,  the transformation \eqref{gerwg524ffwerf}
has a canonical factorization over a transformation
$$E\to (E\to   s(r(E)))$$ for reduced functors $E$.
 On functors  $E$ which in  addition satisfy  strong excision this transformation is an equivalence.
 
On functors  $E:G\CH^{\op}\to \cC$ which are  invariant under relative homeomorphisms we have an equivalence
$r(s(E))\to E$ given by
$$E(X_{+},+) \simeq \Fib( E(X)\oplus E(+)\to E(+))\stackrel{\simeq}{\to} E(X)\ .$$

%
%

If $E:G\CH^{\op}\to \cC$ satisfies prodescent and is homotopy invariant, then   also $s(E)$ has these properties, too.  We now assume that $E$ is invariant under relative homeomorphisms and  show that $s(E)$ satisfies strong excision.  Let
$i:Y\to X$ be an inclusion of a closed invariant subspace with complement $j:U\to X$.
Then we get a map of pairs
$(X_{+},Y_{+})\to (U_{+},+)$ 
  which induces a homeomorphism of complements. By invariance under relative homeomorphisms we get an equivalence
  $$  E(U_{+},+)   \stackrel{\simeq}{ \to}  E(X_{+},Y_{+}) \ .$$    Furthermore,
its  target   is the total fibre
  of $$ \xymatrix{E(X_{+})\ar[r]\ar[d]&E(Y_{+})\ar[d]\\E( +)\ar[r]&E(+)}$$
  and therefore also the fibre of $s(E)(X)\to s(E)(Y)$, while  its domain is by definition $s(E)(U)$.
\end{proof}

 We now turn attention to  covariant functors.
Let $E:G\LCH^{+}\to \cC$ be some functor to a complete  stable $\infty$-category. 
We consider the following list properties
$$\bP=\{homotopy \:invariant, strongly\: excisive, prodescent, Borel-Moore\: homology \}\ .$$
\begin{ddd}\label{opekhrthgrgertgertg}
A functor $E:G\LCH^{+}\to \cC$   has property $P$ in $\bP$ if 
$E^{\op}:G\LCH^{+,\op}\to \cC^{\op}$ has the property $P$.
\end{ddd}

In constructions it is sometimes difficult to check the property of prodescent. In connection with coarse homology theories
the weaker condition of local finiteness from \cref{jkopgwegergferfwr} is sufficient which 
motivates \cref{kohperthrgrtgertgeg}.

Let $X$ be in $G\LCH^{+}$ and $\cW_{X}$ be the family of open invariant relatively $G$-compact  subsets, i.e., open invariant subsets $W$ such that there exists a compact subset $K$ with $W\subseteq GK$.
If $W,W'$ are in $\cW$ and  $W\subseteq W'$, then we have an inclusion morphism
$X\setminus W' \to X\setminus W$ in $G\LCH^{+}$.
\begin{ddd}\label{jkopgwegergferfwr}
The functor $E:G\LCH^{+}\to \cC$ is called  locally finite if for  every $X$ in $G\LCH^{+}$ we have
  $$ \lim_{W\in \cW_{X}}  E(X\setminus W) \simeq 0\ .$$
 \end{ddd}

\begin{ddd}\label{kohperthrgrtgertgeg}
A functor
$E:G\LCH^{+}\to \cC$ is called a weak equivariant   Borel-Moore homology theory if it is  homotopy invariant, strongly excisive
and  locally finite.
\end{ddd}

%

 If $(X,Y)$ is a pair in $G\CH$ and $E:G\CH\to \cC$ is a functor to a stable $\infty$-category, then  dually to \eqref{werfwerfwerfewe} we set
 \begin{equation}\label{woeirjojfnwkjerfwerfwr}E(X,Y):=\Cofib(E(Y)\to E(X))\ .
\end{equation}

\begin{lem}\label{ophertgertgertgertg}
An equivariant  Borel-Moore homology theory is a weak equivariant   Borel-Moore homology theory.
  \end{lem}
\begin{proof}
We must show that an equivariant  Borel-Moore homology theory is  locally finite.
Let $X$ be in $G\LCH^{+}$.
Then
$$  (+,+)\cong \lim_{W\in \cW_{X}}  ((X\setminus W)_{+},  +)   $$
in $G\CH$.
Using strong excision for the first equivalence and prodescent for the second we get
$$
 \lim_{W\in \cW_{X}}  E(X\setminus W) \simeq \lim_{W\in \cW_{X}}   E((X\setminus W)_{+},+)\simeq
 E( \lim_{W\in \cW_{X}} (X\setminus W)_{+} ,+)\simeq
  E(+_,+ )\simeq 0\ .
$$
\end{proof}

\begin{ex}\label{iuiqofujoewfqwefeqwdfq}
We consider the symmetric monoidal $\infty$-category $\Pr^{L}_{\st}$ of presentable stable $\infty$-categories.
Let $\cR$ be a symmetric monoidal presentable $\infty$-category, i.e., an object in $\CAlg(\Pr^{L}_{\st})$, and assume that $\cC$ is in $\Mod_{\Pr^{L}_{\st}}(\cR)$. Then  $\cC$ is enriched in $\cR$ and
 we have a mapping object functor
$$\cC(-,-):\cC^{\op}\times \cC\to \cR\ .$$
For fixed $C$ in $\cC$ the functor $\cC(-,C)$ is exact and sends colimits to limits.  
If $$E:G\LCH^{+,\op}\to \cC$$ is an equivariant  Borel-Moore  cohomology theory, then for any object $C$ of $\cC$
$$\cC(E,C):G\LCH^{+}\to \cR
$$
is an equivariant  Borel-Moore  homology theory.
\hB
\end{ex}

The category $\EE^{G}$ belongs to $\Mod_{\Pr^{L}_{\st}}(\Mod(KU))$.
For $A$ in $\EE$ and $\bC$ in $G\nCcat$, by \cref{kophertgrgrgerg} and \cref{okopkgpowrtgertgetrg} we have 
\begin{equation}\label{pojgkpowergewrweferfr}K^{G,\an}_{\bC,A}\simeq \EE^{G}(\ee^{G}C_{0},\bC^{(G)}_{\std}\otimes A):G\LCH^{+}\to \Mod(KU)\ .
\end{equation}In view of \cref{kopherthertgertgertr} and \cref{iuiqofujoewfqwefeqwdfq} we conclude:
\begin{kor}\label{ogpwerferfewrfwef}
$K^{G,\an}_{\bC,A}$ is an equivariant  Borel-Moore homology theory.
\end{kor}

  In the case of the trivial group we have a classification of Borel-Moore cohomology theories and a systematic construction of Borel-Moore homology theories.
 We let $\Fun^{\BM}(\LCH^{+,\op},\cC)$  denote the full subcategory of $\Fun(\LCH^{+,\op},\cC)$ of Borel-Moore cohomology theories.
 The following is a consequence of \cref{hjrtigjiortgtrgertg} and  a result of D. Clausen (see \cite[Thm. 3.6.13]{NKP}).
 
 Since $\Pr^{L}_{\st}$ is symmetric monoidal we can talk about dualizable presentable stable $\infty$-categories.
 \begin{theorem}\label{gojertpgwerfwerfwerf}
 If $\cC$ is a dualizable presentable stable $\infty$-category, then the 
 evaluation at $*$ induces an equivalence
 \begin{equation}\label{rweferjflewrkfer}\Fun^{\BM}(\LCH^{+,\op},\cC)\to \cC\ , \quad E\mapsto E(*)\ .
\end{equation}
 \end{theorem}
\begin{proof}
The evaluation map has a factorization
\begin{equation}\label{gwrejgpregrtetgwg}\Fun^{\BM}(\LCH^{+,\op},\cC)\stackrel{r}{\simeq}\Fun^{\BM}(\CH^{\op},\cC)\stackrel{!,\simeq}{\to} \Fun^{\BM'}(\CM^{\op},\cC)\stackrel{E\mapsto E(*)}{\simeq} \cC
\end{equation}where $ \Fun^{\BM}(\CH^{\op},\cC)$ is the full subcategory of $ \Fun(\CH^{\op},\cC)$ compact Borel-Moore cohomology theories, and $\BM'$ stands for the full subcategory of $ \Fun(\CH^{\op},\cC)$ 
of   functors which are invariant under relative homeomorphisms  and satisfy prodescent (in contrast to Borel-Moore cohomology we do not require homotopy invariance).
The first map is an equivalence by  \cref{hjrtigjiortgtrgertg}, and third equivalence
is precisely  \cite[Thm. 3.6.13]{NKP} which also implies the automatic homotopy invariance
explaining the equivalence marked by $!$.
\end{proof}

\begin{rem}
The analogue for graded-group valued $\delta$-functors of the fact that the composition of the last two morphisms in \eqref{gwrejgpregrtetgwg}
 detects eqivalences is classically well-known, see \cite{kkka}.
 In the present paper we use the constructive power of \cref{gojertpgwerfwerfwerf} providing \eqref{oigwueifoerwferfwerfwerfw}. \hB
\end{rem}

 \begin{ddd} \label{giuowerpgwrefefwerfwvdf}We call  the inverse  \begin{equation}\label{oigwueifoerwferfwerfwerfw}(-)^{\BM}:\cC\to \Fun^{\BM}(\LCH^{+,\op},\cC)
\end{equation}  of the functor
in \eqref{rweferjflewrkfer} the associated Borel-Moore cohomology functor. \end{ddd}
 The functor in \eqref{oigwueifoerwferfwerfwerfw} thus associates to every object $C$ of $\cC$  the Borel-Moore cohomology theory
$C^{\BM}$ uniquely characterized by $C^{\BM}(*)\simeq C$.
On compact spaces $X$  the value $C^{\BM}(X)$  is given by the evaluation of the sheafification of the constant $\cC$-valued presheaf 
on $\CM$ with value $C$.

Assume that $\phi:\cC\to \cD$ is a colimit-preserving functor between dualizable 
stable $\infty$-categories.
\begin{kor}\label{ohperttrgertgt}
We have an equivalence
$$\phi\circ (-)^{\BM}\simeq \phi(-)^{\BM}:\cC\to \Fun^{\BM}(\LCH^{+,\op},\cD)\ .$$
\end{kor}

For a complete stable $\infty$-category $\cC$
we let $\Fun_{\BM}(\LCH^{+},\cC)$ denote the full subcategory of $\Fun(\LCH^{+},\cC)$ of Borel-Moore homology theories.

\begin{rem}\label{kohpertgrtgergerg}
Let $G\Orb$ denote the category of transitive discrete $G$-spaces, also called the orbit category of $G$.
We consider the inclusion $e:G\Orb\to G\LCH^{+}$.
For finite groups and complete and cocomplete $\cC$
one can also show directly that
$$e^{*}:\Fun_{\BM}(G\LCH^{+},\cC)\to \Fun(G\Orb,\cC)$$
is an equivalence, and also an analoguous result for cohomology theories. The reason that this, in contrast to \cref{gojertpgwerfwerfwerf},  works without  
stronger assumptions on $\cC$ is that we require homotopy invariance on the domain.
Using these results the assumption of dualizability of $\cC$ and $\cD$ in \cref{ohperttrgertgt} could be dropped.
\footnote{The details are part of  an ongoing PhD project.}  \hB
 \end{rem}
Let $R$ be a   ring spectrum and $S$ denote the sphere spectrum.

\begin{ddd}\label{okgpgkerpwofwerferwfwrefwf}
We define the functor
\begin{equation}\label{gwerpokjfpoweferfwerfwer}(-)_{\BM}:\Mod(R)\to \Fun_{\BM}(\LCH^{+},\Mod(R))\ , \qquad M\mapsto \Sp(S^{\BM},M)\ .\end{equation}
\end{ddd}
For justification note that $\Sp(-,M):\Sp\to \Mod(R)$ is exact and sends colimits to limits.

Note that the $E$-theory category $\EE$ is dualizable by \cite{budu}. 
In view of \cref{kopherthertgertgertr}, $\ee C_{0}(*)\simeq \beins_{\EE}$  and \eqref{oigwueifoerwferfwerfwerfw}  we get from \cref{gojertpgwerfwerfwerf}:
\begin{kor}\label{ijogpwregreg9}
We have an equivalence
$$\ee C_{0}\simeq \beins_{\EE}^{\BM}:\LCH^{+,\op}\to \EE\ .$$
\end{kor}
The $K$-theory functor
$$K:\EE\to \Mod(KU)\ , \quad A\mapsto \EE(\beins_{\EE},A)$$
preserves limits and colimits. Using that $K(\beins_{\EE})\simeq KU$ we get by \cref{ohperttrgertgt} a homotopy theoretic description of the   Borel-Moore $K$-cohomology  functor:
\begin{kor}\label{ijogpwregreg9}
We have an equivalence
\begin{equation}\label{btrbrerertgetg}K(\ee C_{0})\simeq  KU^{\BM} 
\end{equation} in 
$\Fun^{\BM}(\LCH^{+,\op},\Mod(KU))$.
\end{kor}

Note that the left-adjoint in \begin{equation}\label{egewfgrefwref}-\wedge KU:\Sp\leftrightarrows \Mod(KU):forget
\end{equation}is a colimit preserving functor between compactly generated and hence dualizable stable $\infty$-categories.
 We conclude from \cref{ohperttrgertgt} that 
\begin{equation}\label{gwerpokfperfwefwer}KU^{\BM}\simeq S^{\BM}\wedge KU
\end{equation} 
in $\Fun^{\BM}(\LCH^{+,\op},\Mod(KU))$.

Combining \cref{ijogpwregreg9} with \eqref{pojgkpowergewrweferfr}  we get:\begin{kor}
We have an equivalence
$$K^{\an}_{A}\simeq \EE( \beins_{\EE}^{\BM},A):\LCH^{+}\to \Mod(KU)$$
which is natural in $A$.
\end{kor}

\begin{prop}\label{hertgertgergtrgrtg}
We have an equivalence
$$K^{\an}_{A}\simeq K(A)_{\BM}$$
in $\Fun^{\BM}(\LCH^{+},\Sp)$ which is natural in $A$.
\end{prop}
\begin{proof}
The UCT-class (UCT stands for universal coefficients) in $\EE$ is defined as the localizing subcategory of objects $A$ in $\EE$
such that the map  $$\EE(A,B) \to \Mod(KU)(K(A),K(B))$$ induced by the 
$K$-theory functor $K=\EE(\beins_{\EE},-):\EE\to \Mod(KU)$ is an equivalence in $\Mod(KU)$ for all $B$ in $\EE$.
We claim that all commutative $C^{*}$-algebras belong to the UCT-class of $\EE$. 
Since every commutative $C^{*}$-algebra is the filtered colimit of its separable subalgebras and the $E$-theory functor  preserves filtered colimits  it suffices to see that all separable commutative $C^{*}$-algebras belong to the UCT-class.
Since  a separable commutative $C^{*}$-algebra $A$ is nuclear the comparison map
$c:\KK\to \EE$ induces an equivalence  $$\KK(A,B)\stackrel{\simeq}{\to} \EE(A,c(B))$$ for every object $B$ in $\KK$. 
It is now a classical fact that  commutative separable $C^{*}$-algebras belong to the 
UCT-class for $\KK$.
This gives the upper horizontal equivalence in the commutative triangle
$$\xymatrix{\KK(A,B)\ar[rr]^-{\simeq}\ar[dr]_{\simeq}&& \Mod(KU)(K(A), K(B))\\ &\EE(A,c(B))\ar@{..>}[ur]^{\simeq}&}\ .$$
it remains to see that $c:\KK\to \EE$ is essentially surjective. 
Since the separable version $c_{\sepa}:\KK_{\sepa}\to \EE_{\sepa}$ of the comparison map is a Verdier quotient by \cite[Prop. 3.61]{budu}
and the map $c$ above is equivalent to 
$$\Ind_{\aleph_{1}} (\KK_{\sepa})\to \Ind_{\aleph_{1}} ( \EE_{\sepa})$$ it is also a localization and hence essentially surjective. This finishes the proof of the claim.

Since $C_{0}$ takes values in commutative $C^{*}$-algebras the functor $\ee C_{0}$ 
has values in the UCT-class.
Consequently we get
an equivalence
\begin{equation}\label{rgojweorferwfrefwerf} \EE(\ee C_{0},A)\simeq \Mod(KU)(K(\ee C_{0}),K(A)):\LCH^{+}\to \Mod(KU)\ .
\end{equation}
We now have the chain of equivalences
\begin{eqnarray*}
K^{\an}_{A}&\stackrel{\eqref{pojgkpowergewrweferfr}}{\simeq}& \EE(\ee C_{0},A)\\&\stackrel{\eqref{rgojweorferwfrefwerf}}{\simeq}& \Mod(KU)(K(\ee C_{0}),K(A))\\&\stackrel{\eqref{btrbrerertgetg}}{\simeq} &\Mod(KU)(KU^{\BM},K(A))\\
&\stackrel{\eqref{gwerpokfperfwefwer}}{\simeq}&
\Mod(KU)(S^{\BM}\wedge KU,K(A))\\&\stackrel{\eqref{egewfgrefwref}}{\simeq}&
\Sp(S^{\BM},K(A))\\&\stackrel{\eqref{gwerpokjfpoweferfwerfwer}}{\simeq}&
K(A)_{\BM}
\end{eqnarray*}
in  $\Fun^{\BM}(\LCH^{+},\Sp)$.
 \end{proof}

\begin{rem}  \cref{hertgertgergtrgrtg} is also a direct consequence
of the result previewed in \cref{kohpertgrtgergerg}. The latter would even imply an equivariant version for finite groups $G$.
\hB \end{rem}

We choose an equivalence  \begin{equation}\label{gwweroijkewrjgfowertw334w}
K^{\an}_{\C}(*)\simeq KU\wedge H\C\simeq \prod_{k\in \Z} \Sigma^{2k} H\C
\end{equation}
in $\Mod(H\C)$.

\begin{kor}
We have an equivalence  \begin{equation}\label{ggwerfrfrewfwwrefw34f}K_{\C}^{\an}\simeq \prod_{k\in \Z} \Sigma^{2k} H\C_{\BM}
\end{equation}
in $\Fun_{\BM}(\LCH^{+},\Mod(H\C))$.
\end{kor}
\begin{proof}
In view of \cref{hertgertgergtrgrtg}, inserting \eqref{gwweroijkewrjgfowertw334w}  into \eqref{gwerpokjfpoweferfwerfwer} in the case of the ring spectrum $R:=H\C$ yields the desired equivalence.\end{proof}

\begin{ddd}\label{9gwerfreferwfwever}
We define the homotopical Chern character as the composition
$$\ch^{h}:K^{\an}\xrightarrow{\ch^{K},\eqref{gwerojopwerferfwerfwerfw}} K^{\an}_{\C}\stackrel{\eqref{ggwerfrfrewfwwrefw34f}}{\simeq} \prod_{k\in \Z} \Sigma^{2k} H\C_{\BM}\ .$$
\end{ddd}

\begin{rem}\label{9gwerfreferwfwever1}
There are actually two apriori different  versions of the homotopical Chern character.
 In addition to the one described in \cref{9gwerfreferwfwever} we could also
 define an even more homotopical  Chern character by 
 $$K^{\an}\stackrel{Prop. \ref{hertgertgergtrgrtg}}{\simeq} K^{\an}(*)_{\BM}\simeq 
 KU_{\BM} \to  (KU\wedge H\C)_{\BM} \simeq \prod_{k\in \Z} \Sigma^{2k} H\C_{\BM} \ ,
  $$ where the third map is given by the unit of $H\C$.  
  Using the  naturality-in-$A$ statement in of  \cref{hertgertgergtrgrtg} and the details of the definition
  $\beins_{\EE}\to \beins_{\EE}\otimes H\C$ from \cref{gkojpwregwerwef}
  one checks that both versions actually coincide
     \hB
\end{rem}

 
We now come back to the equivariant case.
Assume that $E:BG\times \LCH^{+}\to \cC$ 
is such that its adjoint
$BG\to \Fun(\LCH^{+},\cC)$ takes values in Borel-Moore homology theories.

\begin{ddd}\label{okheprthertgertgetg}
We define the Borel-equivariant Borel-Moore homology associated to $E$ by
$$E^{hG}:G\LCH^{+}\xrightarrow{E}\Fun(BG\times BG,\cC)\xrightarrow{\diag_{BG}^{*}}   \Fun(BG,\cC)\xrightarrow{\lim_{BG}}\cC\ .$$
\end{ddd}
The justification is straightforward.

For a group $G$ we  can now extend the definition of the homotopical Chern character to the Borel case.
\begin{ddd} \label{iogjiowergwerfwerfrefw} 
We define the Borel-equivariant homotopical Chern character
$$\ch^{hG}: K^{G,\an}_{A}  \xrightarrow{\beta,\eqref{gwerpoj0erjgeroigerwg}} K^{\an,hG}_{A}\xrightarrow{\lim_{BG}\ch^{h}} ( \prod_{k\in \Z} \Sigma^{2k} H\C_{\BM})^{hG}:G\LCH^{+}\to \Mod(KU)\ .$$
\end{ddd}
Note that $ \ch^{hG}$ reduces to $\ch^{h}$ in the case of the trivial group.

 \subsection{Equivariant Homology theories  and assembly maps}\label{kophrhrtehethtr}

We start with explaining the construction of   Weiss-Williams assembly  
\cite{zbMATH01452550} in the equivariant context and using the $\infty$-category language.
Let $\ell:\Top\to \Spc$ be the Dwyer-Kan localization of the category of topological spaces at the weak homotopy equivalences. This is one of the well-known presentations of the category of spaces.
We have a functor 
\begin{equation}\label{hrthetgretgerg}Y^{G}:G\Top\to \PSh(G\Orb)\ , \quad Y^{G}(S):=\ell \Map_{G\Top}(S,X)\ ,
\end{equation} where $ \Map_{G\Top}(S,X)$ is the topological mapping space from the orbit $S$ considered as a
discrete $G$-topological space to $X$. 
A map $X\to Y$  in $G\Top$ is called an equivariant weak equivalence   if it induces a weak equivalence $X^{H}\to Y^{H}$ on the fixed points sets for all subgroups $H$ of $G$.
By Elemendorf's theorem \cite{elmendorf} the functor \eqref{hrthetgretgerg} presents its target as 
  the Dwyer-Kan localization of  its domain at the equivariant weak equivalences $W_{weak}$.

If $\cC$ is a cocomplete  $\infty$-category, then   we get a functor 
\begin{equation}\label{hretojgoertgrtegegertg} (-)^{\%}:\Fun(G\Orb,\cC) \xrightarrow{\simeq,y_{!}}  \Fun^{\colim}( \PSh(G\Orb),\cC)\xrightarrow{Y^{G,*}}\Fun(G\Top,\cC)\ ,
\end{equation}where $y_{!}$ is left-Kan extension along the Yoneda embedding $y:G\Orb\to \PSh(G\Orb)$.
Since $Y^{G}$ is a localization, the functor $(-)^{\%}$ is an inclusion of a full subcategory. If $\cC$ is stable, then the
 essential image of $(-)^{\%}$ is, by definition,  the category of $\cC$-valued equivariant homology theories.
 
 Assume that $E:G\Orb\to \cC$ is a functor to a cocomplete stable $\infty$-category.
 The following are well-known properties of equivariant homology theories.
 \begin{kor}\label{kopwrgergwergfsdfg}
 The functor
 $E^{\%}:G\Top\to \cC$ is homotopy invariant, excisive for  invariant open decompositions, and it sends disjoint unions of arbitrary families of $G$-spaces to sums.
 \end{kor}

 Assume now that $\cF$ is a family of subgroups of $G$ and let $G_{\cF}\Orb\subseteq G\Orb$ denote the full subcategory of orbits with stabilizers in $\cF$. Using left Kan-extension along this inclusion for the first map we define the functor
  $$ (-)^{\%,\cF}:\Fun(G_{\cF}\Orb,\cC)\to \Fun(G\Orb,\cC)  \xrightarrow{(-)^{\%}}\Fun(G\Top,\cC)\ .$$
  
  Assume that $E:G_{\cF}\Orb\to \cC$ is a functor to a cocomplete stable $\infty$-category.
 \begin{kor}\label{kopwrgergwergfsdfg1}
 The functor
 $E^{\%,\cF}:G\Top\to \cC$ is homotopy invariant, excisive for open decompositions, and it sends disjoint unions of arbitrary families of $G$-spaces to sums.
 \end{kor}

 Let $i:G\CW \to G\Top$ be the inclusion of  the category of $G\CW$-complexes. Part of the proof of Elmendorf's theorem
 is the construction of a cofibrant approximation functor $c:G\Top \to G\CW$ together with a transformation $i\circ c\to \id$
 implemented by equivariant weak equivalences  and an equivalence
 $$G\CW[W_{h}^{-1}]\xrightarrow{\simeq,i} G\Top[W_{weak}^{-1}]\simeq \PSh(G\Orb)\ ,$$
 where $W_{h}$ denotes the collection of equivariant homotopy equivalences.
 
 Let $o_{\cF}:G_{\cF}\Orb\to   G\Top$ denote the canonical embedding. 
 For a functor $F:G\Top\to \cC$ we use the abbreviation.
 $$F^{\%,\cF}:=(o_{\cF}^{*}F)^{\%,\cF}:G\Top\to \cC\ .$$
 As usual we will omit the symbol $\cF$ in the case $\cF=\All$.
 
 Assume now that the functor
 $F $ is a homotopy invariant. Its restriction to $G\CW$  factorizes over the Dwyer-Kan localization at the equivariant homotopy equivalences and  therefore gives rise to functors $F'$ and $F''$ as indicated in the following diagram:
 $$\xymatrix{G\Top\ar@{-->}[drr]^{Y^{G}}\ar[rr]^{F}&&\cC \\ 
 G\CW\ar[u]^{i}\ar[r]&G\CW[W_{h}^{-1}]\ar@{..>}[ur]^{F'}\ar[r]^{\simeq}& \PSh(G\Orb)\ar@{..>}[u]^{F''}\\\ar@/^1cm/[uu]^{o}G\Orb\ar[u]\ar[urr]^{y}&&}\ .
$$ 
Then $y^{*}F''\simeq o^{*}F$.
The counit
 $  y_{!}y^{*}F''\to F'':\PSh(G\Orb)\to \cC$ presents its domain as the final colimit preserving functor over $F''$.
Pulling this back along $Y^{G}$ yields a natural assembly map transformation
$$\asmbl^{F}:F^{\%}:=  Y^{G,*}y_{!} o^{*}F\simeq        Y^{G,*}y_{!}y^{*}F''\to Y^{G,*}F''\stackrel{!}{\to} F$$
which presents $ F^{\%}$ as the final equivariant homology theory over $F$.
In order to obtain the transformation marked by $!$ we argue as follows.
 The cofibrant approximation transformation induces a morphism
$c^{*}i^{*}F\to F$. Since $Y^{G}$ sends weak equivalences to equivalences, applied to the cofibrant approximation transformation  it induces 
an equivalence $Y^{G}\circ i\circ c\to Y^{G}$.
We now  start from the equivalence  $i^{*} Y^{G,*} F''\simeq i^{*}F$ which holds by definition of $F''$. 
We apply $c^{*}$ and get the desired map
$$Y^{G,*}F''\stackrel{\simeq}{\leftarrow} c^{*}i^{*} Y^{G} F''\simeq c^{*} i^{*}F\to F\ .$$

More generally, for a family of subgroups we have an assembly map
\begin{equation}\label{erthertertgdfhgf}\asmbl^{F}_{\cF}:F^{\%,\cF}\to F^{\%}\to F\ .
\end{equation}

   Let $F:G\Top\to \cC$ be some homotopy invariant functor functor.  
   The following can be shown by induction over the number of cells.
   
\begin{kor} If $F$ is  excisive for open (or closed) decompositions and $X$ is homotopy equivalent to a finite $G$-$CW$-complex with stabilizers in $\cF$, then
$$\asmbl^{F}_{\cF,X}: F^{\%,\cF}(X)\to F(X)$$ is an equivalence.
  \end{kor}

 We now consider the category $G\TB$ described in \cref{kopwhwthrh}.   We have a right Bousfield localization
 \begin{equation}\label{wergewrfdgswre} i:G\Top\leftrightarrows G\BT:u\ ,
\end{equation} where $u$ forgets the bornology and 
 $i$ equips a  topological space with the maximal bornology.

 Assume that $F:G\TB\to \cC$ is a homotopy invariant functor to a cocomplete stable $\infty$-category.
We then define  
 the functor
 $$F^{\%,\cF}:=i_{!}(i^{*}F)^{\%,\cF}:G\TB\to \cC\ .$$
 If $F$ is homotopy invariant, then so is $i^{*}F$ and 
  the assembly map \eqref{erthertertgdfhgf} induces a natural transformation
 \begin{equation}\label{jrtzthrtzkpotrzhtzh}
\asmbl_{\cF}^{F}:F^{\%,\cF}\stackrel{def}{=}i_{!}(i^{*}F)^{\%,\cF}\stackrel{i_{!}\asmbl_{\cF}^{i^{*}F}}{\to} i_{!}i^{*}F\stackrel{counit}{\to} F\ .
\end{equation} 

\begin{lem}\label{kophertgergertgert} The functor 
$F^{\%,\cF}:G\BT\to \cC$ is  homotopy invariant and open   excisive.
\end{lem}
\begin{proof}
In view of \eqref{wergewrfdgswre},   for $E:G\Top\to \cC $ we have $i_{!}E\simeq E\circ u$.
If $E$ is homotopy invariant or open excisive, then so is $i_{!}E$.   
The assertion now follows from \cref{kopwrgergwergfsdfg1}.\end{proof}

Let $E:G\TB\to \cC$ be a functor to a complete stable $\infty$-category.
Recall that a subset $W$ of $X$ is $G$-bounded if there exists a bounded subset $B$ of $X$ such that $W\subseteq GB$.

\begin{ddd}\label{kopgwegewrfwerf}
$E$ is locally finite if for every $X$ in $G\TB$  
\begin{equation}\label{sdfvsdfvsfdvsvsdfvwer}  \lim_{W}  E(X\setminus W) \simeq 0 \ ,
\end{equation}
where $W$ runs over the $G$-invariant $G$-bounded subsets of $X$.
\end{ddd}

\begin{construction}\label{ohiperzthtrgetrger}{\em 
We explain the construction of forcing local finiteness, see \cite[Def. 7.15]{buen} for the non-equivariant case.
Let $G\BT^{\cW}$ be the category of pairs $(X,W)$ of $X$ in $G\TB$ and a $G$-invariant  $G$-bounded  subset $W$.
 Morphisms $(X,W)\to (X',W')$ are morphisms $f:X\to X'$  in $G\BT$ such that $f(W)\subseteq W'$.
Let $q:G\TB^{\cW} \to G\TB$ be the forgetful functor sending $(X,W)$ to $X$. 
For every functor $E:G \TB\to \cC$ to a complete stable target we can associate the functor
\begin{equation}\label{gerwferfwerferfweffg}\tilde E:G\TB^{\cW}\to \cC\ , \quad \tilde E(X,W):=\Cofib(E(X\setminus W)\to E(X))\ .
\end{equation} 
It comes with a natural transformation $q^{*}E\to \tilde E$.
We define the locally finite version of $E$ by $$E^{\lf}:=q_{*}q^{*}E:G\TB\to \cC$$ and have a natural transformation 
\begin{equation}\label{ergwerfdsgwer}E\stackrel{unit}{\to}q_{*}q^{*}E\to q_{*}\tilde E=E^{\lf}\ .
\end{equation}
Note that \begin{equation}\label{goijwioejorigwerfwerf}E^{\lf}(X)\simeq \lim_{W} \Cofib(E(X\setminus W)\to E(X))\ ,
\end{equation}
where $W$ runs over the $G$-bounded invariant subsets of $X$.}
 \hB
\end{construction}

Let $E:G\TB\to \cC$ be a functor to a complete stable target.
\begin{lem}\label{lhperthertgerg}
The functor $E^{\lf}:G\BT\to \cC$ is locally finite and $E\to E^{\lf}$  is the initial functor to a locally finite functor.
\end{lem}
The following proposition is the straightforward generalization of  \cite[Lem. 7.35 \& 7.36]{buen} to the equivariant case.
\begin{prop}\label{kopgpwerrefweferfw}
If $E:G\BT\to \cC$ is homotopy  invariant or excisive for open (closed) decompositions, then
$E^{\lf}:G\BT\to \cC$ is also homotopy  invariant or excisive for open (closed) decompositions, respectively.
\end{prop}

\cref{kophertgergertgert} and \cref{kopgpwerrefweferfw} together  imply:

\begin{kor}\label{plhertgertgertetrrgertgt}
The functor $E^{\%,\cF,\lf}:G\BT\to \cC$ is  homotopy  invariant and excisive for open  decompositions.
\end{kor}

If $E:G\BT\to \cC$ is homotopy invariant and locally finite, then we call 
\begin{equation}\label{gerwfwerfwerfwerfw}\asmbl_{\cF}^{E,\lf}:E^{\%,\lf}\xrightarrow{(\asmbl^{E}_{\cF})^{\lf}, \eqref{jrtzthrtzkpotrzhtzh}} E^{\lf}\stackrel{\simeq, \eqref{ergwerfdsgwer}}{\leftarrow} E
\end{equation}
the locally finite assembly map for   $E$ and the family $\cF$.

\begin{kor} If $E:G\BT\to \cC $ is  homotopy invariant, excisive for open (or closed) decompositions and locally finite, and
  $X$ is a locally  finite $G$-$CW$-complex with stabilizers in $\cF$, then the locally finite assembly map 
 $$\asmbl^{E,\lf}_{\cF,X}:E^{\%,\lf}(X)\to E(X)$$ is an equivalence.
 \end{kor}
 
 Recall the functor $\iota^{\topp}:G\BT\to G\LCH^{+}$ from   \eqref{gewrfwerfrwefwe}.
Let $E:G\LCH^{+}\to \cC$ be a functor and recall \cref{kohperthrgrtgertgeg}.
\begin{lem}\label{8kothpetrhtrheth9}
If $E:G\LCH^{+}\to \cC$ is an equivariant weak Borel-Moore   homology theory, then
$\iota^{\topp,*}E:G\BT\to \cC$ is homotopy invariant and locally finite.
\end{lem}
\begin{proof}
Using the compactness of $[0,1]$ one checks for $X$ in $G\BT$ that
$$\iota^{\topp }([0,1]\otimes X)\cong [0,1]\times \iota^{\topp }(X)\ .$$  
Since $E$ is homotopy invariant, so is $ \iota^{\topp,*}E$.

 In order to show that 
$\iota^{\topp,*}E$ is locally finite we consider $X$ in $G\BT$.
We must show that
$$\lim_{W} E(X\setminus W)\simeq 0\ ,$$
where $W$ runs over the invariant $G$-bounded subsets.
We have the inductive system $W\mapsto C_{0}(X\setminus W)$ of commutative $G$-$C^{*}$-algebras
where for $W\subseteq W'$ the map $C_{0}(X\setminus W)\to C_{0}(X\setminus W')  $ is given by restriction along $X\setminus W'\subseteq X\setminus W$.
 For every
$f$ in $ C_{0}(X\setminus W)$ we have
$$\lim_{W\subseteq W'}\| f_{|X\setminus W'}\|=0\ ,$$
where the limits runs over $W'$.
This implies that $\colim_{W}  C_{0}(X\setminus W)\cong 0$.
By Gelfand duality this is equivalent to $\lim_{W} \iota^{\topp}(X\setminus W)\cong \emptyset$. Since $E$ is locally finite in the sense of \cref{jkopgwegergferfwr}
  we conclude that
$\lim_{W} (\iota^{\topp,*}E)(X\setminus W)\simeq 0$.
\end{proof}

Let $E:G\LCH^{+}\to \cC$ be an equivariant  weak Borel-Moore  homology theory.
By \cref{8kothpetrhtrheth9}  we know that
$\iota^{\topp,*}E:G\BT\to \cC$ is homotopy invariant and locally finite.
We get the locally finite assembly map \eqref{gerwfwerfwerfwerfw}
\begin{equation}\label{fervjoierjveoijmewoiwmo}\asmbl_{\cF}^{\iota^{\topp,*}E,\lf}:( \iota^{\topp,*}E)^{\%,\cF,\lf}\to \iota^{\topp,*}E:G\BT\to \cC \ .
\end{equation}

\subsection{Coarsification}

Coarsification of locally finite homology theories has been introduced by  Roe,  see  \cite[Sec. 5.5]{roe_lectures_coarse_geometry}. In this subsection we explain  variants of this construction 
in our contexts.
For a $G$-bornological coarse space $X$   and an invariant coarse entourage $U$   containing the diagonal we  consider the Rips complex
$P_{U}(X)$. It is defined as the realization of the  simplicial complex
with $G$-action whose $n$-simplices are the  $n+1$-element subsets of $X$ such that each pair of elements  belongs to
$U$. We often  interpret points in $P_{U}(X)$ as finitely supported probability measures on $X$.

We realize $n$-simplices geometrically as the intersection of the standard sphere $S^{n}$
with the positive quadrant of $\R^{n+1}$ in order to obtain the spherical path metric structure on them. We then 
equip $P_{U}(X)$ with the $G$-invariant path  metric which induces the spherical path metric on every  simplex.
This metric 
 induces the uniform and coarse structure on $P_{U}(X)$. We further equip $P_{U}(X)$ with the minimal compatible bornology such that the 
 inclusion of the zero skeleton    $X\to P_{U}(X)$ sending $x$ to the Dirac measure $\delta_{x}$ at $x$ is  bornologcial. In this way we will consider $P_{U}(X)$ as an object of $G\UBC$. The inclusion of $X$ as the zero-skeleton of $P_{U}(X)$ is a 
  weak coarse equivalence $X_{U}\to \iota (P_{U}(X))$, i.e., a coarse equivalence after forgetting the $G$-action,   where $\iota:G\UBC\to G\BC$ is the forgetful functor.
  
  Assume that $f:X\to X'$ is a morphism in $G\BC$ and $U'$ is an invariant coarse entourage of $X'$ containing the diagonal such that $f(U)\subseteq U'$. Then we get an induced map 
      $P(f):P_{U}(X)\to P_{U'}(X')$ induced by the push-forward of probablity measures. It is a morphism of $G$-simplicial complexes and hence contractive with respect to the spherical path metrics. Since $f$ is proper, $P(f)$ is a morphism in 
 $G\UBC$. 
  Assume  that
$F:G\UBC\to \cC$ is a functor to a complete stable target. \begin{ddd}\label{lpohertertgrtget} We  define the coarsification of $F:G\UBC\to \cC$ to be the functor
$$F\bP:G\BC\to \cC\ , \quad X\mapsto \colim_{U\in \cC^{G}_{X}} F(P_{U}(X)) \ .$$
\end{ddd}

\begin{rem}\label{kophertgrtgretgertgertg}
In order to make this definition more precise we proceed similarly as in \cref{kophprthgertrtgegrtg}. 
We have already defined the functors indicated by bold arrows in
$$\xymatrix{G\BC^{\cC}\ar[d]_{(X,U)\mapsto X}\ar[rrr]^{(X,U)\mapsto P_{U}(X)} &&&G\UBC  \ar[r]^-{F}&\cC \\G\BC\ar@{..>}[urrrr]_{F\bP}&& }\ .$$
We then define $F\bP$ by left Kan-extension. 
Coarsification is also functorial in $F$, i.e., we have actually a functor
$$(-)\bP:\Fun(G\UBC,\cC)\to \Fun(G\BC,\cC)\ .$$
 \hB
\end{rem}
Coarsification is a tool to produce coarse homology theories.
Recall that $X$ in $G\UBC$ is called flasque if   it admits a selfmorphism $f:X\to X$
such that:\begin{enumerate}
\item  \label{ojgwergjeorfijwref}$f$ is close and  homotopic to the identity, \item 
non-expanding in the sense that
there exist a cofinal set of entourages  $U$ in $\cC_{X}^{G}$ such that $f(U)\subseteq U$, and \item  \label{ojgwergjeorfijwref1} shifting in the sense that for every bounded subset $B$ of $X$ there exists $n$ in $\nat$ such that $f^{n}(X)\cap B=\emptyset$.\end{enumerate}

  \begin{prop}\label{pkehpetrherththgrtg}
If $F:G\UBC\to \cC$ is homotopy invariant, excisive for decompositions of simplicial complexes into subcomplexes and vanishes on flasques,   then $F\bP:G\BC\to \cC$ is an equivariant coarse homology theory.
\end{prop}
\begin{proof}
The argument is the same as for \cite[Prop. 5.2]{ass}. 
The functor $F\bP$ is $u$-continuous by design.
 Homotopy invariance of $F$ implies coarse invariance of $F\bP$. 
  Excisiveness of $F$ for decompositions of simplicial complexes into subcomplexes implies
  excision  for $F\bP$. Finally,   the property of vanishing on flasques is preserved by coarsification.
\end{proof}

\begin{rem}\label{lophertherthtergetg}
The  \cref{kopgwegewrfwerf}  of local finiteness can be interpreted for functors on $G\UBC$. 
For homotopy invariant functors local finiteness implies vanishing on flasques, see the proof of \cite[Lem. 7.21]{buen}. \hB
\end{rem}

 Recall  the notation $\iota^{\prime}$ from \eqref{gewrfwerfrwefwe}.
\begin{ddd}\label{lergrgrthertgertg}
 We define the coarsification of $E:G\BT\to \cC$ by
\begin{equation}\label{erfwerifjoiwefwerfwerf}E\bP:= (\iota^{\prime,*}E)\bP:G\BC\to \cC\ .
\end{equation}
\end{ddd}

Recall that $X$ in $G\BT$ is flasque if it admits a selfmap
$f:X\to X$ which is homotopic to the identity and shifting in the sense of \cref{ojgwergjeorfijwref1} above.
We  then have the analogue of \cref{pkehpetrherththgrtg}:
\begin{prop}\label{pkehpetrherththgrtg1}
If $E:G\BT\to \cC$ is homotopy invariant, excisive for decompositions of simplicial complexes into subcomplexes and vanishes on flasques,   then $E\bP:G\BC\to \cC$ is an equivariant coarse homology theory.
\end{prop}
 %
%
%
%
%
%

\begin{ex}
Assume that $\cC$ is a complete and cocomplete stable $\infty$-category and consider a functor $E:G_{\cF}\Orb\to \cC$. By \cref{plhertgertgertetrrgertgt}, \cref{lhperthertgerg} and \cref{lophertherthtergetg}  the functor  $E^{\%,\cF,\lf}$  
satisfies the assumptions of \cref{pkehpetrherththgrtg1} and gives rise to an equivariant coarse homology theory
$$E^{\%,\cF,\lf}\bP:G\BC\to \cC\ .$$
\hB
\end{ex}

We consider the composition of functors
$$\iota^{\topp,\prime}:G\UBC \xrightarrow{\iota'}G\BT\xrightarrow{\iota^{\topp}} G\LCH^{+} \ ,$$
see \eqref{gewrfwerfrwefwe} for notation.
 \begin{ddd}\label{lhertghetrgrtgertgt}
We define the coarsification of $E:G\LCH^{+}\to \cC$ by
$$E\bP:= (\iota^{\topp,\prime,*}E)\bP:G\BC\to \cC\ .$$
\end{ddd}
We have \begin{equation}\label{gwergoij0fpwerfwerf}
E\bP(X)\simeq \colim_{U\in \cC^{G}_{X}} E(\iota^{\topp,\prime} P_{U}(X))\ .
\end{equation}

\begin{prop}\label{lpkherferferfergertgertgetrg}
If $E:G\LCH^{+}\to \cC$ is a weak equivariant Borel-Moore homology theory, then $E\bP:G\BC\to \cC$ is a coarse homology theory.
\end{prop}
\begin{proof}
We will apply \cref{pkehpetrherththgrtg} to $\iota^{\topp,\prime,*}E$. To this end we verify its assumptions.
It follows from \cref{8kothpetrhtrheth9} that $\iota^{\topp,\prime,*}E$ is homotopy invariant and
locally finite. By  \cref{lophertherthtergetg} it annihilates flasques. 
 
 The functor $\iota^{\topp,\prime}$ sends decompositions of $G$-simplicial complexes 
   into $G$-subcomplexes to closed decompositions.
 Since $E$ is closed excisive, $\iota^{\topp,\prime,*}E$ is excisive for such decompositions.
%
%
%
%
%
\end{proof}


\subsection{The coarse character}

In this section we introduce the coarse character from equivariant coarse homology theory to
the coarsification of Borel-Moore homology. In a first step we construct, on the level of chain complexes, a map from coarse homology to the locally finite singular homology of the Rips complex.  
We then observe based on the theory of assembly developed in  \cref{kophrhrtehethtr} that locally finite singular homology on $G\BT$  has a universal property for mapping out
to other versions of locally finite homology. This will be used in a second step in order to construct a map to the coarsification of the corresponding Borel-Moore homology. Note that Borel-Moore homology  as a homology theory on $G\LCH^{+}$ has a universal property for mapping in, but in the process of coarsification we pull-back the Borel-Moore homology  to
$G\BT$ using $\iota^{\topp,\prime}$, and in this process we loose the corresponding universal property.

Let $k$ be a commutative ring and $A$ be a $k$-module.
We consider the functor
$$C_{\sing}^{G,\LoF}(-,A):G\TB\to \Ch_{k}$$
which sends $X$ to the chain complex of $k$-modules of locally finite $G$-invariant singular chains.  The elements of $C_{\sing}^{G,\LoF}
(X)_{n}$ are $G$-invariant functions $\phi:\Sing(X)_{n}\to A$ such that
for every bounded subset $B$ of $X$ the number of simplices of the chain that meet $B$, i.e., the  cardinality  of the set $$\{\sigma\in \Sing(X)_{n}\mid \sigma(\Delta^{n})\cap B\not=\emptyset \wedge \phi(\sigma)\not=0\}\ ,$$
 is finite.  The differential  $d:C_{\sing}^{G,\LoF}
(X,A)_{n}\to C_{\sing}^{G,\LoF}
(X,A)_{n-1}$ is the usual homological differential. The local finiteness ensures that the sums involved in its definition are finite.  Here we use   the compatibility of the topology and the bornology which implies that  compact subsets are always bounded,
and further that the images of singular simplices are compact and hence bounded.
 
 Recall the notation from  \eqref{bsdfpojkvopsdfvsdfvsdfvfdv}.  We get the functor
\begin{equation}\label{hqwioregqwefdewfqwef}H^{G,\LoF}_{\sing}(-,A):=\ell C_{\sing}^{G,\LoF}(-,A):G\BT\to \Mod(Hk)\ .
\end{equation} 
The following is standard.
\begin{prop}
The functor
$H^{G,\LoF}_{\sing}(-,A)$ is homotopy invariant, open excisive, and vanishes on flasques.
\end{prop}

Recall \cref{lergrgrthertgertg}. As a consequence of \cref{pkehpetrherththgrtg1} we get:
\begin{kor} The coarsification
$$H^{G,\LoF}_{\sing}(-,A)\bP:G\BC\to \Mod(Hk) $$
of $H^{G,\LoF}_{\sing}(-,A)$ is an equivariant  coarse homology theory.
\end{kor}

Let $X$ be in $G\BT$ and $W$ be an invariant $G$-bounded subset. Then the map 
$
C_{\sing}^{G,\LoF}(X\setminus W,A)\to C_{\sing}^{G,\LoF}(X ,A)$ is injective. Consequently
we have an equivalence
\begin{equation}\label{heith0ergterthert}\ell \frac{C_{\sing}^{G,\LoF}(X ,A)}{C_{\sing}^{G,\LoF}(X\setminus W,A)}\simeq \Cofib( H^{G,\LoF}_{\sing}(X\setminus W,A)
\to H^{G,\LoF}_{\sing}(X,A))\ .
\end{equation}
We let  $$H^{G,\lf}_{\sing}(-,A):=(H^{G,\LoF}_{\sing}(-,A))^{\lf}$$
 be the result of making $H^{G,\LoF}_{\sing}(-,A)$ locally finite, see \cref{ohiperzthtrgetrger}.
In view
  of
\eqref{goijwioejorigwerfwerf} and \eqref{heith0ergterthert} we have
$$H^{G,\lf}_{\sing}(X,A)\simeq \lim_{W}\ell \frac{C_{\sing}^{G,\LoF}(X ,A)}{C_{\sing}^{G,\LoF}(X\setminus W,A)}\ ,$$
where $W$ runs over the  $G$-bounded invariant  subsets of $X$.
On the other hand,
the map
$$C_{\sing}^{G,\LoF}(X ,A)\to  \lim_{W} \frac{C_{\sing}^{G,\LoF}(X ,A)}{C_{\sing}^{G,\LoF}(X\setminus W,A)}$$
is an isomorphism of chain complexes.
Hence the natural  transformation \eqref{ergwerfdsgwer}  specializes to 
\begin{equation}\label{hfgfqzwegfuzfqewf}H^{G,\LoF}_{\sing}(-,A)\to H^{G,\lf}_{\sing}(-,A)
\end{equation} which at $X$ is the  canonical map
\begin{equation}\label{hertgrtgergtrg}\ell   \lim_{W} \frac{C_{\sing}^{G,\LoF}(X ,A)}{C_{\sing}^{G,\LoF}(X\setminus W,A)}\to
   \lim_{W}\ell \frac{C_{\sing}^{G,\LoF}(X ,A)}{C_{\sing}^{G,\LoF}(X\setminus W,A)}\ .\end{equation} 
  Note the different orders of $\ell$ and the limit in domain and target of this map.
   
   \begin{rem}
   Note that the structure maps 
   $$ \frac{C_{\sing}^{G,\LoF}(X ,A)}{C_{\sing}^{G,\LoF}(X\setminus W,A)}\to  \frac{C_{\sing}^{G,\LoF}(X ,A)}{C_{\sing}^{G,\LoF}(X\setminus W',A)}$$ for $W\subseteq W'$ are surjective.
   If $X$ admits a countable exhaustion by $G$-bounded  invariant subsets, then 
   we can apply Mittag-Leffler in order to see that \eqref{hertgrtgergtrg} is an equivalence.
 {The general case remains unclear so that we do not know whether $H^{G,\LoF}_{\sing}(-,A)$ is already locally finite.} \hB
   \end{rem}

Recall the  equivariant coarse homology with coefficients in $A$ from \eqref{nfgbgfbddgfber}.
We are going to construct a natural 
 transformation of equivariant coarse homology theories
 \begin{equation}\label{rgrwtbwergwerfwerf}
\kappa^{G}:\rmH\cX^{G}(-,A)\to H^{G,\LoF}_{\sing}(-,A)\bP:G\BC\to \Mod(Hk)\ .
\end{equation}


\begin{construction}\label{huigwerfgerfwrefwerfe}{\em 
The transformation in \eqref{rgrwtbwergwerfwerf}
  will be induced by a chain-level construction.
Let $X$ be    in $G\BC$. 
 For every $U$ in $\cC^{G}_{X}$ containing the diagonal we have
a map \begin{equation}\label{sgsdfgsfgfdgs}\tilde\kappa^{G}_{U}:C\cX^{G}_{U}(X,A)\to C_{\sing}^{G,\LoF}( P_{U}(X),A)\ ,
\end{equation}
where $C\cX^{G}_{U}(X,A)$ is the subcomplex of $C\cX^{G}(X,A)$ of $U$-controlled chains.
 The map $\tilde\kappa^{G}_{U}$ sends
 $\phi$ in $C\cX^{G}_{U}(X,A)_{n}$ to the chain
 $$\tilde \kappa^{G}_{U}(\phi): \Sing_{n}(P_{U}(X))\to A$$ which sends  the singular simplex $ \Delta(x_{0},\dots,x_{n}):\Delta^{n}\to P_{U}(X)$     to $\phi(x_{0},\dots,x_{n})$  and vanishes elsewhere.  Here $\Delta(x_{0},\dots,x_{n})$ is  given in barycentric coordinates by $t=(t_{0},\dots,t_{n})\mapsto \sum_{i=0}^{n} t_{i}\delta_{x_{i}}$,    where we interpret points on $P_{U}(X)$ as finitely supported probability measures.  
   Since $\phi$ is $U$-controlled the condition  $\phi(x_{0},\dots,x_{n})\not=0$ implies that $\Delta(x_{0},\dots,x_{n}) $ is  indeed a well-defined singular simplex in $P_{U}(X)$.
The $G$-invariance or local finiteness of $\phi$ implies $G$-invariance or local finiteness  of $\tilde\kappa^{G}_{U}(\phi)$, respectively.
One checks that $\tilde\kappa^{G}_{U}$ is a map of chain complexes. 

These maps are compatible with the inclusions of chain complexes associated to inclusions of invariant coarse entourages.
Taking the colimit of all $U$ in $\cC_{X}^{G}$ we get a map of chain complexes
$$\tilde \kappa^{G}_{X}:C\cX^{G}(X,A)\to C_{\sing}^{G,\LoF}(-,A)\bP(X)\ .$$ 
The family of maps $(\tilde \kappa^{G}_{X})_{X\in G\BC}$ is a natural transformation of  functors
$$\tilde \kappa^{G}:C\cX^{G}(-,A)\to C_{\sing}^{G,\LoF}(-,A)\bP:G\BC\to \Ch_{k}\ .$$ 
We finally apply $\ell$ from   \eqref{bsdfpojkvopsdfvsdfvsdfvfdv}  in order to get the  transformation
\eqref{rgrwtbwergwerfwerf}
 }\hB\end{construction}
%

We consider the restriction
  \begin{equation}\label{bsdfoivjsdoivjosdfvsdfvsdfvsdf}H_{A}^{G}:G\Orb\xrightarrow{S\mapsto S_{\disc}} G\Top\xrightarrow{i} G\BT \xrightarrow{H^{G,\LoF}_{\sing}(-,A)} \Mod(Hk)
\end{equation} 
  of $H_{\sing}^{G,\LoF}$ to the orbit category, where $i$ is as in \eqref{wergewrfdgswre}. Since $H^{G,\LoF}_{\sing}(-,A)$ is homotopy invariant we have a natural transformation
    $$\asmbl^{H^{G}_{A}}:H_{A}^{G,\%}\to H^{G,\LoF}_{\sing}(-,A):G\BT\to \Mod(Hk)\ .$$ Forcing local finiteness we get a map
    \begin{equation}\label{porjgoiergwerfwerfref}\asmbl^{H^{G}_{A},\lf}:H_{A}^{G,\%,\lf}\to H^{G,\lf}_{\sing}(-,A)\ .
\end{equation}
    
  \begin{ex}\label{ijobpwrgrhre}
  Assume that $T$ in $G\Orb$ is an infinite orbit.
  Since $i(T)$ is $T$ equipped with the maximal bornology it admits no  non-zero $G$-invariant locally finite chains.
  Consequently $H_{A}^{G}(T)=0$. Let $S$ be an infinite $G$-orbit  and $S_{min}$ in $G\TB$ be $S$ equipped with the minimal bornology.
   Using the point-wise formula for the left Kan-extension  in \eqref{hretojgoertgrtegegertg} we have
     $$H_{A}^{G,\%,\lf}(S_{min})\simeq \colim_{T\in G\Orb_{/S_{min,min}}}  H^{G}_{A}(T)  \simeq 0$$ 
     since every $T$ in $G\Orb$ admitting a non-trivial map to $S$ is also infinite.
     
     On the other hand we have 
 $$\pi_{0}H^{G,\LoF}_{\sing}(S_{min},A)\cong \pi_{0}H^{G,\lf}_{\sing}(S_{min},A)\cong A\ .$$
 Hence   the assembly map is not an equivalence for infinite groups. \hB
 \end{ex}

%
%
%
%
%
%
%

\begin{prop}\label{kohprwthegtrgetg} If $G$ is finite and $|G|$ is invertible in $k$, then the assembly map \eqref{porjgoiergwerfwerfref} is an equivalence.
\end{prop}
\begin{proof}
Let $$C^{G}_{\sing}(-,A):G\BT\to \Ch_{k}$$ be the functor which sends $X$ in $G\BT$ to the chain complex of finite $G$-invariant $A$-valued singular chains. We further define
$$H^{G}_{\sing}(-,A):=\ell C^{G}_{\sing}(-,A):G\BT\to \Mod(Hk)\ .$$
\begin{lem}\label{kohperthrtgetrg} If $G$ is finite and $|G|$ is invertible in $k$, then
$H^{G}_{\sing}(-,A)$ is an equivariant homology theory on $G\Top$.
\end{lem}
\begin{proof}
We use the isomorphism $C^{G}_{\sing}(X,A)\cong \lim_{BG}C_{\sing}(X,A)$. Since $|G|$ is invertible in $k$ 
we have equivalences $\colim_{BG}\ell\simeq \lim_{BG}\ell\simeq \ell \lim_{BG}$ of functors from $BG$ to $\Mod(Hk)$, where $\ell:\Ch_{k}\to \Mod(Hk)$. Consequently,
$$H^{G}_{\sing}(X,A)\simeq \colim_{BG} H_{\sing}(X,A)\ .$$
We now use that $H_{\sing}(-,A)$ is a non-equivariant homology theory and corresponds to a colimit-preserving functor from $\Spc$ to $\Mod(Hk)$ denoted by the same symbol.
We can therefore write
$H^{G}_{\sing}(-,A)$ as the composition
$$G\Top\stackrel{Y^{G}}{\to} \PSh(G\Orb)\stackrel{j_{G}^{*}}{\to} \Fun(BG,\Spc)\stackrel{H_{\sing}(-,A)}{\to}\Fun(BG,\Mod(Hk))\stackrel{\colim_{BG}}{\to} \Mod(Hk)\ ,$$
where $j_{G}:BG\to G\Orb$ is the inclusion.
It follows that
$H^{G}_{\sing}(-,A)$ is the equivariant homology theory represented by  the functor
 $$ j_{G,!} \colim_{BG}H_{\sing}(-,A):G\Orb\to \Mod(Hk)\ .$$
 \end{proof}

It follows from \cref{kohperthrtgetrg} that the assembly map
$$H^{G,\%}_{A}\stackrel{\simeq}{\to}  H^{G}_{\sing}(-,A)$$ is an equivalence. 

Since  $G$ is finite the condition    of being $G$-bounded is equivalent to being  bounded.
Let $W$ be a $G$-bounded  invariant subset of $X$.
Then we have an  isomorphism of chain complexes
$$ \frac{C_{\sing}^{G}(X ,A)}{C_{\sing}^{G}(X\setminus W,A)} \to  \frac{C_{\sing}^{G,\LoF}(X ,A)}{C_{\sing}^{G,\LoF}(X\setminus W,A)}$$
induced by the canonical inclusions.
Hence for every $G$ bounded invariant subset $W$  we get an equivalence
$$\Cofib(H^{G}_{\sing}(X\setminus W,A)\to H_{\sing}^{G}(X,A))\stackrel{\simeq}{\to} \Cofib(H^{G,\LoF}_{\sing}(X\setminus W,A)\to H_{\sing}^{G,\LoF}(X,A))\ .$$ 
Applying $\ell$ and taking the limit we get an equivalence
$$H^{G}_{\sing}(X,A)^{\lf}\stackrel{\simeq}{\to} H^{G,\lf}_{\sing}(X,A):G\BT\to \Mod(Hk)\ .$$
Composing these equivalences we get the desired equivalence 
$$H_{A}^{G,\%,\lf}\stackrel{\simeq}{\to}  H^{G}_{\sing}(-,A)^{\lf}\stackrel{\simeq}{\to}  H^{G,\lf}_{\sing}(-,A)\ .$$
\end{proof}
We assume that $G$ is finite and that $|G|$ is invertible in $k$.
\begin{ddd}\label{oihop0rthrtgtrgertg}
The locally finite coarse character is the natural transformation
\begin{equation}\label{gwergwergwrfregw}\chi^{G,\lf}:\rmH\cX^{G}(-,A)\stackrel{\kappa^{G},\eqref{rgrwtbwergwerfwerf}}{\to} H^{G,\LoF}_{\sing}(-,A)\bP\stackrel{\eqref{hfgfqzwegfuzfqewf}}{\to} H^{G,\lf}_{\sing}(-,A)\bP\stackrel{\simeq, \eqref{porjgoiergwerfwerfref} }{\leftarrow}
H_{A}^{G,\%,\lf}\bP:G\BC\to \Mod(Hk)\ .
\end{equation} \end{ddd}
In view of \cref{ijobpwrgrhre} and the presence of the inverted map in \eqref{gwergwergwrfregw}
the  locally finite coarse character is only defined for finite groups $G$ with $|G|$ invertible in $k$.


In the case of the trivial group $G$, by \cref{okgpgkerpwofwerferwfwrefwf} we can associate to every spectrum $F$ a Borel-Moore homology theory $F_{\BM}$. We use this in order to construct the coarse Borel-Moore character.

\begin{ddd}\label{iogopwrefwerfewrferwf}
The  coarse Borel-Moore   character is defined as the natural transformation
\begin{equation}\label{gwergwergwrfregw1}\chi_{\BM}:\rmH\cX(-,A)\xrightarrow{\chi^{\lf},\eqref{gwergwergwrfregw}}HA^{\%,\lf} \bP\xrightarrow{\asmbl^{\iota^{\topp,*} HA_{\BM} ,\lf },\eqref{fervjoierjveoijmewoiwmo}}HA_{\BM}\bP
 :\BC\to \Mod(Hk)\ .
\end{equation} \end{ddd}

\begin{ddd} We define the periodic Borel-Moore character as the  composition
\begin{equation}\label{gwergwergwrfregw1}P(\chi_{\BM}):\PH\cX(-,A) \xrightarrow{P(\chi_{BM})} P(HA_{\BM}\bP)\to (\prod_{k\in \Z}\Sigma^{2k} HA_{\BM})\bP:\BC\to \Mod(Hk)\ ,
\end{equation}  \end{ddd}
The   second map in \eqref{gwergwergwrfregw1}  is the canonical one induced by the projections
$P(HA_{\BM}\bP)\to \Sigma^{2k}HA_{\BM}\bP$ for all $k$ in $\Z$.

We consider  \eqref{gwergwergwrfregw1} as a transformation of functors from
$BG\times \BC$ to $\Mod(Hk)$ such that  $G$-acts trivially. Specializing \eqref{boijiodfbfgbdfgbd}
it  naturally induces the Borel-equivariant version 
\begin{equation}\label{fqwedwqedwfrefrdq}P(\chi_{\BM})^{hG}:\PH\cX^{hG}(-,A)  \to ((\prod_{k\in \Z}\Sigma^{2k} HA_{\BM})\bP)^{hG}:G\BC\to \Mod(Hk) 
\end{equation} 
of the  periodic Borel-Moore character.

 \subsection{The topological transgression} \label{kogprwfregergerfewrf}

In this section we construct for  any weak equivariant  Borel-Moore homology theory $E^{G}:G\LCH^{+}\to \cC$   the topological transgression
as a natural transformation
 \begin{equation}\label{fqwedwedwedqwe}
 T^{G,\topp}:E^{G}\bP\to \Sigma E^{G}\circ \partial_{h}:G\BC\to \cC\ .
\end{equation}
It is a homotopy theoretic   generalization 
of the transgression defined in \cite{wulff_axioms}  to the context of all of $G\BC$.

 Assume that $X$ is in $G\BC$ and let $U$ be an invariant coarse entourage of $X$ containing the diagonal.
 We consider $(X,U)$ as an object of $G\BC^{\cC}$ from  \cref{kophprthgertrtgegrtg}.
  By \cite[Lem. 3.6]{Bunke:2024aa} (see \cref{pkhoprtorptgiporetgertgetrg})
  we can define the Higson corona $ \partial_{h} P_{U}(X)$ of the metric Rips complex using  continuous  functions with the condition on the variation instead of  bounded functions as in \cref{kiogwergewrfwrefw}. We have an exact sequence
   \begin{equation}\label{gwerfwrefwefree}0\to C_{0}(P_{U}(X))\to C_{\cB_{P_{U}(X)}}(P_{U}(X))\to C(\partial_{h} P_{U}(X))\to 0
\end{equation}
   of commutative $G$-$C^{*}$-algebras. Recall that $X_{U}$ in $G\BC$ denotes the underlying $G$-bornological space of $X$ equipped with the $G$-coarse structure generated by $U$.
   The inclusion  $X_{U}\to P_{U}(X)$ as the zero skeleton is equivariant  and becomes a   coarse equivalence after forgetting the $G$-action. By \cref{kopbrwgregrefwr} we therefore get an isomorphism
   $$\partial_{h}X_{U}\stackrel{\cong}{\to} \partial_{h} P_{U}(X)$$ of Higson coronas
   in $G\CH$.
   This isomorphism is the component at $(X,U)$ of a natural isomorphism
   of functors
   $$\partial_{h}\stackrel{\cong}{\to}  \partial_{h} P_{-}(-):G\BC^{\cC}\to G\CH\ .$$

    The exact  sequence \eqref{gwerfwrefwefree}  presents an equivariant  compactification
   $\overline{ \iota^{\topp,\prime} (P_{U}(X))}$ of $\iota^{\topp,\prime}(P_{U}(X))$ by $\partial_{h}X_{U}$.
    A morphism $(X,U)\to (X',U')$ in $G\BC^{\cC}$  induces a map of pairs 
   $$(\overline{ \iota^{\topp,\prime} (P_{U}(X))},\partial_{h}X_{U})\to ( \overline{ \iota^{\topp,\prime} (P_{U'}(X'))},\partial_{h}X'_{U'})$$ in $G\CH $. 
      We thus have  a functor  
$$G\BC^{\cC}\to G\CH^{(2)}\ , \qquad (X,U)\mapsto (\overline{ \iota^{\topp,\prime} (P_{U}(X))},\partial_{h}X_{U})\ ,$$
where $G\CH^{(2)}$ denotes the category of pairs in $G\CH$.

   Assume now that $E^{G}:G\LCH^{+}\to \cC$ is a weak equivariant   Borel-Moore homology theory.
   Using the notation \eqref{woeirjojfnwkjerfwerfwr} it induces the functor
   $$E^{G}(-,-):G\CH^{(2)}\to \cC\ , \qquad (X,Y)\mapsto E^{G}(X,Y)\ .$$
  We have the functor $$(-_{2}):G\CH^{(2)}\to G\CH\ , \quad (X,Y)\mapsto Y$$
  and a natural transformation
  $$\partial: E^{G}(-,-)\to \Sigma E^{G}(-_{2}):G\CH^{(2)}\to \cC$$
  which for $(X,Y)$ in $G\CH^{(2)}$ is the boundary map
  $E^{G}(X,Y)\to \Sigma E^{G}(Y)$.
   
     
   We define a natural   transformation
     of functors     \begin{equation}\label{fvwerpojvofvdsvdfvds} E^{G}(\iota^{\topp,\prime} P_{-}(-))\to \Sigma E\circ  \partial_{h}:G\BC^{\cC} \to \cC\end{equation}  whose component at $(X,U)$ in $G\BC^{\cC}$ is the composition  
   \begin{equation}\label{fvwerpojvofvdsvdfvds1}E^{G}(\iota^{\topp,\prime}P_{U}(X)) \simeq  E^{G}( \overline{ \iota^{\topp,\prime} (P_{U}(X))},\partial_{h}X_{U}) \stackrel{\partial}{\to} \Sigma E^{G}( \partial_{h}X_{U})\to   \Sigma E^{G}( \partial_{h}X)  \ ,
\end{equation} where we use the strong excision property of $E^{G}$ for the first equivalence and the last map is induced by the canonical map  $X_{U}\to X$.     
  \begin{ddd} \label{ojohpertgtrgretgtgg}We define the topological transgression
  $$T^{G,\topp} :E^{G}\bP \to  \Sigma E^{G}\circ  \partial_{h}:G\BC\to \cC$$
  as the left Kan-extension of \eqref{fvwerpojvofvdsvdfvds}
   indicated in
  \begin{equation}\label{}\xymatrix{G\BC^{\cC}\ar[dd]\ar@/^0.3cm/[rr]^{\Sigma E^{G}\partial_{h}} \ar@/^-0.3cm/[rr]^{\Uparrow\eqref{fvwerpojvofvdsvdfvds}}_{E(\iota^{\topp,\prime}P_{-}(-))}& &\cC\ar@{=}[dd]\\&&\\G\BC\ar@{..>}@/^0.3cm/[rr]^{\Sigma E^{G}\partial_{h}}\ar@{..>}@/^-0.3cm/[rr]_{E^{G}\bP}^{T^{G,\topp}\Uparrow}&&\cC}
\end{equation}
 \end{ddd}

We refer to \cref{kophertgrtgretgertgertg} for the justification of the domain.
In order to see that we get the correct target we use that
the target of  \eqref{fvwerpojvofvdsvdfvds}  is pulled back along the forgetful functor  $G\BC^{\cC}\to G\BC$.

In our main diagram \eqref{ferwferfwrefw} we need the Borel-equivariant version of the topological transgression
\begin{equation}\label{iogjoiergjowerfwerfw}
(T^{\topp})^{hG}:(E\bP)^{hG}\to \Sigma E^{hG}\circ \partial_{h}\ .
\end{equation}
for a functor
$E$ in $\Fun(BG,\Fun_{\BM}(\CH,\cC))$.  
It is defined as the composition
$$(E\bP)^{hG}\stackrel{def}{\simeq} (\lim_{BG}E\bP)^{u}\xrightarrow{(\lim_{BG}T^{\topp})^{u}} (\lim_{BG} \Sigma E\circ \partial_{h} )^{u}\to
  \Sigma E^{hG}\circ \partial_{h}\ .$$
 The last map is induced by the canonical maps
\begin{align*}  (\lim_{BG} \Sigma E\circ \partial_{h} )^{u}(X)\stackrel{def}{\simeq} &\colim_{\cC_{X}^{G}} \lim_{BG} \Sigma E(\partial_{h}X_{U})\\ &\stackrel{\can}{\to} 
 \lim_{BG}  \colim_{U\in \cC_{X}^{G}}\Sigma E(\partial_{h}X_{U})\stackrel{\can}{\to}  
  \lim_{BG}   \Sigma E(\colim_{U\in \cC_{X}^{G}}\partial_{h}X_{U})
\to    \Sigma E^{hG}( \partial_{h}X )\end{align*}
where the last map  is induced by the canonical map (in fact an isomorphism) $$\colim_{U\in \cC_{X}^{G}}\partial_{h}X_{U}\to\partial_{h}X$$ in $G\CH$.
 %

\begin{rem}\label{oijhopterhertgtrgetrg}
Let $E^{G}:G\LCH^{+}\to \cC$ be a weak equivariant Borel-Moore homology theory. 
Then by \cref{khoeprthretgertge} and similarly as in \cref{jhtrhoeriotgertgetrhehhfj} we can construct a factorization
$$\xymatrix{&&\Sigma R(E^{G} \partial_{h})\ar[d] \\E^{G}\bP\ar[rr]^{T^{G,\topp}}\ar@{..>}[urr]&&\Sigma E^{G}\partial_{h}}\ .$$ The dotted arrow is the best approximation of the topological transgression
by a natural transformation between equivariant coarse homology theories.
Non-vanishing of $T^{G,\topp}$ implies non-vanishing of  $R(E^{G} \partial_{h})$ and of the dotted arrow.

Consider, e.g., the example of a trivial group and $X=\R^{n}$ in $\BC$ with the metric structures. Then we have a canonical morphism $q:\partial_{h} \R^{n}\to S^{n-1}$. Since $\R^{n}$ is coarsifying (see \cite[Def. 7.8 and Prop. 7.13]{ass})
we have $\Sigma^{n}E(*)\simeq E(\R^{n})\simeq E \bP(\R^{n})$ and
the composition \begin{equation}\label{gerfwerfweg542zze} \Sigma^{n}E(*)\simeq E \bP(\R^{n})\stackrel{T^{\topp}}{\to} \Sigma E(\partial_{h}\R^{n}) \stackrel{q}{\to} \Sigma E(S^{n-1})\stackrel{\simeq}{\to} \Sigma^{n}E(*)\oplus \Sigma E(*)
\end{equation} 
is the inclusion of the first summand.  \hB
\end{rem}

\section{Construction of fillers}\label{jogopwefewrfewfd}

%
%
%
%
%
%
%
%
%
%
%
%
%
%
%
%
%
%
%
%
%
%
%
%

 \subsection{Cone functors}\label{oijgowpererfwerfwref}

 In this subsection we recall the definitions of the cone functors
 $$\cO, \cO^{\infty}:G\UBC\to \BC$$ and the cone sequences \eqref{gerijoki0o5g4g}, \eqref{ferfwerfwerffrewe}.
 We refer to \cite[Sec. 5.3.3]{buen}, \cite[Sec. 8]{ass} for the non-equivariant, and to \cite[Sec. 9.4]{equicoarse} for the equivariant case. As there are slightly varying conventions how to treat the time variable we give complete definitions here.
 
Let $\iota:G\UBC\to G\BC$ be the forgetful functor. We first describe the functor $\cO^{\infty}$.   For $X$ in $G\UBC$ 
the cone $\cO^{\infty}(X)$ in $G\BC$ is obtained from $\iota(\R\otimes X)$ in $G\BC$  by taking the new $G$-coarse structure $\cC_{\cO^{\infty}(X)}$ consisting of all coarse entourages $U$ of $\R\otimes X$ with the additional property that for every uniform entourage
$V$ of $X$ there exists $s$ in $\R$ such that $U_{t}\subseteq V$ for  every $t$ in $[s,\infty)$, where
$U_{t}\subseteq X\times X$ is the restriction of $U$ to the slice $ (\{t\}\times X\})\times  (\{t\}\times X\})$.
 If $f:X\to X'$ is a morphism in $G\UBC$, then the map of sets    $ \id_{\R}\times f: \R\times X\to  \R\times X'$
  functorially induces a morphism $\cO^{\infty}(f):\cO^{\infty}(X)\to \cO^{\infty}(X')$.

 The cone $\cO(X)$ in $G\BC$  is defined as the subspace \begin{equation}\label{regwerfwerf234}\cO(X):=\{(t,x)\in \cO^{\infty}(X)\mid t\ge 0\}
\end{equation} of $\cO^{\infty}(X)$ with the induced structures.     \begin{ddd} \label{koopehrrtgetrgeg}The functors $\cO^{\infty},\cO:G\UBC\to G\BC$ described above are called the (geometric) cone-at-$\infty$ functor and the cone functor.
 \end{ddd}

 For $X$ in $G\UBC$ we have an inclusion
 $\iota X\to \cO(X)$, $x\mapsto (0,x)$. The cone boundary
 $\partial^{\cone}:\cO^{\infty}(X)\to \R\otimes X$
 is given by the identity of underlying $G$-sets.
 The sequence
 \begin{equation}\label{gerijoki0o5g4g}\iota X\to \cO(X)\to \cO^{\infty}(X)\stackrel{\partial^{\cone}}{\to} \iota(\R\otimes  X)
\end{equation}
 is called the geometric cone sequence.
Applying the universal equivariant coarse homology functor $\Yo^{G}:G\BC\to G\Sp\cX$ from \eqref{gweroihgwuierjfoiewrfiojwerofwerfw} and using  $\Yo^{G}\iota(\R\otimes X)\simeq \Sigma \Yo^{G} \iota X$ we get a fibre sequence
 \begin{equation}\label{ferfwerfwerffrewe}
 \Yo^{G}\iota X\to \Yo^{G}\cO(X)\to \Yo^{G}\cO^{\infty}(X)\to \Sigma \Yo^{G} \iota X
 \end{equation}
 in $G\Sp\cX$
 which is also called the cone sequence \cite[Cor. 9.30]{equicoarse}.
 The fibre sequence \eqref{ferfwerfwerffrewe} shows that
 $\Yo^{G}\cO^{\infty}(X)\simeq \Cofib(\Yo^{G}\iota X\to \Yo^{G}\cO(X))$ which motivates the name 
 "cone-at-$\infty$" for $\cO^{\infty}$.
 
 If $S$ is a $G$-set, then 
  $\cO(S_{min,min,disc})\cong [0,\infty)\otimes S_{min,min}$ is flasque. Hence in view of the fibre sequence  \eqref{ferfwerfwerffrewe}
 the cone  boundary induces an equivalence
 \begin{equation}\label{twrwrt344545e}
  \Yo^{G}\cO^{\infty}(S_{min,min,disc})\stackrel{\simeq}{\to} \Sigma \Yo^{G}S_{min,min}\ ,
\end{equation}
see also \cite[Prop.9.35]{equicoarse}.

 \begin{rem}\label{hgwiueghergwergwe}
If $E:G\BC\to \cC$ is a strong coarse homology theory, then the compositions of $E$ with the cone functors 
 are equivariant  local homology theories $$E\cO^{\infty}, E\cO:G\UBC\to \cC\ ,$$ see  \cite[Lem. 9.2]{ass} for the non-equivariant case, and also \cref{okhprertgertgrtger}. They are in particular homotopy invariant, annihilate flasques and   excisive for decompositions of simplicial complexes. So by \cref{pkehpetrherththgrtg} they can serve as an input for the coarsification. 
 
 In  the non-equivariant case,  the functor $E\cO^{\infty} \bP$ is the domain of the coarse assembly map whose fibre is represented by $E\cO  \bP$, see  \cite[Sec. 9]{ass}.
 \hB\end{rem}

 \subsection{Normalization at a point}\label{ojgwpegerfewrfwref}

This section provides our starting point for the stepwise construction of fillers of our squares.  We show that the homotopy Chern character
can be normalized such that  the diagram  \eqref{gwerwefwerfrefffw} commutes   when we insert
$Z=*$.

\begin{prop}\label{kophertgertgterg}
There exists a unique choice of the equivalence  \eqref{vdsfpovkspdfvfdvsf} used in the construction of $\ch^{h}$ such that the big cell in the diagram \eqref{gwerwefwerfrefffw}  commutes at $Z=*$.
\end{prop}
\begin{proof}
We evaluate the diagram \eqref{gwerwefwerfrefffw} at $Z=*$. In the lower right corner we have 
$$\Sigma (\prod_{k\in \Z} \Sigma^{2k} H\C_{\BM})(\iota^{\topp}_{u}(*))\simeq \prod_{k\in \Z} \Sigma^{2k+1} H\C \ .$$
If $F$ is any spectrum, then the datum of a map
$f:F\to  \prod_{k\in \Z} \Sigma^{2k+1} H\C$ is equivalent to a family $(\pr_{k}\circ f)_{k\in \Z}$ of maps $\pr_{k}\circ f:F\to \Sigma^{2k+1}H\C$. For each $k$
such a map is in turn  uniquely determined up to equivalence by a map
$f_{k}:\pi_{2k+1}F\to \pi_{2k+1}\Sigma^{2k+1}H\C\cong \C$.
 
 By \eqref{twrwrt344545e}
 the cone boundary   induces an equivalence \begin{equation}\label{gijoergrtg}\partial^{\cone}:\Yo(\cO^{\infty}(*))\stackrel{\simeq}{\to}\Sigma \Yo(*)\ .
\end{equation} Hence $\Yo(\cO^{\infty}(*)) $  belongs to the localizing subcategory \eqref{fqwewefedqedew} 
(which in the present non-equivariant case is the one generated by $\Yo(*)$),
and the comparison map $c $ from \cref{kopgwegwerfrfwref1}  is an equivalence at 
$\cO^{\infty}(*)$ by  
\cref{gkopwreferfwerfrwef}.
The upper left corner of  \eqref{gwerwefwerfrefffw} at $Z=\cO^{\infty}(Z)$ can therefore be calculated
 using the equivalence $$ K\cX^{\ctr}(\cO^{\infty}(*))\stackrel{c}{\simeq} K\cX(\cO^{\infty}(*))\stackrel{\partial^{\cone},\eqref{gijoergrtg}}{\simeq} \Sigma K\cX(*)\simeq \Sigma KU \ .$$ We have 
$\pi_{2k+1} \Sigma KU=\Z$ and  $\pi_{2k}\Sigma KU\cong 0$ for every $k$ in $\Z$.   
 
In view of these calculations 
the transformations $\phi$ and $\psi$ are   given by families $(\phi_{k})_{k\in \Z}, (\psi_{k})_{k\in \Z}$
of maps $\phi_{k},\psi_{k}:\Z\to \C$. 
Thereby $\phi$ depends on the choice of the equivalence \eqref{vdsfpovkspdfvfdvsf} and $\psi$ does not.

We must show that we can normalize \eqref{vdsfpovkspdfvfdvsf}   and therefore $\ch^{h}$ uniquely such that 
$\phi_{k}=\psi_{k}$ for all $k$ in $\Z$.  The map $\phi_{k}$ depends homogeneously on
the component $$\Z\cong \pi_{2k+1}\Sigma KU\to \pi_{2k+1} \Sigma KU\wedge H\C\stackrel{\eqref{vdsfpovkspdfvfdvsf} }{\simeq} \pi_{2k+1} \prod_{k\in \Z} \Sigma^{2k+1} H\C \stackrel{\pr_{k}}{\to}  \pi_{2k+1}\Sigma^{2k+1}H\C\cong \C$$
 of \eqref{vdsfpovkspdfvfdvsf}.
 If we fix one non-zero choice and show that the resulting $\phi_{k}$ is not zero, then there is a unique normalization of this component so that $\phi_{k}=\psi_{k}$.

 \begin{lem}\label{ophertgertgetrge}
 The component $\phi_{k}$ is non-zero.
 \end{lem}
 \begin{proof}
By   \cref{khopperttrgeg} below we have a commutative triangle
$$
 \xymatrix{K\cX(\cO^{\infty}(*))\ar[rr]^{p^{\an}}\ar[dr]^{T^{\an}}&&\Sigma K^{\an}(\iota^{\topp}_{u}(*))\\&\Sigma K^{\an}(\partial_{h}\cO^{\infty}(*))\ar[ur]^{\pi}&}\ .$$
 Note that the map $T^{\mot}$ in \eqref{gweihrguiwheirfhiwefwerfwefwerf} is the identity and $k$ in \eqref{gweihrguiwheirfhiwefwerfwefwerf} can be identified with $\pi$   from \eqref{ewrferfrfvsdfvfd} when $X=\cO^{\infty}(Y)$ as in the our present situation. By \cite[Thm. 1.5]{bel-paschke} the morphism $p^{\an}$ is an equivalence. Therefore $\pi\circ T^{\an}$ is an equivalence, too. The component of $\ch^{h}$ on $*$ is precisely
$$\Sigma K^{\an}(*)\simeq \Sigma KU \to  \Sigma KU\wedge H\C\stackrel{\eqref{vdsfpovkspdfvfdvsf} }{\simeq}   \prod_{k\in \Z} \Sigma^{2k+1} H\C  \ .$$
 It follows that
 $$\phi_{k}:\Z\cong \pi_{2k+1} K\cX^{\ctr}(\cO^{\infty}(*))\xrightarrow{ \pr_{k}\circ \ch^{h} \circ \pi\circ T^{\an}\circ c} \pi_{2k+1} \Sigma^{2k+1}H\C\cong \C$$ is injective, in particular non-zero. 
 \end{proof} The \cref{ophertgertgetrge} finishes the proof of   \cref{kophertgertgterg}
 \end{proof}

  \subsection{The  motivic transgression}\label{gwegregweg}

  In this section we introduce the concept of the motivic transgression. 
  It will be used to extend the commutativity of  the diagram \eqref{gwerwefwerfrefw} from the case where
  $X$ is a cone over some uniform bornological coarse space $Y$ to the situation where $X$ only contains
   such a cone coarsely. We furthermore describe the corresponding projection map on the level of Higson coronas.
    
     We start with recalling   some basics about coarsely excisive pairs.
   Let $X$ be in $G\BC$.
    According to  \cite[Def. 3.3]{equicoarse} a subset $C$ of $X$ is nice if the inclusion $C\to U[C]$ is a coarse equivalence  for every invariant coarse entourage $U$ of $X$ containing the diagonal.
  Let  $A,B$ be invariant subspaces of $X$ such that $A\cup B=X$.  
 The following is the corrected version of
    \cite[Def. 4.13]{equicoarse} as stated in \cite[Def. 2.41]{Bunke:2021ab}.
    \begin{ddd}\label{kohperthertge}
    The pair $(A,B)$ is coarsely excisive, if 
    $A$, $A\cap B$ are nice in $X$ and $V[A]\cap B$ is nice in $X$ for a cofinal family of invariant entourages $V$ of $X$, and for every  coarse entourage $U$ of $X$ there exists another coarse entourage $V$ such that
    $U[A]\cap U[B]\subseteq V[A\cap B]$.
    \end{ddd}

     \begin{rem}\label{korethkpertgertgertg}
   The point of  \cref{kohperthertge}  is that it implies for a coarsely excisive pair $(A,B)$  that the square
  $$ \xymatrix{\Yo^{G}(A\cap B)\ar[r]\ar[d] & \Yo^{G}(A)\ar[d] \\ \Yo^{G}(B)\ar[r] &\Yo^{G}(X) } $$
   is a push-out in $G\Sp\cX$. \footnote{The conditions stated in  \cite[Def. 4.13]{equicoarse} are not sufficient to get this conclusion as stated in \cite[Cor. 4.14]{equicoarse}.}\hB
\end{rem}
   
        \begin{lem} \label{kohperthertgertgertg}If   $$\xymatrix{A\cap B\ar[r]\ar[d] &A \ar[d] \\ B\ar[r] &X } $$
   is a push-out in $G\BC$, $A$ and $A\cap B$ are nice in $X$,  and
   $V[A]\cap B$ is nice in $X$ for a cofinal family of invariant entourages $V$ of $X$, then the pair
   $(A,B)$ is coarsely excisive. 
   \end{lem}
   \begin{proof}
   We must verify the last condition stated in \cref{kohperthertge}.
   We define a coarse structure $\cC'$ on $X$ as follows:
   An entourage $U$ of $X$ belongs to $\cC'$ if
there exist coarse entourages $V$ of $A$ and $W$ of $B$ such that
$$U\subseteq V\cup W\cup V\circ W\cup W\circ V\ .$$
We have $\diag(X)\in \cC'$ and $\cC'$ is closed under taking subsets.
We have
$U^{-1}\subseteq V^{-1}\cup W^{-1}\cup W^{-1}\circ V^{-1}\cup V^{-1}\circ W^{-1}$ so that $\cC'$ is closed under the flip.
If $U'\subseteq  V'\cup W'\cup V'\circ W'\cup W'\circ V'$, then  
\begin{equation}\label{gwerfewrfvdfsv}U\cup U'\subseteq V\cup V' \cup W\cup W'\cup (V\cup V')\circ (W\cup W')\cup (W\cup W')\circ (V\cup V')\ .
\end{equation}
We see that $\cC'$ is closed under finite unions.
We now show that $\cC'$ is closed under compositions.   We note that
$U\circ U'$ is the union of sixteen terms. In view of \eqref{gwerfewrfvdfsv} it suffices to see that all of them belong to $\cC'$.
The two terms of the form $V\circ V'$ or $W\circ W'$ clearly belong to $\cC'$.
There are further six terms of the form $$V\circ W' \ , W\circ V' \ ,V\circ V'\circ W'\ , W\circ W'\circ V'\ , V\circ W\circ W'\ , W\circ V\circ V'$$ which belong to  $\cC'$. 
There are six terms
$$ V\circ W'\circ V'\ , W\circ V'\circ W'\ , V\circ W\circ V'\ , W\circ V\circ W'\ ,   V\circ W\circ W'\circ V'\ , W\circ V\circ V'\circ W'\ ,$$ 
and finally the  two terms 
$$    V\circ W\circ V'\circ W' \ ,V\circ W\circ W'\circ V'  \ .$$
All these terms are shown to be in $\cC'$ by the following kind of argument.
 The restriction of $W\circ V'$ to $A\cap B$ is an entourage of $A\cap B$.
Hence we  have $V\circ W\circ V'  \subseteq V''$ and $ V\circ W\circ V'\circ W' \subseteq V''\circ W'$ for suitable $V''$. A similar reasoning applies to the other terms. 
We conclude that $U\circ U'\in \cC'$.  
 
 The arguments above show that  $\cC'$ is a coarse structure.
We have $\cC'\subseteq \cC$ and $\cC'$ contains the generators of $\cC$. It follows that $\cC'=\cC$.
We now observe that
$$U[A]\cap U[B] \subseteq (V\cup W\cup V\circ W\cup W\circ V)[A\cap B]\ .$$   \end{proof}


 
%

We consider $X'$ in $G\BC$ and $Y$ in $G\UBC$. We form  the cone
$\cO(Y)$ in $G\BC$ according to    \cref{koopehrrtgetrgeg}, and further assume a coarse subspace inclusion
$\cO(Y)\hookrightarrow X'$. We let $Z:=X'\setminus (0,\infty)\times Y$ denote the complement of this cone.   
We then define $\hat X'$ in $G\BC$ by the push-out
 \begin{equation}\label{gwerfwerfw3fweferfwer}\xymatrix{ (-\infty,0]\otimes  \iota Y \ar[r]\ar[d] &\cO^{\infty}(Y)  \ar[d]^{j} \\ (-\infty,0]\otimes Z \ar[r] &\hat X' } \ .
\end{equation}  We have a canonical inclusion $i:X'\to \hat X'$ which identifies $\cO(Y)$ as a subspace of $X'$ with the positive part of the cone $\cO^{\infty}(Y)$ and the subspace $Z$ of $X'$ with $\{0\}\times Z$.
   \begin{lem} The decomposition $( (-\infty,0]\otimes Z,\cO^{\infty}(Y))$ of $\hat X'$ is coarsely excisive.
   \end{lem}
   \begin{proof}
We want to   apply \cref{kohperthertgertgertg} for $A= (-\infty,0]\otimes Z$ and $B= \cO^{\infty}(Y)$.
To this end  it suffices to check that $V[(-\infty,0]\otimes Z)]\cap \cO^{\infty}(Y)$ is nice
   in $\hat X$ for every invariant coarse entourage $V$ of $\hat X'$, and that 
$(-\infty,0]\otimes Z$  is nice in $X$.      The case of $V=\diag(X)$ then settles niceness of $(-\infty,0]\otimes Z)\cap \cO^{\infty}(Y)$.
   
   
   For the first condition  must show that 
   for any invariant  coarse entourage
   $U$ of $\hat X'$   the map 
   $$V[(-\infty,0]\otimes Z]\cap \cO^{\infty}(Y) \to U[V[(-\infty,0]\otimes Z]\cap \cO^{\infty}(Y)]$$
   is an equivariant  coarse equivalence. 
  The complement of $ V[(-\infty,0]\otimes Z] \cap \cO^{\infty}(Y)$ in $U[V[(-\infty,0]\otimes Z]\cap \cO^{\infty}(Y)]$ consists of points $(t,y)$ in $\cO^{\infty}(Y)$ with $t\in [0,c]$ for some $c$ only depending on $U$ and $V$.
 An equivariant  coarse  inverse of the inclusion is given by
   the map
   which  is the identity on $V[(-\infty,0]\otimes Z]\cap \cO^{\infty}(Y)$ and sends the point $(t,y)$   in the complement of this set to $(0,y)$ 
   
 We now show that  $(-\infty,0]\otimes Z$  is nice in $X$.  Let $U$ be an invariant  coarse entourage    of $X$.
 We must show that
 $(-\infty,0]\otimes Z\to U[(-\infty,0]\otimes Z]$ is an equivariant coarse equivalence.
  We define a coarse inverse $ U[(-\infty,0]\otimes Z]\to (-\infty,0]\otimes Z$ as follows.
  This map is the identity on $ (-\infty,0]\otimes Z$. Any point in $U[(-\infty,0]\otimes Z]\setminus (-\infty,0]\otimes Z$
  is of the form $(t,y)$ for $t\in [0,c]$ with $c$ only depending on $U$. Our map sends this point to $(0,y)$ in $(-\infty,0]\otimes Z$.
   \end{proof}

   The
  flasqueness of the spaces in the left column of \eqref{gwerfwerfw3fweferfwer} shows that the inclusion $j$ induces an equivalence \begin{equation}\label{ufwiuew9fuwq9e8ud9qwdwed}\Yo^{G}(j):\Yo^{G}(\cO^{\infty}(Y))\stackrel{\simeq}{\to} \Yo^{G}(\hat X')\ .
\end{equation}
 The map 
 \begin{equation}\label{gwerfwerfrwerfref234} \Yo^{G}(X')\stackrel{i}{\to} \Yo^{G}(\hat X') \stackrel{j,\simeq}{\leftarrow} \Yo^{G}(\cO^{\infty}(Y))
\end{equation}
 in $G\Sp\cX$ is a special case of the motivic transgresion.
 
   We now generalize the geometric situation.
 Let $\cO(Y)\to X'$  be as above,   
  $X$ be in $G\BC$,  and $\psi:X\to X'$ be a morphism in 
 $G\BC$. 
 \begin{ddd}\label{kophrgertgertrgertgergert}
 In this situation we say that the motivic transgression 
 \begin{equation}\label{bsoidjfiovwefsdvsfdvsdfvsdfsv}T^{\mot}: \Yo^{G}(X)\xrightarrow{\Yo(\psi)} \Yo^{G}(X')\xrightarrow{\eqref{gwerfwerfrwerfref234}}  \Yo^{G}(\cO^{\infty}(Y))\ .
\end{equation} is defined.
 \end{ddd}

 The construction of the motivic transgression is complemented by   projection maps on the level of Higson coronas which we describe now.
 In $X'$ we consider
the  invariant big family 
 $\cY$    generated by the  bounded subsets of $\{0\}\times Y$. Note that $\cY\subseteq \cB$, where $\cB$ denotes  the bornology of $X'$.
 We define a homomorphism of commutative $G$-$C^{*}$-algebras \begin{equation}\label{ewrfwerfwrfwefer}C_{u,0}(Y)\to \ell^{\infty}_{\cY}( X')\ ,  \quad \phi\mapsto \tilde \phi\ ,
\end{equation}
where $\tilde \phi(t,y):=\phi(y)$ for $(t,y)$ in $\cO(Y)$ and $\phi(x)=0$ for $x\in Z$
(the justification is similar is \cite[Lem. 5.7]{bel-paschke}).
This gives a homomorphism of commutative $G$-$C^{*}$-algebras
$$C_{u,0}(Y)\to \ell^{\infty}_{\cY}(X')/\ell^{\infty}(\cY)\to  \ell^{\infty}_{\cB}(X')/\ell^{\infty}(\cB)=C(\partial_{h}X')\ .$$
Its Gelfand dual is a map
\begin{equation}\label{gwerfwerfwerf252}
\partial_{h}X'\to \iota^{\topp}_{u}(Y)\ .
\end{equation}
\begin{ddd} If the motivic transgression  \eqref{bsoidjfiovwefsdvsfdvsdfvsdfsv} is defined, then we define the map 
 \begin{equation}\label{ferfjwerfmerfireojoijo}k:  \partial_{h} X\xrightarrow{\partial_{h}\psi}\partial_{h}  X'\xrightarrow{\eqref{gwerfwerfwerf252}} \iota_{u}^{\topp}(Y)\ .
\end{equation}
\end{ddd}

 \begin{ex} Let $Y$ be in $G\UBC$.
 For $X=X'=\cO(Y)$ the motivic transgression 
 $T^{\mot}:\Yo^{G}(\cO(Y))\to \Yo^{G}( \cO^{\infty}(Y))$
 is defined and induced by the canonical inclusion $\cO(Y)\to \cO^{\infty}(Y)$.
 In this case the map $k$ is given by the composition
\begin{equation}\label{kghguhwiuerhgiuergfwerfwerfwerf}k:\partial_{h} \cO(Y)\to\partial_{h} \cO^{\infty}(Y)\stackrel{\eqref{ewrferfrfvsdfvfd}}{\to} \iota^{\topp}_{u}(Y)\ .
\end{equation} \hB
 \end{ex}

 \subsection{The analytic Paschke morphism}

The following theorem compares the $K$-theoretic transgression introduced in   \cref{rguweruigowerferfrfwr}  with the 
 analytic Paschke transformation
\begin{equation}\label{verfwerffvsfdv}p^{G,\an}:K\cX^{G}_{\bC,A,G_{can,max}}\circ \cO^{\infty}(-)  \to   K^{G,\an}_{\bC,A}\circ \iota_{u}^{\topp}:G\UBC\to \Mod(KU)
\end{equation} from  \cite[Thm. 1.5 \& Def. 6.3]{bel-paschke}.
 We assume that the motivic transgression is defined according to \cref{kophrgertgertrgertgergert}.

\begin{theorem}\label{khopperttrgeg}
We have a commutative square \begin{equation}\label{gweihrguiwheirfhiwefwerfwefwerf}\xymatrix{K\cX^{G}_{\bC,A,G_{can,max}}( X)\ar[d]_{T^{\mot}}\ar[r]^{T^{G,\an}} &\ar[d]^{k, \eqref{ferfjwerfmerfireojoijo}}
\Sigma K^{G,\an}_{\bC,A}( \partial_{h}X)  \\ 
 K\cX^{G}_{\bC,A,G_{can,max}}( \cO^{\infty}(Y))\ar[r]^-{p^{G,\an}}
& \Sigma K^{G,\an}_{\bC,A }( \iota_{u}^{\topp}(Y)) 
}\ . 
\end{equation}
\end{theorem}
\begin{proof}
Both maps  $k\circ T^{G,\an}_{X}$ and $p_{Y}^{G,\an}\circ T^{\mot}$ are induced by the composition of multiplication maps of $C^{*}$-categories  $$\mu,\mu':C_{u,0}(Y)\otimes \bV_{\bC}^{G}(X)\to \bQ^{(G)}_{\std}$$  with the diagonal $\delta_{Y}$ similar as in \eqref{vwiejviowevdsfvsdfv}. 
 The result will follow from a comparison of these multiplication  maps.
 
  We first describe $\mu$ by unfolding the definitions.  It sends the object $(C,\rho,\nu)$ of $C_{u,0}(Y)\otimes \bV_{\bC}^{G}(X) $  to the object $(C,\rho)$  in $ \bQ^{(G)}_{\std}$ and the morphism $\phi \otimes A: (C,\rho,\nu)\to (C',\rho',\nu')$ in $C_{u,0}(Y)\otimes \bV_{\bC}^{G}(X) $ to  the morphism $[\tilde \phi A]$ in $ \bQ^{(G)}_{\std}$,
  where  $[\tilde \phi A]$ is an abbreviation for
  $[\nu'(\psi^{*}\tilde \phi)A]$, see  \eqref{ewrfwerfwrfwefer} and
 \eqref{giouweiorjgoiwejfioeferwfw43} for notation.
  
 Unfolding \cref{kophrgertgertrgertgergert} the map $T^{\mot}$ is induced by the functor
 \begin{equation}\label{vdsfoijvosdfvsvfsfv} \bV_{\bC}^{G}(X)\stackrel{\psi_{*}}{\to}  \bV_{\bC}^{G}(\hat X')\to  \bV_{\bC}^{G}(\hat X')/ \bV_{\bC}^{G}(\{(-\infty,0]\times Z\}\subseteq \hat X')
\end{equation}  followed by an inverse of the unitary equivalence induced by $j$ in \eqref{gwerfwerfw3fweferfwer}
  $$   \bV_{\bC}^{G}(\cO^{\infty}(Y))/ \bV_{\bC}^{G}(\{ (-\infty,0]\times Y\}\subseteq \cO^{\infty}(Y))\to \bV_{\bC}^{G}(\hat X')/ \bV_{\bC}^{G}(\{(-\infty,0]\times Z\}\subseteq \hat X')$$ (see \eqref{ufwiuew9fuwq9e8ud9qwdwed} for the motivic version).
  For the inverse we choose a functor
  which sends the object $(C,\rho,\nu')$  in $\bV_{\bC}^{G}(\hat X')$ to the object $(\nu'(\cO^{\infty}(Y))C, \nu'(\cO^{\infty}(Y))\rho,\nu'_{|\cO^{\infty}(Y)})$
  and the morphism    $[A]$  in $\bV_{\bC}^{G}(\hat X')/ \bV_{\bC}^{G}(\{(-\infty,0]\times Z\}\subseteq \hat X')$ to $[\nu'(\cO^{\infty}(Y))A]$ in $ \bV_{\bC}^{G}(\cO^{\infty}(Y))/ \bV_{\bC}^{G}(\{ (-\infty,0]\times Y\}\subseteq \cO^{\infty}(Y))$. Implicitly we have chosen images of the projections $\nu'(\cO^{\infty}(Y))$.

  The composition  \eqref{vdsfoijvosdfvsvfsfv} sends $(C,\rho,\nu)$ in $ \bV_{\bC}^{G}(X)$ to
  $(\nu'(\cO(Y))C, \nu'(\cO(Y))\rho, \nu_{|\cO(Y)})$ and
  $A$ to $[\nu'(\cO(Y))A]$, where $\nu':=\psi_{*}\nu$. 
  
   The   morphism $ \mu'$ is now given by the functor which sends
  $(C,\rho,\nu)$ to $(\nu'(\cO (Y))C,\nu'(\cO (Y))\rho)$ and
  $\phi\otimes A$ to $[\tilde \phi  \nu'(\cO (Y))A]$. 
  
  Since $\tilde \phi A=\tilde \phi  \nu'(\cO (Y))A$ we see that
  $\mu$ and $\mu'$  are MvN equivalent. This MvN-equivalence is implemented by the family of isometries
  $v_{(C,\rho,\nu)}:\nu'(\cO(Y))C\to C$. Since $$\ee(\mu)\simeq \ee(\mu'):\EE^{G}(C_{u,0}(Y),C_{u,0}(Y)\otimes \bV_{\bC}^{G}(X)\otimes A)\to \EE^{G}(C_{u,0}(Y),\bQ^{G}_{\std}\otimes A)$$ we conclude that 
  $$k\circ T^{\an}_{X}\simeq \mu\circ \delta_{Y}\simeq  \mu'\circ \delta_{Y}\simeq p^{\an}_{Y}\circ T^{\mot}\ .$$
  
%
%
%
%
%
%
%
%
%
%
%
\end{proof}
%
%
%

The commutative square in \eqref{gweihrguiwheirfhiwefwerfwefwerf} has the draw-back that the right lower corner
is a locally finite functor in the variable $Y$ while the left lower corner is not known to  be locally finite. So we do not expect that the lower horizontal map $p^{G}_{Y,\an}$ is an equivalence for  unbounded $Y$ in general.  In the following we improve  this point.

To any functor $E:G\UBC\to \cC$ to a   stable and  complete target we functorially associate a locally finite version
$E^{\lf}:G\UBC\to \cC$ and a natural transformation \begin{equation}\label{fiwqehiudhqwidqwedqwedqd} E\to E^{\lf}
:G\UBC\to \cC\end{equation} (in analogy to \eqref{ergwerfdsgwer}) such that
$$E^{\lf}(Z):=\lim_{W} \Cofib(E(Z\setminus W)\to E(Z))\ ,$$
where $W$ runs over all invariant $G$-bounded subsets of $Z$ (see \eqref{goijwioejorigwerfwerf}).
This construction preserves homotopy invariance and excisiveness by the analogue of \cref{kopgpwerrefweferfw}.
Following the conventions from \cite{bel-paschke} we abbreviate
$$\Sigma K^{G,\cX}_{\bC ,A}(-):= K\cX^{G}_{\bC,A,G_{can,max}}\circ \cO^{\infty}(-):G\UBC\to \Mod(KU)\ .$$
Since   $K^{G,\an,}_{\bC,A}$ is a Borel-Moore homology theory by \cref{ogpwerferfewrfwef} and therefore a weak  Borel-Moore homology theory by \cref{ophertgertgertgertg} it follows from the $\iota^{\topp}_{u}$ version
of \cref{8kothpetrhtrheth9} (with the same proof) that $K^{G,\an,}_{\bC,A}\circ \iota^{\topp}_{u}$ is already locally finite.
We can extend the diagram in \cref{khopperttrgeg} as follows:
$$\xymatrix{K\cX^{G}_{\bC,A,G_{can,max}}( X)\ar[d]_{T^{\mot} }\ar[r]^{T^{G,\an}_{X}} &\ar[d]^{k}
\Sigma K^{G,\an}_{\bC,A}( \partial_{h}X)  \\  \Sigma K^{G,\cX}_{\bC,A }( Y)\ar[r]^{p^{G,\an}_{Y}}\ar[d]^{\eqref{fiwqehiudhqwidqwedqwedqd}}&
\Sigma K^{G,\an}_{\bC,A}(\iota_{u}^{\topp}(Y)) 
   \\  \Sigma K^{G,\cX,\lf}_{\bC ,A}( Y) \ar[ur]^{p_{Y}^{G,\an,\lf}}&  } \ .$$

The Paschke transformation is a natural transformation between 
functors which are homotopy invariant and excisive for cell attachments.
By \cite[Thm. 1.5]{bel-paschke} it is an equivalence on objects in $G\UBC$ of the form $S_{min,min}$ for
$S$ in $G_{\Fin}\Orb$.  
\begin{ddd}
We call $W$ in $G\UBC$ locally finite if it is a retract of a locally finite $G$-$CW$-complex with finite stabilizers.  \end{ddd}

The following is an immediate consequence of  \cite[Thm. 1.5]{bel-paschke} and the definitions.

\begin{kor}If $W$ is locally finite, then
$$p^{G,\an,\lf}:K^{G,\cX,\lf}_{\bC,A}(W)\to K^{G,\an,\lf}_{\bC,A}(\iota^{\topp}_{u}(W))$$ is an equivalence.
\end{kor}

\begin{kor}
If $Y$ is locally finite, then we have a commutative square
$$\xymatrix{K\cX^{G}_{\bC,A,G_{can,max}}( X)\ar[d]_{\eqref{fiwqehiudhqwidqwedqwedqd}\circ T^{\mot} }\ar[r]^{T^{G,\an}_{X}} &\ar[d]^{ k}
\Sigma K^{G,\an}_{\bC,A}( \partial_{h}X)  \\  \Sigma K^{G,\cX,\lf}_{\bC,A }( Y)\ar[r]^{p^{G,\an,\lf}}_{\simeq}& \Sigma K^{G,\an}_{\bC,A }( \iota^{\topp}_{u}(Y))  } \ .$$
 \end{kor}

%
%

 \subsection{The topological Paschke map}\label{htekoperhrtgrt}

   Let $E:G\LCH^{+}\to \cC$ be a weak equivariant  Borel-Moore homology theory (\cref{kohperthrgrtgertgeg}). We will construct the topologial Paschke morphism 
as a natural transformation
    \begin{equation}\label{gweiojrfiowerfwerfwerfwe1}p^{G,\topp}:E \bP\circ \cO^{\infty}\to \Sigma E\circ \iota_{u}^{\topp}:G\UBC\to \cC\ .\end{equation}
  It is an analogue of the analytic Paschke morphism \eqref{verfwerffvsfdv}.
  
  \begin{rem}
  The domain of the analytic Paschke morphism  \eqref{verfwerffvsfdv} is an equivariant local homology theory 
  on $G\UBC$, i.e., homotopy invariant, excisive, $u$-continuous and trivial on flasques, see \cref{hgwiueghergwergwe}.
 This follows from  (the equivariant version of) \cite[Lem. 9.6]{ass} since $K\cX^{G}_{\bC,A,G_{can,max}}$ is a strong coarse homology theory. In contrast, the equivariant coarse homology theory $E\bP$ in the domain of the topological Paschke morphism \eqref{gweiojrfiowerfwerfwerfwe1} is not known to be strong.
 Therefore the domain of the topological Paschke morphism $E \bP\circ \cO^{\infty}$ is not known to be  a local homology theory since it may not  annihilate flasques. \hB
  \end{rem}

  We start with describing    the transformation \eqref{ewrferfrfvsdfvfd}
 \begin{equation}\label{gwreggjweiog}
 \pi:\partial_{h} \cO^{\infty}\to \iota_{u}^{\topp}:G\UBC\to G\LCH^{+}\ .
\end{equation}
in detail.
 Its component at $Y$ in $G\UBC$  is  the morphism \begin{equation}\label{}
\pi_{Y}:\partial_{h}\cO^{\infty}(Y)\to   \iota^{\topp}_{u}(Y) 
\end{equation}
defined as the Gelfand dual of the homomorphism of unital commutative $G$-$C^{*}$-algebras
 \begin{equation}\label{fqwoihdioewdewdqedq}C_{u,0}(Y)
 \to  \frac{\ell_{\cB_{\cO^{\infty}(Y)}}^{\infty}(\cO^{\infty}(Y))}{\ell^{\infty}(\cB_{\cO^{\infty}(Y)})}
   = C(\partial_{h}\cO^{\infty}(Y)) \ ,
\end{equation}     which
 sends a function $\phi$ in $C_{u,0}(Y)
 $ to the class of the function
   \begin{equation}\label{gerfewrfwerfwrfwre} \tilde \phi:\cO^{\infty}(Y)\to \C\ , \qquad  (t,y) \mapsto   \left\{\begin{array}{cc}
\phi(y)&t\ge 0 \\0 &t<0 \end{array} \right.
   \end{equation}  
   in $\ell_{\cB_{\cO^{\infty}(Y)}}^{\infty}(\cO^{\infty}(Y))$.
The decay of the variation of $\tilde \phi$ 
away from bounded subsets of the cone is implied by  the decay of $\phi$ outside of bounded subsets of $Y$
together with the uniform continuity of $\phi$ and the  characterization of the coarse entourages of the cone for large times, see \cref{oijgowpererfwerfwref} and  \cite[Lem. 5.7]{bel-paschke}. One thus checks that 
  \eqref{fqwoihdioewdewdqedq} is well-defined and 
  that $\pi:=(\pi_{Y})_{Y\in G\UBC}$ is a natural transformation.
  
 
%
%
  Precomposing   the topological transgression    from \eqref{fqwedwedwedqwe} with the functor  $\cO^{\infty}$
  we get the natural transformation
 \begin{equation}\label{frfwrwerfwerfwer}T^{G,\topp} \cO^{\infty} :E \bP\circ \cO^{\infty}\to \Sigma E\circ \partial_{h}\circ  \cO^{\infty}:G\UBC\to \cC\ .
\end{equation}   

   \begin{ddd} 
   We define the topological Paschke morphism as the composition
   \begin{equation}\label{hkeorptkgpertgertg}p^{G,\topp}:E \bP\circ \cO^{\infty}\xrightarrow{T^{G,\topp}  \cO^{\infty},\eqref{frfwrwerfwerfwer}}\Sigma E\circ \partial_{h}\circ  \cO^{\infty}
 \xrightarrow{\Sigma E \pi, \eqref{gwreggjweiog}}  \Sigma E\circ\iota^{\topp}_{u}:G\UBC\to \cC\ .\end{equation}
 \end{ddd}
   
%
%
%
%

   We now assume that    the motivic transgression  $T^{\mot}:\Yo(X) \to \Yo(\cO^{\infty}(Y))$ from \cref{kophrgertgertrgertgergert}  is defined. Recall $k$ from \eqref{ferfjwerfmerfireojoijo}.
  Since $E  \bP:G\BC\to \cC$ is an equivariant  coarse homology  theory by \cref{lpkherferferfergertgertgetrg}, it can be applied to the motivic transgression.
    \begin{theorem}\label{kophrthertgtrgetrge}
 The following square commutes:
 $$\xymatrix{E\bP(X)\ar[r]^{T^{\mot} }\ar[d]_{T^{G,\topp}}&E\bP(\cO^{\infty}(Y)\ar[d]^{p^{G,\topp}})\\\Sigma E(\partial_{h}X)\ar[r]^{k}&\Sigma E(\iota_{u}^{\topp}(Y))&}\ .$$
 \end{theorem}

\begin{proof}
We expand the square as follows:
$$\xymatrix{E\bP(X)\ar[r]\ar[d]^{T^{G,\topp}}&E\bP(\hat X')\ar[d]^{T^{G,\topp} }&\ar[d]^{T^{G,\topp} }\ar[l]^{\eqref{gwerfwerfw3fweferfwer},j}_{\simeq }E\bP(\cO^{\infty}(Y))\ar@/^2cm/[dd]^{p^{G,\topp}}\\ \ar[rrd]^{k}\Sigma  E(\partial_{h}X)\ar[r]^{\partial_{h}\psi}&\Sigma E(\partial_{h}\hat X')\ar[rd]^{\eqref{gwerfwerfwerf252}}&\Sigma E(\partial_{h}\cO^{\infty}(Y))\ar[l] \ar[d]^{\pi }\\&&\Sigma E(\iota_{u}^{\topp}(Y))}\ .$$
The   arrow labeled by \eqref{gwerfwerfwerf252}  is the Gelfand dual of
$$C_{u,0}(Y)\stackrel{}{\to}\ell^{\infty}_{\cB_{\cO^{\infty}(Y)}}(\cO^{\infty}(Y))\to \ell^{\infty}_{\cB_{\hat X'}} (\hat X')\to \frac{\ell^{\infty}_{\cB_{\hat X'}} (\hat X')}{\ell^{\infty}(\cB_{\hat X'})}\ ,$$
where the first map in the lower line is given by \eqref{gerfewrfwerfwrfwre} 
 and
the second map is extension by zero.
 The left lower triangle commutes by the definition \eqref{ferfjwerfmerfireojoijo} of $k$, and the commutativity of the right  lower triangle  can easily be checked on the level of  Gelfand dual algebras.
\end{proof}

 \subsection{Higson-dominated compactifications}

We consider  a bornological coarse space  $X$.
 Let $A$ be any closed unital subalgebra of $\ell^{\infty}_{\cB}(X)$
 containing $\ell^{\infty}(\cB)$  (see \eqref{gerwfwerfwerfwf} and \eqref{gerwfwerfwerfwf1}).
    By Gelfand duality it determines
 a compact Hausdorff space $\bar X$ such that $C(\bar X)=A$.
 We can consider $\bar X$ as a compactification of the discrete space $X$. 
 The Gelfand dual of the  right vertical map in the following map of exact sequences of commutative $C^{*}$-algebras $$
\xymatrix{
0\ar[r]&\ell^{\infty}(\cB)\ar@{=}[d]\ar[r]&A\ar[r]\ar[d]&  A/\ell^{\infty}(\cB)\ar[r]\ar[d]&0\\
0\ar[r]&\ell^{\infty}(\cB)\ar[r]&\ell^{\infty}_{\cB}(X)\ar[r]&C(\partial_{h}X)\ar[r]&0}
$$
 is a map of compact Hausdorff spaces   $\pi:\partial_{h}X\to Y$  where $Y$ is determined by $C(Y)=A/\ell^{\infty}(\cB)$.

 \begin{ddd}
 We call $\bar X$ a Higson-dominated compactification of $X$ with boundary $Y$.
  \end{ddd}
  
We can identify the set $X$ considered as a  discrete topological spaces  with the Gelfand dual of $\ell^{\infty}(\cB_{min})$, where $\cB_{min}$ is the minimal bornology on $X$ consisting of finite subsets.
We have a natural map $$\bar X \supseteq X\stackrel{\id_{X}}{\to} X$$   in $\LCH^{+}$ corresponding via Gelfand duality to the 
inclusion $\ell^{\infty}(\cB_{min})\to\ell^{\infty}(\cB)\to A$.
We let $p:X\to \bar X$  be the corresponding open  inclusion map.

   \begin{ass} \label{irhtgirtgrgerhertg}We assume the following data to be given:
 \begin{enumerate}
 \item
 a neighbourhood $U$ of $Y$ in $\bar X$ and a retraction $r:U\to Y$.
 \item 
  a proper, bornological  and controlled function $s:X\to [0,\infty)$.
  \item  a point $y_{0}$ in $Y$.
 \end{enumerate}
 \end{ass}



  We define a map of sets
  \begin{equation}\label{gwerpokpfwerfw}\psi :X\to [0,\infty)\times Y\ , \quad \psi(x):=\left\{\begin{array}{cc}(s(x),y_{0}) &x\not\in p^{-1}(U)\\ (s(x),r(p(x)))& x\in p^{-1}(U)  \end{array} \right.
\end{equation}   Since $Y$ is a compact Hausdorff space it has a canonical  uniform bornological coarse structure.
  By definition its bornological and  coarse structures  are the maximal ones.   The uniform structure consists of all neighbourhoods of the diagonal.
   Using this uniform bornological coarse structure  we equip $[0,\infty)\times Y$ with the bornological coarse structure of the  cone $\cO(Y)$, see \eqref{regwerfwerf234} and the text before that.
 
 \begin{lem}   The map  $\psi:X\to \cO(Y)$ in \eqref{gwerpokpfwerfw} is a morphism of bornological coarse spaces.\end{lem}
 \begin{proof}
 First of all, any bounded subset of $\cO(Y)$ is contained in  a set of the form $[0,R]\times Y$ for some $R$ in $[0,\infty)$. Then $\psi^{-1}([0,R]\times Y)\subseteq   s^{-1}([0,R])$ and is therefore bounded since $s$ is proper.
 This shows that $\psi$ is proper.
 
 Let  $W$  be any entourage of $X$.
 Since $s$ is controlled
 we have   \begin{equation}\label{gweroifjowierfwerfwerf}\sup_{(x,x')\in W} |s(x)-s(x')|=:C<\infty\ .
\end{equation}
This shows that the first component of $\psi$ is controlled.

The  subset   $X\setminus p^{-1} (U)$ is a bounded subset of $X$. Indeed, since $\bar X$ is compact Hausdorff and hence normal, 
 by Tietze's extension theorem  there exist a function $f$ in $C(\bar X)$ such that
$f_{| \bar X\setminus U}=1$ and $f_{|Y}=0$. We then have $f\in \ker( C(\bar X)\to C(Y))=   \ell^{\infty}(\cB)$.
This implies $\lim_{B\in \cB}\| f_{|X\setminus B}\|=0$.   In particular there exists $B$ in $\cB$ such that
$ X\setminus p^{-1}(U)\subseteq B$. 

Since $s$ is bornological the set
$s(X\setminus p^{-1}(U))$ is bounded in $[0,\infty)$.

Since $Y$ has the maximal coarse structure, 
it remains to show that for every  neighbourhood $V$ of the diagonal of $Y$ there exists
$R$ in $[0,\infty)$ such that $(x,x')\in W$ and 
$s(x)\ge R$ implies $(r(p(x)),r(p(x')))\in V$. Here we take $R$ so large that $
s(X\setminus p^{-1}(U))\subseteq [0,R-C)$.
Then $s(x')\ge R-C$ by \eqref{gweroifjowierfwerfwerf} and therefore 
 $x,x'\in   p^{-1}(U)$ and $r(p(x)), r(p(x'))$ in $Y$ are defined.

Assume the contrary. Then there exists a neighbourhood $V$ of the diagonal of $Y$ and a  sequence
$((x_{i},x'_{i}))_{i\in \nat}$ in $W$ with $s(x_{i})\to \infty$ and
$(r(p(x_{i})),r(p(x_{i}')))\not\in V$.
By compactness of $Y$ we can assume that $r'(p(x_{i}))\to z$ and $r'(p(x_{i}'))\to z'$.
Then $z'\not=z$.

Since $\bar X$ is normal as observed above,   Tietze's extension theorem gives a continuous function $f:\bar X\to [0,1]$ such that $f(z)=0$ and $f(z')=1$. Since $\bar X$ is a Higson-dominated compactification
 we have $f_{|X} \in \ell^{\infty}_{\cB}(X)$. Since $s$ is bornological   there exists
$S$ in $[0,\infty)$ such that $s(x)\ge S$ and $(x,x')\in W$ implies that  $|f(x)-f(x')|\le 1/2$.

In particular, $|f(x_{i})-f(x_{i}')|\le 1/2$ for all $i$ in $\nat $ such that $s(x_{i})\ge S$.  
Taking the limit  over $i$ in $\nat$ we conclude $|f(z)-f(z')|\le 1/2$. This a contradiction.
\end{proof}
 
\begin{kor}\label{hojpertrtgertgerth9}
If $X$ is a bornological coarse space and  $\bar X$ is  a Higson-dominated compactification of $X$ with boundary $Y$
with the data described in  \cref{irhtgirtgrgerhertg}, then the motivic transgression $T^{\mot}$ is defined (with $X'=\cO(Y)$ and the canonical map $\pi:\partial_{h}X\to Y$).

\end{kor}

Let $\bar X$ be a Higson-dominated compactification of $X$ with non-empty boundary $Y$.
  If the bornological coarse space $X$ is presented by a metric $d$.  then
 we can choose a base point $x_{0}$ and  set
 $$s(-):=d(x_{0},-):X\to [0,\infty)\ .$$
 This function is proper, bornological and controlled.

Assume that $\bar X$ is metrizable and that $Y$ is an absolute neighbourhood retract.
Then there exists a neighborhood $U$ of $Y$ in $\bar X$ and a retraction $r:U\to Y$.

\begin{kor}\label{hojpertrtgertgerth91}
If $X$ is a metrizable bornological coarse space and  $\bar X$ is  a  metrizable Higson-dominated compactification of $X$ with non-empty boundary $Y$ such that $Y$ is  an  absolute neighbourhood retract, then the motivic transgression $T^{\mot}$  is defined  (with $X'=\cO(Y)$ and the canonical map $\pi:\partial_{h}X\to Y$). 

\end{kor}

Note that there is no relation between the metrics of $X$ and $\bar X$ are required.

Note that finite CW-complexes are absolute neighbourhood retracts.
\begin{kor}\label{opkherptokgpertgrtge}
If $X$ is a metrizable bornological coarse space and  $\bar X$ is  a  metrizable Higson-dominated compactification of $X$  by a non-empty  finite CW-complex $Y$, then the motivic transgression $T^{\mot}$  is defined  (with $X'=\cO(Y)$ and the canonical map $\pi:\partial_{h}X\to Y$).

\end{kor}

\subsection{Extending commutativity using homotopy theory}\label{gwerjofewrfwer}

In this   section we explain how to extend morphisms between  natural transformations.
This is used in \cref{koipgwergwerfwerfrw}.

  Let $\cD,\cC$ be cocomplete $\infty$-categories,
 $E,F:\cD\to \cC$ be  colimit-preserving functors, and $\phi,\psi:E\to F$ be natural transformations.
 Let $f:\cG \to \cD$ be some functor and assume that \begin{equation}\label{friq9fqwedewdqewd}f^{*}\phi\simeq f^{*}\psi\ .
\end{equation} 
 We form the left-Kan extension   $$\xymatrix{\cG\ar[rr]^{f}\ar[dr]^{y}&&\cC\\&\PSh(\cG)\ar@{..>}[ur]^{\hat f}&}$$
 where $\hat f$ preserves colimits.
 By the universal property of presheaves we have an equivalence
 $$y^{*}:\Fun^{\colim}(\PSh(\cG),\cC)\stackrel{\simeq}{\to} \Fun(\cG,\cC)\ .$$
 The equivalence $$y^{*}\hat f^{*}\phi\simeq f^{*}\phi\simeq f^{*}\psi \simeq y^{*}\hat f^{*}\psi$$ therefore induces an equivalence
 $\hat f^{*}\phi\simeq \hat f^{*}\psi$.
 In particular, for every presheaf $Y$ in $\PSh(\cG)$ we have an equivalence
 $$\phi_{\hat f(Y)}\simeq \psi_{\hat f(Y)}:E(\hat f(Y))\to F(\hat f(Y))\ .$$
We conclude:
\begin{kor}\label{okgprherhreth}
For every object $D$ in the essential image of $\hat f$ we have an equivalence
$$\phi_{D}\simeq  \psi_{D}:E(D)\to F(D)$$ of morphisms in $\cC$.
\end{kor}
The equivalence \eqref{friq9fqwedewdqewd} of natural transformations  therefore extends to an objectwise 
 equivalence on the essential image of $\hat f$. Note that 
 this argument does not give an equivalence of natural transformations between the restrictions of 
 $\phi $ and $\psi$ to the essential image of $\hat f$. 

 \begin{ex}\label{kohpertgrtgertgertg}
 Assume that $\cC=\Spc$ or $\cC=\Sp$. The existence of the equivalence 
 $ \phi_{D}\simeq   \psi_{D}$ implies that
 $\phi_{D,*}=\psi_{D,*}:\pi_{*} E(D)\to \pi_{*}F(D)$.
 We therefore get an  equality of  natural transformations between the graded-group valued functors $\pi_{*}E$ and $\pi_{*}F$
 on the essential image of $\hat f$. \hB
 \end{ex}
 
 If one needs  some spectrum-valued naturality one can use the following.
Let $\cI$ be some $\infty$-category and consider the bold part of 
the diagram
$$\xymatrix{\cI\ar[rr]^{i}\ar@{..>}[ddrr]^{\hat i}&& \cD\ar@/^0.3cm/[rr]^{E} \ar@/^-0.3cm/[rr]_{F}&&\cC\\&&&\simeq&\\&&\PSh(\cG)\ar[uu]^{\hat f}\ar@/^0.3cm/[ruur]^{\hat f^{*}E}\ar@/^-0.3cm/[ruur]_{\hat f^{*}F}&&}\ .$$

\begin{kor}\label{guwergergefwwefwr}
When $i$ admits a lift $\hat i$, then we have an equivalence $i^{*}E\simeq i^{*}F$.
\end{kor}

  \subsection{Normal equivariant homology theories}\label{okhprertgertgrtger}

  In \cite{ass} we introduced the notion of a local homology theory.
 In the reference the  axioms  were chosen so that the coarsification of a local homology theory is
 a coarse homology theory and that the composition 
of the cone functor $\cO^{\infty}:G\UBC\to G\BC$ with a strong equivariant coarse homology theory is an equivariant  local homology theory, see also \cref{hgwiueghergwergwe}. These sidebars leave some freedom for the choice of the excision axiom.
In the present paper we fix the excision axiom such that pull-backs of weak equivariant Borel-Moore
homology theories via $\iota^{\topp}_{u}:G\UBC\to G\LCH^{+}$ are excisive.

   Let $X$ be in $G\UBC$ and $A$ be an invariant closed subset. 
 \begin{ddd}
 We call $A$ normal if the restriction $C_{0,u}(X)\to C_{0,u}(A)$ is surjective.
 \end{ddd}
We refer to \cref{kopwhwthrh} for the notation.

 We consider a functor  $E:G\UBC\to \cC$  to a cocomplete stable $\infty$-category.
 In analogy to  \cite[Def. 3.12]{ass} we adopt the following definition.
 
 \begin{ddd}\label{koprthrtgertg}  $E$ is called a  normal equivariant
 local homology theory if it
 \begin{enumerate}
\item is homotopy invariant, 
\item\label{grwergwefwerfwerfrrewfwr} is excisive for  equivariant coarsely and uniformly excisive decompositions into normal closed subspaces,
\item is $u$-continuous,
\item vanishes on flasques.
 \end{enumerate}
   \end{ddd}
   Here normal refers to the additional normality condition in \cref{grwergwefwerfwerfrrewfwr}.
 
   \begin{ex}
 If $X$ is a $G$-simplicial complex with the spherical path metric and $A$ is an invariant  subcomplex, 
 then $A$ is a normal subset.  This is all what is needed to conclude that
 the coarsification \cref{lpohertertgrtget} of a normal equivariant local homology theory is a coarse homology theory.
 \hB
 \end{ex}

  As in \cite[Sec. 4]{ass} we can construct
  the universal  normal equivariant local homology theory
  $$\Yo \cN:G\UBC\stackrel{\Sigma^{\infty}_{+}y}{\to} \Fun(G\UBC,\Sp)\stackrel{L}{\to} G\Sp\cN$$
  where $y$ is the Yoneda embedding and $L$ is a left Bousfield localization
  which precisely forces that $\Yo\cN$ satisfies the axioms listed in \cref{koprthrtgertg}.
  Note that $G\Sp\cN$ is a presentable stable $\infty$-category and for every cocomplete stable $\infty$-category
  $\cC$ we have an equivalence 
  $$\Yo\cN^{*}:\Fun^{\colim}(G\Sp\cN,\cC)\stackrel{\simeq}{\to} \Fun^{\mathrm{\nloc}}(G\UBC,\cC)$$
  where the superscript $\nloc$ stands for the full subcategory of $ \Fun(G\UBC,\cC)$ of    normal  equivariant  local homology  theories.
 
 \begin{ex}\label{ogpwergerfwerfwref}
 Here are our main examples of  normal equivariant local homology theories.
\begin{enumerate}
\item If $E:G\BC\to \cC$ is a strong  equivariant coarse homology theory, then the composition
$$E\cO^{\infty}:G\UBC\xrightarrow{\cO^{\infty}} G\BC\xrightarrow{E}\cC$$ is a normal  equivariant local homology theory (the proof is analogous to the non-equivariant case \cite[Lem. 9.6]{ass}).
\item  If $E:G\LCH^{+}\to \cC$ is an equivariant  weak Borel-Moore   homology theory, then the composition
$$E\iota^{\topp}_{u} :G\UBC\xrightarrow{ \iota^{\topp}_{u}} G\LCH^{+}\xrightarrow{E}\cC$$
is a normal equivariant  local homology theory. The functor is obviously $u$-continuous.
The functor
$\iota^{\topp}_{u}$ preserves homotopies, flasques and sends equivariant coarsely and uniformly excisive decompositions into normal closed subspaces to invariant closed decompositions. We then
use that $E$ is homotopy invariant, vanishes on flasques   in $G\LCH^{+}$, and is strongly excisive  in order to verify the conditions from \cref{koprthrtgertg}.
\end{enumerate}
\hB\end{ex}

 Evaluating the diagram \eqref{gwerwefwerfrefw}
at $X=\cO^{\infty}(Z)$ and  the canonical map $$\pi:\partial_{h}\cO^{\infty}(Z)\to \iota^{\topp}_{u}(Z)$$ from \eqref{ewrferfrfvsdfvfd}   for $Z$ in $G\UBC$
 we get the diagram 
\eqref{gwerwefwerfrefw11}.

{\tiny \begin{equation}\label{gwerwefwerfrefw11}\hspace{-2.3cm}
\xymatrix{ K\cX_{G_{can,max}}^{G,\ctr}(\cO^{\infty}(Z))\ar[rr]^{c^{G}}  \ar[dd]^{\rmch^{G,\alg}}&& K\cX_{G_{can,max}}^{G}(\cO^{\infty}(Z)) \ar[rr]^{T^{G,\an}}&&\ar[dl]\Sigma K^{G,\an}(\partial_{h}\cO^{\infty}(Z))\ar@{..>}[dddd]^{\beta}  \\   &  & & 
\Sigma K^{G,\an}(\iota^{\topp}_{u}(Z))\ar[ddd]^{\beta}& \\
  \PCH\cX^{G}_{G_{can,max}}(\cO^{\infty}(Z))\ar[d]^{\tau^{G}}&  & & &   \\ \ar[d]^{\beta} \PH\cX^{G}_{G_{can,max}}(\cO^{\infty}(Z),\C) &  & &  &  \\ 
\PCH^{hG}(\cO^{\infty}(Z),\C)\ar[ddd]^{P(\chi_{\BM} )^{hG}\circ \pr_{X}} & &&  \Sigma K^{\an,hG}(\iota^{\topp}_{u}(Z) )\ar[dd]^{ \ch^{hG}  } & \Sigma K^{\an,hG}(\partial_{h} \cO^{\infty}(Z))\ar@{..>}[ddd]^{ \ch^{hG}   }    \\
&  &&  &   \\ 
& && \Sigma  (\prod_{k\in \Z} \Sigma^{2k}H\C_{\BM})^{hG}(\iota^{\topp}_{u}(Z)) &\\ 
  (\prod_{k\in \Z} \Sigma^{2k}H\C_{\BM} \bP)^{hG} (\cO^{\infty}(Z))\ar[rrrr]^{(T^{\topp})^{hG}}\ &&&&\ar[ul]\Sigma  (\prod_{k\in \Z} \Sigma^{2k}H\C_{\BM})^{hG}(\partial_{h}\cO^{\infty}(Z))  }\ .
\end{equation}
}
It describes two natural transformations
\begin{equation}\label{hrtgrtgtrg36zzez}\phi,\psi:K\cX_{G_{can,max}}^{G,\ctr}\cO^{\infty} \to \Sigma  (\prod_{k\in \Z} \Sigma^{2k}H\C_{\BM})^{hG}(  \iota^{\topp}_{u}(-)):G\UBC\to \Sp\ .
\end{equation}
where $ \phi$ is the clockwise, and $\psi$ is the counterclockwise path.  Note that $$K\cX^{G,\ctr}_{G_{can,max}}:G\BC\to \Sp$$ is an equivariant strong coarse homology theory, and that 
$$\Sigma  (\prod_{k\in \Z} \Sigma^{2k}H\C_{\BM})^{hG}:G\LCH^{+}\to \Sp$$
is an equivariant  Borel-Moore homology theory.  By \cref{ogpwergerfwerfwref}  the
 domain and target of \eqref{hrtgrtgtrg36zzez} are normal equivariant local homology theories and  can be interpreted as colimit preserving functors
from $G\Sp\cN$ to $\Sp$.

We now consider the case of a trivial group $G$. Then the diagram \eqref{gwerwefwerfrefw11} specializes to \eqref{gwerwefwerfrefffw}.
In \cref{kophertgertgterg} we have chosen the normalization of $\ch^{h}$ such that $\phi_{*}\simeq \psi_{*}$.
We apply the general theory  from \cref{gwerjofewrfwer}   to $\cG=\{*\}$. There is a unique colimit-preserving functor
\begin{equation}\label{bwepo0ivergervdfvsfdvfdv}\hat *:\Spc\to \Sp\cN
\end{equation} sending $*$ to $\Yo\cN(*)$.

\begin{kor}
We have $\phi_{D}\simeq \psi_{D}$ for every $D$ in the essential image of $\hat *$.
\end{kor}

Note that a finite CW-complex is a compact Hausdorff space and therefore canonically an object of $\UBC$.
We have an inclusion $\CW^{\fin}\to \UBC$.
\begin{lem}\label{koprthrthergrtge}
We have a lift
$$\xymatrix{&\Spc\ar[dr]^{\hat *}&\\\CW^{\fin}\ar@{..>}[ur]\ar[r]&\UBC\ar[r]^{\Yo\cN}&\Sp\cN}$$
\end{lem}
\begin{proof}
 We have 
 a factorization
$$\xymatrix{ \UBC\ar[rr]^{\Yo\cN}&&\Sp\cN\\\CW^{\fin}\ar[u]\ar[d]^{!}&&\Spc\ar[u]^{\hat *}_{\eqref{bwepo0ivergervdfvsfdvfdv}}\\\ar@{..>}[uurr]\CW^{\fin}[W^{-1}]\ar[rr]^{\simeq}&&\Spc^{\fin}\ar[u]}\ ,$$
where the dotted arrow exists by the universal property of the marked arrow being the Dwyer-Kan localization at the homotopy equivalences since the composition 
 $$\CW^{\fin}\to \UBC\stackrel{\Yo\cN}{\to }\Sp\cN $$
is homotopy invariant. Since the latter is also 
excisive for cell-attachements the dotted arrow preserves finite colimits.
This implies that the lower triangle also commutes.
The assertion of the lemma follows.
\end{proof}

By combining  \cref{guwergergefwwefwr} and \cref{koprthrthergrtge}  we can conclude:

\begin{kor} \label{koprthrtwerferfwerfwhergrtge}The restriction of the diagram \eqref{gwerwefwerfrefffw} to $\CW^{\fin}$ commutes, i.e., we
 have $\phi_{|\CW^{\fin}}\simeq \psi_{|\CW^{\fin}}$.
\end{kor}

\subsection{Pairing with cohomology classes}\label{ergerwfwerfwrf}

In this subsection we evaluate the right corners of \eqref{ferwferfwrefw} on appropriate cohomology classes and discuss the consequences of the  commutativity of this square.

Let $R$ be a commutative ring spectrum and $M$ be in $\Mod(R)$. Then according to \cref{giuowerpgwrefefwerfwvdf} we can form the Borel-Moore cohomology theory 
$$R^{\BM}:L\LCH^{+,\op}\to \Mod(R)\ .$$ The dual Borel-Moore homology theory   \cref{okgpgkerpwofwerferwfwrefwf} can be written, using $R^{\BM}\simeq S^{\BM}\wedge R$, as $$M_{\BM}:=\Sp(S^{\BM},M)\simeq \Mod(R)(R^{\BM},M):\LCH^{+}\to \Mod(R)\ .$$ 
Hence for $Y$ in $\LCH^{+}$ we get a binatural evaluation pairing
\begin{equation}\label{vboijowrferfvewrf} \ev:M_{\BM}(Y)\otimes_{R}R^{\BM}(Y)\to M\ .
\end{equation}

We now consider $K$-theory with values in $\Mod(KU)$. By \cref{ijogpwregreg9} and \cref{hertgertgergtrgrtg} we have the equivalences
$$K(C_{0}(-))\simeq KU^{\BM}\ , \quad  \mbox{and} \quad K^{\an}\simeq KU_{\BM}\ ,$$
where $K:\nCalg\to \Mod(KU)$ is the topological $K$-theory functor for $C^{*}$-algebras, written as $K\circ e$ in  \cref{ijogpwregreg9}.    Hence by specializing \eqref{vboijowrferfvewrf} we have a bi-natural pairing
\begin{equation}\label{feroifjowerfwerfew}K^{\an}(Y)\otimes_{KU} K(C_{0}(Y))\to KU\ .
\end{equation}
For $X$ in $\BC$  we have the coarse corona pairing
$$-\cap^{\cX}-:K\cX(X)\otimes_{KU} K(C_{0}(\partial_{h}X))\to K\cX(\cB) $$
from \cite[Def. 3.35]{Bunke:2024aa}.
The restrictions of the  canonical map  $X\to *$ of sets  to bounded subsets of $X$ are 
 morphisms in $\BC$ and hence induce a morphism $K\cX(\cB)\to K\cX(*)\simeq KU$.
\begin{lem}\label{koprherhgrtgertg}
The following square commutes:
$$\xymatrix{K\cX(X)\otimes_{KU} K(C_{0}(\partial_{h}X))\ar[r]^{T^{\an}\otimes \id}\ar[d]^{-\cap^{\cX}-} & \Sigma K^{\an}(\partial_{h}X)\otimes K(C_{0}(\partial_{h}X))  \ar[d]^{\eqref{feroifjowerfwerfew}} \\ \Sigma  K\cX(\cB)\ar[r] &\Sigma KU } $$\end{lem}
\begin{proof}By \cite[Def. 3.35]{Bunke:2024aa} and using the pairing $\nu$
from \cite[Def. 3.34]{Bunke:2024aa}
the down-right composition    is given by  \begin{eqnarray}
K\cX(X)\otimes_{KU} K(C_{0}(\partial_{h}X))&\simeq& \EE(\C,\bV(X))\otimes_{KU} \EE(\C,C(\partial_{h}X))\nonumber\\&\stackrel{(1)}{\to}&
 \EE(\C,\bV(X))\otimes C(\partial_{h}X))\nonumber\\&\stackrel{\nu}{\to}&
 \EE(\C, \bV(X)/\bV(\cB\subseteq X))\nonumber\\&\stackrel{\partial}{\to}&
 \Sigma \EE(\C,\bV(\cB\subseteq X))\nonumber\\&\to& \Sigma KU\label{iguwoegwrefeferfw}\ ,
\end{eqnarray}
where $\nu$ is described in  \cite[Lem. 3.34]{Bunke:2024aa}.
The right-down composition is, according to \cref{rguweruigowerferfrfwr}, given by
 \begin{eqnarray*}
K\cX(X)\otimes_{KU} K(C_{0}(\partial_{h}X))&\simeq& \EE(\C,\bV(X))\otimes_{KU} \EE(\C,C(\partial_{h}X))\\&\stackrel{\delta_{X}}{\to}&
 \EE(C(\partial_{h}X),\bV(X)\otimes C(\partial_{h}X))\otimes \EE(\C,C(\partial_{h}X))\\&\stackrel{\mu}{\to}& \EE(C(\partial_{h}X), \bQ)\otimes_{KU} \EE(\C,C(\partial_{h}X))\\&\stackrel{\partial}{\to}&\Sigma \EE(C(\partial_{h}X), \C)\otimes_{KU} \EE(\C,C(\partial_{h}X))\\&\stackrel{\eqref{feroifjowerfwerfew}}{\to}&
\Sigma \EE(\C,\C)\\&\simeq&\Sigma KU
\end{eqnarray*}
Here $\bQ$ is the Calkin category $\Hilb(\C)/\Hilb_{c}(\C)$.
These maps are equivalent. In order to see this
we observe that if we compose $\nu$ with $\bV(X)/\bV(\cB\subseteq  X)\to \bQ$, then we get the pairing $\mu$.
Furthermore, the composition of the  pairing \eqref{feroifjowerfwerfew} with the diagonal morphism
can be reduced to the application of the symmetric monoidal structure as in step marked by $(1)$ in \eqref{iguwoegwrefeferfw}.
\end{proof}

By \eqref{gwweroijkewrjgfowertw334w} the  target of the  homotopical  Chern character \eqref{9gwerfreferwfwever} can be rewritten as  follows:
$$ \prod_{k\in \Z} \Sigma^{2k}H\C_{\BM}\simeq (KU\wedge H\C)_{\BM}\ .$$
We therefore have a bi-natural pairing
\begin{equation}\label{gwpoejropfwoer0f9wefref} \hspace{-0.5cm}( \prod_{k\in \Z} \Sigma^{2k}H\C_{\BM})(Y) \otimes  (KU\wedge H\C)^{\BM}(Y)      \simeq    (KU\wedge H\C)_{\BM}(Y)\otimes_{KU\wedge H\C} (KU\wedge H\C)^{\BM}(Y) \to KU\wedge H\C\ .
\end{equation} 
The morphism of ring spectra $\epsilon:KU\to KU\wedge H\C$ induced by the unit of $H\C$
induces a cohomological topological Chern character
$$\rmch_{h}:K(C_{0}(-))\simeq KU^{\BM}\to (KU\wedge H\C)^{\BM}\ .$$
The following is a formal consequence of the definitions:
\begin{lem}
The following  square commutes
$$\xymatrix{K^{\an}(Y)\otimes_{KU} K(C_{0}(Y))\ar[rr]^-{\rmch^{h}\otimes \rmch_{h}}\ar[d]^{\eqref{feroifjowerfwerfew}} && (\prod_{k\in \Z} \Sigma^{2k}H\C_{\BM})(Y)\otimes (KU\wedge H\C)^{\BM}(Y) \ar[d]^{\eqref{vboijowrferfvewrf}} \\  KU\ar[rr]^{\epsilon}&& KU\wedge H\C} \ .$$
\end{lem}

Assume that $X$ in $\BC$, $Y$ in $\UBC$ and $\pi:\partial_{h}X\to \iota^{\topp}_{u}(Y)$
satisfy the assumptions of  \cref{iowerjgowegjweriogrwegwe}.
 The following is a formal consequence of \cref{iowerjgowegjweriogrwegwe} and \cref{koprherhgrtgertg} for the lower right triangle.
\begin{kor}The following diagram commutes:
$${\tiny \xymatrix{K\cX^{\ctr}(X)\otimes K(C_{0}( \iota^{\topp}_{u}(Y))\ar[d]^{P(\chi_{\BM})\tau\rmch^{\alg}\otimes  \rmch_{h}}\ar[r]^{c\otimes \id }&K\cX(X)\otimes K(C_{0}( \iota^{\topp}_{u}(Y)))\ar[d]^{\id\otimes \pi^{*}} \ar[r]^{T^{\an}\otimes \id}& \ar[d]^{\id\otimes \pi^{*}}\Sigma K^{\an}(\partial_{h}X)\otimes K(C_{0}( \iota^{\topp}_{u}(Y))) \\ (\prod_{k\in \Z} \Sigma^{2k}H\C_{\BM})\bP(X)\otimes (KU\wedge H\C)^{\BM}( \iota^{\topp}_{u}(Y)) \ar[d]^{T^{\topp}\otimes \id}&\ar[dr]^{-\cap^{\cX}-}K\cX(X)\otimes K(C_{0}(\partial_{h}X )) \ar[r]^{T^{\an}\otimes \id}& \Sigma K^{\an}(\partial_{h}X)\otimes K(C_{0}( \partial_{h}X) )\ar[d]^{\eqref{feroifjowerfwerfew}}\\ \Sigma (\prod_{k\in \Z} \Sigma^{2k}H\C_{\BM})(\partial_{h}X)\otimes (KU\wedge H\C)^{\BM}( \iota^{\topp}_{u}(Y))\ar[r]^-{\eqref{gwpoejropfwoer0f9wefref}} &\Sigma KU\wedge H\C&\ar[l]_{\epsilon} \Sigma KU}}\ .$$
\end{kor}

 \begin{kor}\label{koghperthetrgertge}
 If $x$ is in $\pi_{i}K\cX^{\ctr}(X)$ and $\xi$ is in $\pi_{-j} K(C_{0}(\iota^{\topp}_{u}(Y)))$, then
 we have
\begin{equation}\label{hrtepogkeportgertgetrge}\epsilon(c(x)\cap^{\cX} \pi^{*}\xi)=\langle T^{\topp}P(\chi_{\BM})\tau\rmch^{\alg},\rmch_{h}(\xi)\rangle
\end{equation} 
 in $\pi_{i-j+1}(KU)\otimes \C$.
 \end{kor}
   
\begin{rem}  The equality \eqref{hrtepogkeportgertgetrge}
is in particular an integrality statement for its right-hand side. 
The equality \eqref{hrtepogkeportgertgetrge} is our version of  \cite[(1.1)]{Engel:2025aa}.
In  \cite{Engel:2025aa} the cycle $x$ is given  in terms of operators on an $X$-controlled
Hilbert space and $\rmch_{h}(\xi)$ is described in terms of a \v{C}ech-type decomposition of $Y$.
In this case the  right-hand side of \eqref{hrtepogkeportgertgetrge} is   made explicit
 and expressed in terms of an integral over $X$ of a term involving the integral kernels  of the operators determining    $x$ and functions describing  the decomposition of $Y$. 
It was one of the main goals of   \cite{Engel:2025aa} to verify 
the integrality of this integral (which in the situation of the reference  is apriori given)  by interpreting  it as the  right-hand side of \eqref{hrtepogkeportgertgetrge}
and comparing it with the left-hand side. \hB
  \end{rem}

%
%


\subsection{Two Chern characters}

 The construction of Chern character maps out of analytic $K$-homology
 has a long history and motivated the introduction of cyclic homology  \cite{Connes_1985}.
 Using the universal property of the equivariant $KK$-theory functor $\kk^{G}:G\nCalg\to \KK^{G}$, constructions of Chern characters to versions of bivariant cyclic
 homology have been considered in \cite{Cuntz_1997}, \cite{Nistor_1993}, \cite{Puschnigg_2003}, \cite{Voigt_2007}.
 In the present paper we added a version   $\rmch^{h}$  described  in \cref{9gwerfreferwfwever}  which does not involve cyclic homology at all. In this section we compare it with a more algebraic version involving cyclic homology.

 The inner square in \eqref{gwerwefwerfrefw1} is a natural (we do not say natural commutative) square
\begin{equation}\label{vewfcwecvfvsfdv}\xymatrix{K\cX^{G,\ctr}_{G_{can,max}}\cO^{\infty}\ar[r]^{p^{G,\an}c^{G}}\ar[d]_{ P(\chi_{BM})^{hG}\pr \beta \tau^{G}\rmch^{G,\alg}}  &\Sigma K^{G,\an} \iota^{\topp}_{u}\ar[d]^{\ch^{hG}\beta} \\ (\prod_{k\in \Z} \Sigma^{2k}H\bC_{BM}\bP)^{hG}\cO^{\infty}\ar[r]^{p^{G,\topp}} &\Sigma (\prod_{k\in \Z}\Sigma^{2k}H\C_{BM})^{hG}\iota^{\topp}_{u} } 
\end{equation}
of $\Sp$-valued functors defined on $G\UBC$.
Since the functors on the right-hand side are locally finite we can  factorize the horizontal maps as follows: 
\begin{equation}\label{}\xymatrix{ \ar[d]_{ P(\chi_{BM})^{hG}\pr \beta \tau^{G}\rmch^{G,\alg}}   K\cX^{G,\ctr}_{G_{can,max}}\cO^{\infty}\ar[r] &\ar[d]_{b:= (P(\chi_{BM})^{hG}\pr \beta \tau^{G}\rmch^{G,\alg})^{\lf}}   (K\cX_{G_{can,max}}^{G,\ctr}\cO^{\infty})^{\lf}\ar[r]^{a:=p^{G,\an,\lf}c^{G,\lf}}  &\Sigma K^{G,\an} \iota^{\topp}_{u}\ar[d]^{\rmch^{hG}\beta} \\  (\prod_{k\in \Z} \Sigma^{2k}H\bC_{BM}\bP)^{hG}\cO^{\infty}\ar[r]& ((\prod_{k\in \Z} \Sigma^{2k}H\bC_{BM}\bP)^{hG}\cO^{\infty})^{\lf}\ar[r]^{c}  &\Sigma (\prod_{k\in \Z}\Sigma^{2k}H\C_{BM})^{hG}\iota^{\topp}_{u} } 
\end{equation}
If $Y$ in $G\UBC$ is such that the component $a_{Y}$ is an equivalence, then the composition $$\rmch^{G,\prime}_{Y}:=c_{Y}b_{Y}a_{Y}^{-1}:\Sigma K^{G,\an}(\iota^{\topp}_{u}(Y))\to \Sigma (\prod_{k\in \Z}\Sigma^{2k}H\C_{BM})^{hG}(\iota^{\topp}_{u} (Y))$$ 
is an alternative to the Borel-equivariant Chern character $ \rmch^{hG}\beta$.

By \cite[Thm. 1.5]{bel-paschke} we know that 
  $p^{G,\an}_{Y}$ and $c^{G}$ are equivalences if 
$Y$ in $G\UBC$ is homotopy equivalent to a $G$-finite $G$-simplicial complex with finite stabilizers.
This implies that $p^{G,\an,\lf}_{Y}$ and $c^{G,\lf}$ are equivalences if $Y$ is homotopy equivalent to a locally finite 
$G$-simplicial complex with proper $G$-action.

\begin{kor}
$$\rmch^{G,\prime}_{Y}:\Sigma K^{G,\an}(\iota^{\topp}_{u}(Y))\to \Sigma (\prod_{k\in \Z}\Sigma^{2k}H\C_{BM})^{hG}(\iota^{\topp}_{u} (Y))$$ is defined provided $Y$ in $G\UBC$ is  homotopy equivalent to a locally finite 
$G$-simplicial complex with proper $G$-action.
\end{kor}

Recall that $\rmch^{hG}\beta $ is defined by homotopy theory, using the identification of  analytic $K$-homology with the Borel-Moore homology generated by $KU$ and the identification $KU\wedge H\C\simeq \prod_{k\in \Z}\Sigma^{2k}H\C$. In contrast, at least philosopically, $\rmch^{G,\prime}$ applies the algebraic Chern character $\rmch^{G,\alg}$ to a 
suitable algebraic  cycle like a summable Fredholm module. There is no obvious reason why these
constructions should give the same results.

\begin{prob}
Do we have an equivalence $\rmch^{G,\prime}_{Y}\simeq \rmch^{hG}_{Y}\beta_{Y}$ for $Y$ in $G\UBC$ which is homotopy equivalent to locally finite $G$-simplicial complex with finite stabilizers.
\end{prob}

The case of a finite group follows from \cref{iowerjgowegjweriogrwegwe} applied to $X=\cO^{\infty}(Y)$ and $\pi$ the component of \eqref{ewrferfrfvsdfvfd} at $Y$.
\begin{kor}
If $G$ is finite, then $\rmch^{G,\prime}_{Y}\simeq \rmch^{hG}_{Y}\beta_{Y}$ for every $Y$ in $G\UBC$ which is homotopy equivalent to a locally finite $G$-simplicial complex.
\end{kor}

   \bibliographystyle{alpha}
\bibliography{forschung2021}

\begin{thebibliography}{BEKW20d}

\bibitem[BC20]{zbMATH07160436}
U.~Bunke and D.-C. Cisinski.
\newblock A universal coarse {{\(K\)}}-theory.
\newblock {\em New York J. Math.}, 26:1--27, 2020.

\bibitem[BD24]{budu}
U.~Bunke and B.~Duenzinger.
\newblock ${E}$-theory is compactly assembled.
\newblock
  \href{https://arxiv.org/pdf/2402.18228.pdf}{https://arxiv.org/pdf/2402.18228.pdf},
  02 2024.

\bibitem[BE]{cank}
U.~Bunke and A.~Engel.
\newblock {Additive $C^{*}$-categories and $K$-theory}.
\newblock \href{https://arxiv.org/abs/2010.14830}{arXiv:2010.14830}.

\bibitem[BE20a]{ass}
U.~Bunke and A.~Engel.
\newblock {Coarse assembly maps}.
\newblock {\em J.\ Noncommut.\ Geom.}, 14(4):1245--1303, 2020.

\bibitem[BE20b]{buen}
U.~Bunke and A.~Engel.
\newblock {\em Homotopy theory with bornological coarse spaces}, volume 2269 of
  {\em Lecture Notes in Math.}
\newblock Springer, 2020.
\newblock \href{https://arxiv.org/abs/1607.03657}{arXiv:1607.03657}.

\bibitem[BE23]{coarsek}
U.~Bunke and A.~Engel.
\newblock Topological equivariant coarse {K}-homology.
\newblock {\em Annals of K-Theory}, 8(2):141--220, June 2023.

\bibitem[BE25]{Bunke:2017aa}
U.~Bunke and A.~Engel.
\newblock The coarse index class with support.
\newblock {\em Tunisian Journal of Mathematics}, 7(2):339--378, May 2025.

\bibitem[BEKW20a]{Bunke:ab}
U.~Bunke, A.~Engel, D.~Kasprowski, and C.~Winges.
\newblock Homotopy theory with marked additive categories.
\newblock {\em Theory Appl. Categ.}, 35:371--416, 2020.

\bibitem[BEKW20b]{Bunke_20202}
U.~Bunke, A.~Engel, D.~Kasprowski, and C.~Winges.
\newblock Injectivity results for coarse homology theories.
\newblock {\em Proceedings of the London Mathematical Society},
  121(6):1619--1684, August 2020.

\bibitem[BEKW20c]{trans}
U.~Bunke, A.~Engel, D.~Kasprowski, and C.~Winges.
\newblock Transfers in coarse homology.
\newblock {\em M{\"u}nster J. Math.}, 13(2):353--424, 2020.

\bibitem[BEKW20d]{equicoarse}
U.~Bunke, A.~Engel, D.~Kasprowski, and Ch. Winges.
\newblock {Equivariant coarse homotopy theory and coarse algebraic
  $K$-homology}.
\newblock In {\em {$K$-Theory in Algebra, Analysis and Topology}}, volume 749
  of {\em Contemp.\ Math.}, pages 13--104, 2020.

\bibitem[BELa]{bel-paschke}
U.~Bunke, A.~Engel, and M.~Land.
\newblock {Paschke duality and assembly maps}.
\newblock \href{http://arxiv.org/abs/2107.02843}{arxiv:2107.02843}.

\bibitem[BELb]{KKG}
U.~Bunke, A.~Engel, and M.~Land.
\newblock A stable $\infty$-category for equivariant $\mathrm{K\!K}$-theory.
\newblock \href{https://arxiv.org/pdf/2102.13372.pdf}{arxiv:2102.13372}.

\bibitem[BKW21]{Bunke:2021ab}
U.~Bunke, D.~Kasprowski, and C.~Winges.
\newblock On the {F}arrell-{J}ones conjecture for localising invariants.
\newblock
  \href{https://arxiv.org/pdf/2111.02490.pdf}{https://arxiv.org/pdf/2111.02490.pdf},
  11 2021.

\bibitem[BL24]{Bunke:2024aa}
U.~Bunke and M.~Ludewig.
\newblock Coronas and {C}allias type operators in coarse geometry.
\newblock
  \href{https://arxiv.org/pdf/2411.01646.pdf}{https://arxiv.org/pdf/2411.01646.pdf},
  11 2024.

\bibitem[Bun19]{startcats}
U.~Bunke.
\newblock {Homotopy theory with *-categories}.
\newblock {\em Theory Appl.\ Categ.}, 34(27):781--853, 2019.

\bibitem[Bun24]{crosscat}
U.~Bunke.
\newblock Non-unital {{\(C^*\)}}-categories, (co)limits, crossed products and
  exactness.
\newblock {\em High. Struct.}, 8(2):163--209, 2024.

\bibitem[Bun25]{Bunke:2025aa}
U.~Bunke.
\newblock Branched coarse coverings and transfer maps.
\newblock
  \href{https://arxiv.org/pdf/2502.01497.pdf}{https://arxiv.org/pdf/2502.01497.pdf},
  02 2025.

\bibitem[Cap19]{caputi-diss}
L.~Caputi.
\newblock {\em Hochschild and cyclic homology for bornological coarse spaces}.
\newblock PhD thesis, Universit{\"a}t Regensburh, 2019.

\bibitem[Cap20]{Caputi_2020}
L.~Caputi.
\newblock Cyclic homology for bornological coarse spaces.
\newblock {\em Journal of Homotopy and Related Structures}, 15(3--4):463--493,
  July 2020.

\bibitem[CH90]{zbMATH04182148}
A.~Connes and N.~Higson.
\newblock Deformations, asymptotic morphisms and bivariant {{\(K\)}}-theory.
\newblock {\em C. R. Acad. Sci., Paris, S{\'e}r. I}, 311(2):101--106, 1990.

\bibitem[CM90]{Connes_1990}
A.~Connes and H.~Moscovici.
\newblock Cyclic cohomology, the novikov conjecture and hyperbolic groups.
\newblock {\em Topology}, 29(3):345--388, 1990.

\bibitem[CM95]{Connes_1995}
A.~Connes and H.~Moscovici.
\newblock The local index formula in noncommutative geometry.
\newblock {\em Geometric and Functional Analysis}, 5(2):174--243, March 1995.

\bibitem[Con85]{Connes_1985}
A.~Connes.
\newblock Non-commutative differential geometry.
\newblock {\em Publications math{\'e}matiques de l'IH{\'E}S}, 62(1):41--144,
  December 1985.

\bibitem[CT08]{Corti_as_2008}
W.~Corti{\~n}as and A.~Thom.
\newblock Comparison between algebraic and topological {K}-theory of locally
  convex algebras.
\newblock {\em Advances in Mathematics}, 218(1):266--307, May 2008.

\bibitem[Cun97]{Cuntz_1997}
J.~Cuntz.
\newblock Bivariant $k$-theory for locally convex algebras and the chern-connes
  character.
\newblock {\em Documenta Mathematica}, 2:139--182, 1997.

\bibitem[EL25]{Engel:2025aa}
A.~Engel and M.~Ludewig.
\newblock Transgressing the algebraic coarse character map.
\newblock
  \href{https://arxiv.org/pdf/2507.10816.pdf}{https://arxiv.org/pdf/2507.10816.pdf},
  07 2025.

\bibitem[Elm83]{elmendorf}
A.~D. Elmendorf.
\newblock {Systems of Fixed Point Sets}.
\newblock {\em Transactions Amer.\ Math.\ Soc.}, 277(1):275--284, 1983.

\bibitem[GHT00]{Guentner_2000}
E.~Guentner, N.~Higson, and J.~Trout.
\newblock Equivariant ${E}$-theory for ${C}^{*}$-algebras.
\newblock {\em Memoirs of the American Mathematical Society}, 148(703):0--0,
  2000.

\bibitem[Goo85]{Goodwillie_1985}
Th.~G. Goodwillie.
\newblock Cyclic homology, derivations, and the free loopspace.
\newblock {\em Topology}, 24(2):187--215, 1985.

\bibitem[Goo86]{Goodwillie_1986}
Th.~G. Goodwillie.
\newblock Relative algebraic k-theory and cyclic homology.
\newblock {\em The Annals of Mathematics}, 124(2):347, September 1986.

\bibitem[Hei21]{heissdiss}
D.~Hei{\ss}.
\newblock Equivariant homotopy theory in the presence of bornologies.
\newblock PhD thesis Uni-Regensburg,
  \href{https://epub.uni-regensburg.de/46332/}{https://epub.uni-regensburg.de/46332/},
  Juli 2021.

\bibitem[Hig90]{MR1068250}
N.~Higson.
\newblock Categories of fractions and excision in {$KK$}-theory.
\newblock {\em J. Pure Appl. Algebra}, 65(2):119--138, 1990.

\bibitem[Joa03]{joachimcat}
M.~Joachim.
\newblock {$K$-homology of $C^{\ast}$-categories and symmetric spectra
  representing $K$-homology}.
\newblock {\em Math. Ann.}, 327:641--670, 2003.

\bibitem[Kar64]{kkka}
M.~Karoubi.
\newblock Les isomorphismes de {C}hern et de {T}hom-{G}ysin an {K}-th'eorie.
\newblock {\em S'eminair Henri Cartan 1963-64}, 16:1--16, 1963-64.

\bibitem[Kas87]{Kassel_1987}
Ch. Kassel.
\newblock Cyclic homology, comodules, and mixed complexes.
\newblock {\em Journal of Algebra}, 107(1):195--216, April 1987.

\bibitem[Kas88]{kasparovinvent}
G.~G. Kasparov.
\newblock {Equivariant $K\!K$-theory and the Novikov conjecture}.
\newblock {\em Invent.\ Math.}, 91(1):147--201, 1988.

\bibitem[Kel99]{Keller_1999}
B.~Keller.
\newblock On the cyclic homology of exact categories.
\newblock {\em Journal of Pure and Applied Algebra}, 136(1):1--56, March 1999.

\bibitem[KNP]{NKP}
A.~Krause, Th. Nikolaus, and P.~P{\"u}tzst{\"u}ck.
\newblock Scheaves on manifolds.
\newblock
  \href{https://www.uni-muenster.de/IVV5WS/WebHop/user/nikolaus/Papers/sheaves-on-manifolds.pdf}{https://www.uni-muenster.de/IVV5WS/WebHop/user/nikolaus/Papers/sheaves-on-manifolds.pdf}.

\bibitem[Lod98]{zbMATH01093754}
J.-L. Loday.
\newblock {\em Cyclic homology.}, volume 301 of {\em Grundlehren Math. Wiss.}
\newblock Berlin: Springer, 2nd ed. edition, 1998.

\bibitem[LR06]{L_CK_2006}
W.~L{\"u}ck and H.~Reich.
\newblock Detecting ${K}$-theory by cyclic homology.
\newblock {\em Proceedings of the London Mathematical Society}, 93(3):593--634,
  October 2006.

\bibitem[LS16]{L_ck_2016}
W.~L{\"u}ck and W.~Steimle.
\newblock A twisted {B}ass{\textendash}{H}eller{\textendash}{S}wan
  decomposition for the algebraic {K}-theory of additive categories.
\newblock {\em Forum Mathematicum}, 28(1), jan 2016.

\bibitem[LT25]{Ludewig:2025aa}
M.~Ludewig and G.~Ch. Thiang.
\newblock Large-scale quantization of trace i: Finite propagation operators.
\newblock
  \href{https://arxiv.org/pdf/2506.10957.pdf}{https://arxiv.org/pdf/2506.10957.pdf},
  06 2025.

\bibitem[Nis93]{Nistor_1993}
V.~Nistor.
\newblock A bivariant {C}hern--{C}onnes character.
\newblock {\em The Annals of Mathematics}, 138(3):555, November 1993.

\bibitem[Pus03]{Puschnigg_2003}
M.~Puschnigg.
\newblock Diffeotopy functors of ind-algebras and local cyclic cohomology.
\newblock {\em Documenta Mathematica}, 8:143--245, 2003.

\bibitem[QR10]{quro}
Y.~Quiao and J.~Roe.
\newblock On the localization algebra of {G}uoliang {Y}u.
\newblock {\em Forum Mathematicum}, 22:657--665, 2010.

\bibitem[Roe93]{MR1147350}
J.~Roe.
\newblock Coarse cohomology and index theory on complete {R}iemannian
  manifolds.
\newblock {\em Mem. Amer. Math. Soc.}, 104(497):x+90, 1993.

\bibitem[Roe03]{roe_lectures_coarse_geometry}
J.~Roe.
\newblock {\em {Lectures on Coarse Geometry}}, volume~31 of {\em University
  Lecture Series}.
\newblock American Mathematical Society, 2003.

\bibitem[Sch04]{Schlichting_2004}
M.~Schlichting.
\newblock Delooping the {K}-theory of exact categories.
\newblock {\em Topology}, 43(5):1089--1103, sep 2004.

\bibitem[Tab11]{Tabuada_2011}
G.~Tabuada.
\newblock A universal characterization of the {C}hern character maps.
\newblock {\em Proceedings of the American Mathematical Society},
  139(04):1263--1263, April 2011.

\bibitem[Voi07]{Voigt_2007}
Ch. Voigt.
\newblock Equivariant local cyclic homology and the equivariant chern-connes
  character.
\newblock {\em Documenta Mathematica}, 12:313--359, 2007.

\bibitem[Wei89]{zbMATH04095731}
Ch.~A. Weibel.
\newblock Homotopy algebraic {K}-theory.
\newblock Algebraic {{\(K\)}}-theory and algebraic number theory, {Proc}.
  {Semin}., {Honolulu}/{Hawaii} 1987, {Contemp}. {Math}. 83, 461-488 (1989).,
  1989.

\bibitem[Wul22]{wulff_axioms}
Ch. Wulff.
\newblock Equivariant coarse (co-)homology theories.
\newblock {\em SIGMA, Symmetry Integrability Geom. Methods Appl.}, 18:paper
  057, 62, 2022.

\bibitem[WW95]{zbMATH01452550}
M.~Weiss and B.~Williams.
\newblock Assembly.
\newblock In {\em Novikov conjectures, index theorems and rigidity. Vol. 2.
  Based on a conference of the Mathematisches Forschungsinstitut Oberwolfach in
  September 1993}, pages 332--352. Cambridge: Cambridge University Press, 1995.

\bibitem[Yu95]{Yu_1995}
G.L. Yu.
\newblock Cyclic cohomology and higher indexes for noncompact complete
  manifolds.
\newblock {\em Journal of Functional Analysis}, 133(2):442--473, November 1995.

\end{thebibliography}

\end{document}